\documentclass[10pt]{amsart}

\usepackage{xcolor}
\definecolor{winered}{rgb}{0.8,0,0}
\definecolor{deepblue}{rgb}{0,0,0.8}
\usepackage[colorlinks,linkcolor=black,citecolor=black,urlcolor=black]{hyperref}
\usepackage{amsmath}
\usepackage{amsthm}
\usepackage{amssymb}
\usepackage{mathrsfs}
\usepackage{tikz-cd}
\newtheorem{thm}{Theorem}[subsection]
\newtheorem{prop}[thm]{Proposition}
\newtheorem{cor}[thm]{Corollary}

\newtheorem{lem}[thm]{Lemma}
\theoremstyle{definition}
\newtheorem{df}[thm]{Definition}
\newtheorem{rmk}[thm]{Remark}
\newtheorem{exm}[thm]{Example}

\newtheorem{const}[thm]{Construction}

\theoremstyle{remark}
\numberwithin{equation}{subsection}

\newcommand{\bC}{\mathbf{C}}

\newcommand{\A}{\mathbb{A}}

\newcommand{\E}{\mathbb{E}}

\newcommand{\G}{\mathbb{G}}

\newcommand{\N}{\mathbb{N}}
\renewcommand{\P}{\mathbb{P}}
\newcommand{\Q}{\mathbb{Q}}
\newcommand{\R}{\mathbb{R}}
\newcommand{\T}{\mathbb{T}}

\newcommand{\Z}{\mathbb{Z}}

\newcommand{\cA}{\mathcal{A}}
\newcommand{\cB}{\mathcal{B}}
\newcommand{\cC}{\mathcal{C}}

\newcommand{\cE}{\mathcal{E}}
\newcommand{\cF}{\mathcal{F}}
\newcommand{\cG}{\mathcal{G}}

\newcommand{\cM}{\mathcal{M}}

\newcommand{\cO}{\mathcal{O}}
\newcommand{\cP}{\mathcal{P}}

\newcommand{\cU}{\mathcal{U}}
\newcommand{\cV}{\mathcal{V}}

\newcommand{\rH}{\mathrm{H}}

\newcommand{\rM}{\mathrm{M}}

\newcommand{\sC}{\mathscr{C}}

\newcommand{\sT}{\mathscr{T}}

\newcommand{\sX}{\mathscr{X}}

\newcommand{\et}{\mathrm{\acute{e}t}}

\DeclareMathOperator{\Hom}{Hom}

\DeclareMathOperator{\Spec}{Spec}

\newcommand{\colim}{\mathop{\mathrm{colim}}}
\newcommand{\limit}{\mathop{\mathrm{lim}}}

\newcommand{\id}{\mathrm{id}}
\newcommand{\ul}{\underline}
\newcommand{\ol}{\overline}

\newcommand{\DM}{\mathrm{DM}}
\newcommand{\eff}{{\mathrm{eff}}}

\newcommand{\Shv}{\mathrm{Sh}}

\newcommand{\Spt}{\mathrm{Sp}}
\newcommand{\Sm}{\mathrm{Sm}}

\newcommand{\CAlg}{\mathrm{CAlg}}

\newcommand{\PrL}{\mathrm{Pr}^{\mathrm{L}}}
\newcommand{\Ho}{\mathrm{Ho}}

\newcommand{\boxx}{\square}
\newcommand{\pt}{\mathrm{pt}}
\newcommand{\fib}{\mathrm{fib}}

\newcommand{\tor}{\mathrm{tor}}
\newcommand{\tf}{\mathrm{tf}}

\newcommand{\rank}{\mathrm{rank}}

\newcommand{\SH}{\mathrm{SH}}

\newcommand{\logSH}{\mathrm{logSH}}

\newcommand{\map}{\mathrm{hom}}

\newcommand{\DLambda}{\mathrm{D}(\mathrm{Mod}_{\Lambda})}

\newcommand{\logDA}{\mathrm{log}\mathrm{DA}}
\newcommand{\DA}{\mathrm{DA}}

\newcommand{\Psh}{\mathrm{PSh}}

\newcommand{\cofib}{\mathrm{cofib}}
\newcommand{\colimit}{\mathop{\mathrm{colim}}}

\newcommand{\Sch}{\mathrm{Sch}}
\newcommand{\lSch}{\mathrm{lSch}}

\newcommand{\lSm}{\mathrm{lSm}}
\newcommand{\lFt}{\mathrm{lFt}}
\newcommand{\sFt}{\mathrm{sFt}}
\newcommand{\Ft}{\mathrm{Ft}}
\newcommand{\sSm}{\mathrm{sSm}}

\newcommand{\lSmIFan}{\mathrm{lSmIFan}}

\newcommand{\infPsh}{\mathrm{PSh}}
\newcommand{\infShv}{\mathrm{Sh}}
\newcommand{\Spc}{\mathrm{Spc}}

\newcommand{\Gmm}{\mathbb{G}_m^{\mathrm{log}}}
\newcommand{\op}{\mathrm{op}}

\newcommand{\logDMeff}{\mathrm{log}\mathrm{DM}^{\mathrm{eff}}}
\newcommand{\logDM}{\mathrm{log}{\mathrm{DM}}}
\newcommand{\DMeff}{\mathrm{DM}^{\mathrm{eff}}}
\newcommand{\Map}{\mathrm{Hom}}
\newcommand{\uMap}{\underline{\mathrm{Hom}}}

\newcommand{\gp}{\mathrm{gp}}

\newcommand{\unit}{\mathbf{1}}

\newcommand{\Stab}{\mathrm{Sp}}

\newcommand{\ver}{\mathrm{ver}}
\newcommand{\dver}{\mathrm{dver}}
\newcommand{\divi}{\mathrm{div}}
\newcommand{\ML}{\mathrm{M}\Lambda}
\newcommand{\MGL}{\mathrm{MGL}}

\newcommand{\KGL}{\mathrm{KGL}}

\newcommand{\dtau}{\mathrm{d}\tau}
\newcommand{\sNis}{\mathrm{sNis}}
\newcommand{\dNis}{\mathrm{dNis}}
\newcommand{\ketale}{\mathrm{k\acute{e}t}}
\newcommand{\detale}{\mathrm{d\acute{e}t}}
\newcommand{\setale}{\mathrm{s\acute{e}t}}
\newcommand{\Nis}{\mathrm{Nis}}
\newcommand{\letale}{\mathrm{l\acute{e}t}}
\newcommand{\Zar}{\mathrm{Zar}}

\usepackage[shortlabels]{enumitem}
\usepackage{array}
\usepackage{bbm}

\begin{document}
\title{$\A^1$-homotopy theory of log schemes}
\author{Doosung Park}
\address{Bergische Universit{\"a}t Wuppertal,
Fakult{\"a}t Mathematik und Naturwissenschaften
\\
Gau{\ss}strasse 20, 42119 Wuppertal, Germany}
\email{dpark@uni-wuppertal.de}

\begin{abstract}We construct the $\mathbb{A}^1$-local stable motivic homotopy categories of fs log schemes.
For schemes with the trivial log structure, our construction is equivalent to the original construction of Morel-Voevodsky.
We prove that our construction satisfies the localization property.
As a consequence, we obtain the Grothendieck six-functor formalism for strict morphisms of fs log schemes.
We extend $\mathbb{A}^1$-invariant cohomology theories of schemes to fs log schemes.
In particular, we define motivic cohomology, homotopy K-theory, and algebraic cobordism of fs log schemes.
For any fs log scheme log smooth over a scheme, we express cohomology of its boundary in terms of cohomology of schemes.
\end{abstract}

\maketitle

\section{Introduction}

The goal of \cite{logDM} and \cite{logSH} is to build a suitable framework of motivic homotopy theory of fs log schemes.
What makes their theory distinct from $\A^1$-homotopy theory of schemes \cite{MV} is that $\A^1$ is not inverted.
Instead, they invert $\boxx:=(\P^1,\infty)$, which is the log scheme associated with the open immersion $\P^1-\{\infty\}\to \P^1$.
With this replacement,
the $\infty$-category of logarithmic motives $\logDM(k)$ is constructed in \cite{logDM} for every field $k$, and the stable log motivic $\infty$-category $\logSH(S)$ is constructed in \cite{logSH} for every fs log scheme $S$.

\

One can incorporate various non $\A^1$-invariant cohomology theories in this setting.
For example, Hodge cohomology $H_{\Zar}^p(-,\Omega^q)$ is representable in $\logDM(k)$  for every perfect field $k$ by \cite[Theorem 9.7.1]{logDM}, and logarithmic topological Hochschild homology $\mathrm{THH}$ is representable in $\logSH(B)$ for every scheme $B$ by \cite{logSH}.
For the representability of logarithmic Hochschild homology $\mathrm{HH}$, we refer to \cite{BLPO}.

\

Unfortunately, the Grothendieck six-functor formalism is unknown in this non $\A^1$-invariant setting.
One of the main obstacles is that the localization property does not hold, which consists of the following two conditions for a strict closed immersion $i$ of fs log schemes with its open complement $j$:
\begin{itemize}
\item The induced sequence $j_\sharp j^*\to \id \to i_*i^*$ is a cofiber sequence,
\item $i_*$ is fully faithful.
\end{itemize}
If we have such a property, then $i^*$ has to be an equivalence whenever $i$ is a thickening.
However, topological Hochschild homology is not nil-invariant, which disproves the localization property.
It is an interesting question to determine which parts of the Grothendieck six-functor formalism remain valid in the non $\A^1$-invariant setting.

\

The purpose of this paper is to initiate the study of the Grothendieck six-functor formalism for fs log schemes after inverting $\A^1$.
Our construction of the stable motivic $\infty$-category $\SH(X)$ for fs log schemes $X$ is as follows.
We only consider fs log schemes with Zariski log structure, i.e., the structure sheaf of monoid $\cM_X$ on $X$ is a Zariski sheaf instead of an \'etale sheaf.
We first consider the $\infty$-category of dividing Nisnevich sheaves \cite[Definition 3.1.5]{logDM} of pointed spaces on the category $\lSm/X$ of fs log schemes log smooth over $X$.
Then we apply $(\A^1\cup \ver)$-localization and $\P^1$-stabilization to construct $\SH(X)$.
The class $\ver$ consists of open immersions $j\colon U\to V$ in $\lSm/X$ such that the induced open immersion
\[
j'\colon U-\partial_X U\to V-\partial_X V
\]
is an isomorphism, where $\partial_X U$ denotes the collection of points of $U$ that are not vertical over $X$,
see Definition \ref{vert.1}.
If $U$ is exact log smooth over $X$, then topologically $\partial_X U$ is the relative boundary of $U$ over $X$ according to \cite[Theorem 3.7]{zbMATH05809283}.
Hence the collar neighborhood theorem \cite{MR0133812} ensures that inverting $j'$ is reasonable to consider at least in the case that $U$ and $V$ are exact log smooth over $X$.

\

What we have done in this paper are as follows:
\begin{itemize}
\item We prove that our $\SH(X)$ is equivalent to the original construction \cite[Definition 5.7]{zbMATH01194164} whenever $X$ is a scheme with the trivial log structure.
We also discuss the abelian variant $\DA(X,\Lambda)$  for any commutative ring $\Lambda$ and the variants $\SH_{\setale}(X)$, $\SH_{\ketale}(X)$, $\DA_{\setale}(X,\Lambda)$, and $\DA_{\ketale}(X,\Lambda)$ using the strict \'etale and Kummer \'etale topologies.
This is in \S 2.
\item We prove the localization property for $\SH(X)$ and $\DA(X,\Lambda)$.
As a consequence, we obtain the Grothendieck six-functor formalism for strict morphisms of fs log schemes from the works of Ayoub \cite{Ayo071},\cite{Ayo072} and Cisinski-D\'eglise \cite{CD19}.
This consists of \S 3.
\item We provide a canonical long exact sequence of cohomology groups
\begin{align*}
\cdots
&\to
\E^{p,q}(\ul{X})
\to
\E^{p,q}(X-\partial X)
\oplus
\E^{p,q}(\ul{\partial X})
\to
\E^{p,q}(\partial X)
\\
&\to
\E^{p+1,q}(\ul{X})
\to
\cdots
\end{align*}
for every base scheme $B$, object $\E\in \SH(B)$, fs log scheme $X$ log smooth over $B$, and integer $q$.
Here, $\partial X$ is the strict closed subscheme of $X$ with the reduced scheme structure consisting of the points with nontrivial log structures.
See Corollary \ref{coh.3} for the details.
\end{itemize}

\

One of the main applications of $\A^1$-homotopy theory of fs log schemes is to construct a motivic nearby cycles functor that behaves well in positive and mixed characteristics too,
see \cite[Theorems 1.1--1.4]{lognearby}.
This requires the Grothendieck six-functor formalism for fs log schemes,
which is established in the sequel papers \cite{divspc}, \cite{logGysin}, \cite{logshriek}, and \cite{logsix}.

\

The $\A^1$-localization will expel various non $\A^1$-invariant cohomology theories of schemes.
Nevertheless, $\A^1$-invariant cohomology theories of schemes still remain, and our theory gives a natural way to extend such cohomology theories to fs log schemes.
In particular, we can define
\begin{itemize}
\item $\A^1$-invariant motivic cohomology $H_{\cM}^{p}(X,\Lambda(q))$,
\item homotopy $K$-theory $\mathrm{KH}_p(X)$,
\item and $\A^1$-invariant algebraic cobordism $\mathrm{MGL}^{p,q}(X)$
\end{itemize}
for every fs log scheme $X$, commutative ring $\Lambda$, and integers $p$ and $q$.

\

If $X$ is an fs log scheme proper and log smooth over a scheme, then our cohomology of $\partial X$ is isomorphic to the cohomology of the stable homotopy type at infinity of $X-\partial X$ defined by Dubouloz-D\'eglise-{\O}stv{\ae}r, compare \cite[(4.4.3.a)]{2206.01564} and Theorem \ref{local.8}.
See also \cite{MR2302525} for Levine's notion of motivic tubular neighborhoods.
We also have a similar result with Wildeshaus' boundary motives, see \cite[Proposition 2.4]{zbMATH05042488}.
Hence our theory provides a logarithmic theoretical background for these notions.

\

This paper is a reboot of the author's thesis \cite{ParThesis},  \cite{1707.09435}.
Many classes of morphisms were inverted there for the purpose of the Grothendieck six-functor formalism.
However, that makes the computation of cohomology harder.
Unlike this approach, we invert fewer morphisms that are enough to deduce the above results.

\subsection*{Notation and conventions}
Every fs log scheme in this paper has a Zariski log structure unless otherwise stated.
As in \cite{logDM},
this is no loss of generality for log motives thanks to Niziol's result \cite[Theorem 5.10]{Niz}:
If $X$ is an fs log scheme with an \'etale log structure,
then there exists a log blow-up $Y\to X$ such that $Y$ has a Zariski log structure.

Our standard reference for notation and terminology in log geometry is Ogus's book \cite{Ogu}.
An fs log scheme $X$ is quasi-compact, finite dimensional, or noetherian if $\ul{X}$ is so.
A morphism of fs log schemes $f$ is separated or proper if $\ul{f}$ is so.
We employ the following notation throughout the paper.

\begin{tabular}{l|l}
$\Sch$ & the category of finite dimensional separated noetherian schemes
\\
$\lSch$ &  the category of finite dimensional separated noetherian
\\
&  fs log schemes
\\
$\Sm$ & the class of smooth morphisms in $\Sch$
\\
$\lSm$ & the class of log smooth morphisms in $\lSch$
\\
$\Spc$ & the $\infty$-category of spaces
\\
$\Spc_*$ & the $\infty$-category of pointed spaces
\\
$\Spt$ & the $\infty$-category of spectra
\\
$\Lambda$ & a commutative ring
\\
$\Hom_{\cC}$ & the hom space in an $\infty$-category $\cC$
\\
$\hom_{\cC}$ & the hom spectrum in a stable $\infty$-category $\cC$
\\
$\DLambda$ & the $\infty$-category of chain complexes of $\Lambda$-modules
\\
$\infPsh(\cC,\cU)$ & the $\infty$-category of presheaves on $\cC$ with values in an $\infty$-category $\cU$
\\
$\infShv_t(\cC,\cU)$ & the $\infty$-category of $t$-sheaves on $\cC$ with values in an $\infty$-category $\cU$
\\
$\id\xrightarrow{ad}f_*f^*$ & the unit of an adjunction $(f^*,f_*)$
\\
$f^*f_*\xrightarrow{ad'}\id$ & the counit of an adjunction $(f^*,f_*)$
\end{tabular}

\subsection*{Acknowledgements}
This research was conducted in the framework of the DFG-funded research training group GRK 2240: \emph{Algebro-Geometric Methods in Algebra, Arithmetic and Topology}.
We thank Alberto Vezzani for indicating gaps in a previous version.
We are also grateful to the referee for a detailed and helpful report, which pointed out places in the draft requiring corrections and improved the exposition.

\section{Construction of \texorpdfstring{$\SH$}{SH}}

After explaining several topologies on fs log schemes in \S \ref{topology}, we review the construction of $\logSH$ in \S \ref{review}.
We study basic properties of vertical boundaries in \S \ref{vertical_boundaries}.
The notion of vertical boundaries will play an important role throughout this paper.
We also define a class of morphisms $\ver$, which we will often invert.

The purpose of \S \ref{removing} is to show that for $S\in \Sch$, if we invert $\A^1$ in $\logSH(S)$, then $\ver$ is inverted too.
In \S \ref{vertical_localizations}, we construct $\SH(S)$ for any $S\in \lSch$.
We also show that our $\SH(S)$ is equivalent to the original construction of Morel-Voevodsky if $S\in \Sch$.
The proof is done by giving an explicit description of the localization functor $L_{ver}$.
Furthermore, for $S\in \Sch$, we give an equivalence of $\infty$-categories
\[
\SH(S)
\simeq
(\A^1)^{-1}\logSH(S)
\]

\subsection{Review of topologies on fs log schemes}
\label{topology}

Throughout the paper,
$\tau$ will be one of the strict Nisnevich,
strict \'etale,
and Kummer \'etale topologies on $\lSch$.
Then let $\ul{\tau}$ denote the Nisnevich, \'etale, and \'etale topologies on $\Sch$ respectively,
and let $\dtau$ denote the dividing Nisnevich, dividing \'etale, and log \'etale topologies on $\lSch$ respectively.
This is summarized in the following table with the abbreviations of the topologies:

\

\begin{center}
\begin{tabular}{|c|c|c|}
\hline
$\tau$ & $\ul{\tau}$ & $\dtau$
\\
\hline
$\sNis$ & $\Nis$ & $\dNis$
\\
$\setale$ & $\et$ & $\detale$
\\
$\ketale$ & $\et$ & $\letale$
\\
\hline
\end{tabular}
\end{center}
In this section,
we review the above topologies on fs log schemes.

Recall that the \emph{strict Nisnevich} (resp.\ \emph{strict \'etale}) topology on $\lSch$ is the topology generated by the families $\{f_i\colon U_i\to X\}_{i\in I}$ such that each $f_i$ is strict and $\{\ul{f_i}\colon \ul{U_i}\to \ul{X}\}_{i\in I}$ is a Nisnevich (resp.\ \'etale) covering.
The associated topology is the \emph{strict Nisnevich} (resp.\ \emph{strict \'etale}) topology, and the abbreviation is $\sNis$ (resp.\ $\setale$).

Recall that a \emph{strict Nisnevich} distinguished square is a square in $\lSch$ of the form
\[
Q
:=
\begin{tikzcd}
Y'\ar[d]\ar[r]&
Y\ar[d]
\\
X'\ar[r]&
X
\end{tikzcd}
\]
such that every morphism in $Q$ is strict and the underlying square of schemes $\ul{Q}$ is a Nisnevich distinguished square.
Note that a presheaf on $\lSch$ with values in an $\infty$-category $\cU$ is a strict Nisnevich sheaf if and only if $\cF(Q)$ is cartesian for every strict Nisnevich distinguished square $Q$,
which is a direct analog of the well-known result for the Nisnevich topology.

A morphism in $\lSch$ is called a \emph{dividing cover} if it is a surjective proper log \'etale monomorphism.
Recall from \cite[Definitions 3.1.4, 3.1.5]{logDM} that the \emph{dividing Nisnevich} (resp.\ \emph{dividing \'etale}) topology on $\lSch$ is the coarsest topology finer than the strict Nisnevich (resp.\ strict \'etale) topology and containing dividing covers as covering families.
The abbreviation is $\dNis$ (resp.\ $\detale$).

Recall from \cite[\S 2.1, 9.1]{MR1922832} that the \emph{Kummer \'etale} (resp.\ \emph{log \'etale}) topology on $\lSch$ is the topology generated by the families of Kummer \'etale (resp.\ log \'etale) morphisms $\{f_i\colon U_i\to X\}_{i\in I}$ in $\lSch$ such that $\coprod_{i\in I} U_i \to X$ is surjective (resp.\ universally surjective).
The abbreviation is $\ketale$ (resp.\ $\letale$).

\begin{prop}
\label{divlocal.8}
Let $\cF$ be a presheaf on $\lSch$ with values in an $\infty$-category $\cU$.
\begin{enumerate}
\item[\textup{(1)}]
$\cF$ is a dividing Nisnevich sheaf if and only if $\cF$ is a strict Nisnevich sheaf and invariant under dividing covers.
\item[\textup{(2)}]
$\cF$ is a dividing \'etale sheaf if and only if $\cF$ is a strict Nisnevich sheaf and invariant under dividing covers.
\item[\textup{(3)}]
$\cF$ is a log \'etale sheaf if and only if $\cF$ is a Kummer \'etale sheaf and invariant under dividing covers.
\end{enumerate}
\end{prop}
\begin{proof}
Since dividing covers are monomorphisms in the category of fs log schemes, $\cF$ satisfies descent for dividing covers if and only if $\cF$ is invariant under dividing covers.
This shows (1) and (2).
Together with \cite[Proposition 3.9]{Nak2},
we also have (3).
\end{proof}

\subsection{Review of the construction of \texorpdfstring{$\logSH$}{logSH}}
\label{review}

In this subsection, we review the construction of $\logSH$ in \cite{logSH}.
See also \cite{logDMCras}.

\begin{df}
For every fan $\Sigma$, let $\T_\Sigma$ be the fs log scheme whose underlying scheme is the toric variety associated with $\Sigma$ and whose log structure is the compactifying log structure associated with the open immersion from the torus of the toric variety.

If $B\in \Sch$ and $P$ is an fs monoid, we set
\[
\G_{m,B}:=\G_m\times B,
\;
\A_{P,B}:=\A_P\times B,
\;
\T_{\Sigma,B}:=\T_{\Sigma}\times B.
\]
\end{df}

Next, we recall the definition of premotivic $\infty$-categories, which is a direct adaptation of premotivic triangulated categories in \cite{CD19}.

\begin{df}
\label{shv.7}
Let $\bC$ be a category with a class of morphisms $\cP$ that contains all isomorphisms and is closed under compositions and pullbacks.
Let
\[
\sT
\colon
\bC^{\op}
\to
\CAlg(\PrL)
\]
be a functor of $\infty$-categories.
For a morphism $f$ in $\bC$, we set $f^*:=\sT(f)$.
A right adjoint of $f^*$ is denoted by $f_*$.
We say that $\sT$ is a \emph{$\cP$-premotivic $\infty$-category over $\bC$} if the following conditions are satisfied:
\begin{enumerate}
\item[(i)]
For every morphism $f$ in $\cP$, $f^*$ admits a left adjoint $f_\sharp$.
\item[(ii)]
For every cartesian square
\[
\begin{tikzcd}
X'\ar[r,"g'"]\ar[d,"f'"']&
X\ar[d,"f"]
\\
S'\ar[r,"g"]&
S
\end{tikzcd}
\]
in $\bC$ such that $f\in \cP$, the Beck-Chevalley transformation
\[
Ex\colon f_\sharp' g'^* \to g^*f_\sharp
\]
is an isomorphism.
\item[(iii)] For every morphism $f$ in $\cP$, the Beck-Chevalley transformation
\[
Ex\colon
f_\sharp((-)\otimes f^*(-))
\to
f_\sharp (-) \otimes (-)
\]
is an isomorphism.
\end{enumerate}
For $S\in \bC$, let $\unit_S$ (or simply $\unit$) be the monoidal unit in $\sT(S)$.
If $f\colon X\to S$ is a morphism in $\cP$, then we set $M(X):=f_\sharp \unit_X$.

Let $\sT'$ be another $\cP$-premotivic $\infty$-category over $\bC$.
We say that a natural transformation $\varphi\colon \sT\to \sT'$ is a \emph{functor of $\cP$-premotivic $\infty$-categories} if the Beck-Chevalley transformation
\[
f_\sharp \varphi(Y)
\to
\varphi(X) f_\sharp
\]
is an isomorphism for every $\cP$-morphism $f\colon Y\to X$.
\end{df}

\begin{const}
\label{shv.8}
With the above notation, suppose that $\cA$ is a class of morphisms in $\sT(-)$ closed under $f^*$ for all morphisms $f$ in $\cC$, $f_\sharp$ for all morphisms $f$ in $\cP$, and $\otimes M(Y)$ for all $Y\in \cP/-$.
Then as in Robalo's thesis \cite[\S 9.1]{Robalothesis}, we can construct the pointwise localization $\cA^{-1}\sT$ and show that it is a $\cP$-premotivic $\infty$-category.
Here, $\cA^{-1}\sT(X)$ is the $\cA$-localization of $\sT(X)$ in the sense of \cite[\S 5.5.4]{HTT} for all $X\in \cC$.
\end{const}

\begin{df}
Let $\cC$ be a presentable symmetric monoidal $\infty$-category, and let $X$ be an object of $\cC$.
We set
\[
\Stab_X(\cC)
:=
\limit(\cdots \xrightarrow{\Omega_{X}} \cC\xrightarrow{\Omega_{X}} \cC),
\]
where $\Omega_{X}:=\ul{\Map}(-,X)$, and $\ul{\Map}$ denotes the internal hom.
See \cite[\S 2.2.1]{zbMATH06374152} for the details.

If the cyclic permutation $(123)$ on $X\otimes X\otimes X$ is isomorphic to the identity morphism, then there exists a presentable symmetric monoidal $\infty$-category whose underlying $\infty$-category is $\Stab_X(\cC)$ by \cite[Corollary 2.22]{zbMATH06374152}.
\end{df}

\begin{df}
\label{logSH.1}
In $\A^1$-homotopy theory of schemes, we have the commutative diagram of $\Sm$-premotivic $\infty$-categories over $\Sch$
\[
\begin{tikzcd}
\rH_{\ul{\tau}} \ar[r,"(-)_+"]&
\rH_{\ul{\tau},*} \ar[r,"\Sigma_{S^1}^\infty"]&
\SH_{\ul{\tau}}^\eff\ar[r,"\Sigma_{\G_m}^\infty"]\ar[d,"{\Lambda^{(-)}}"']&
\SH_{\ul{\tau}}\ar[d,"{\Lambda^{(-)}}"]
\\
&
&
\DA_{\ul{\tau}}^\eff(-,\Lambda)\ar[r,"\Sigma_{\G_m}^\infty"]&
\DA_{\ul{\tau}}(-,\Lambda),
\end{tikzcd}
\]
where
\begin{gather*}
\rH_{\ul{\tau}}(S)
:=
(\A^1)^{-1}\infShv_{\ul{\tau}}(\Sm/S,\Spc),
\\
\rH_{\ul{\tau},*}(S)
:=
(\A^1)^{-1}\infShv_{\ul{\tau}}(\Sm/S,\Spc_*),
\\
\SH_{\ul{\tau}}^\eff(S)
:=
(\A^1)^{-1}\infShv_{\ul{\tau}}(\Sm/S,\Spt),
\\
\DA_{\ul{\tau}}^\eff(S,\Lambda)
:=
(\A^1)^{-1}\infShv_{\ul{\tau}}(\Sm/S,\DLambda),
\\
\SH_{\ul{\tau}}(S)
:=
\Stab_{\G_m}(\SH_{\ul{\tau}}^\eff(S)),
\;
\DA_{\ul{\tau}}(S,\Lambda)
:=
\Stab_{\G_m}(\DA_{\ul{\tau}}^\eff(S,\Lambda)).
\end{gather*}
For simplicity of notation, $\A^1$ denotes the class of projections $X\times \A^1\to X$ for $X\in \Sm/S$ (or $X\in \lSm/S$).
We often omit $\Nis$ in $\rH_{\Nis}$ and the other five $\infty$-categories if $\ul{\tau}=\Nis$.
For a presheaf $\cF$ on $\Sm/S$ (or $\lSm/S$), let $\Lambda^{\cF}$ be the free presheaf of $\Lambda$-modules associated with $\cF$.
\end{df}

\begin{df}
For $S\in \Sch$ and $X\in \Sm/S$ with a strict normal crossing divisor $Z$ on $X$ over $S$ \cite[Definition 7.2.1]{logDM}, let $(X,Z)$ be the fs log scheme whose underlying scheme is $X$ and whose log structure is the compactifying log structure \cite[Definition III.1.6.1]{Ogu} associated with the open immersion $X-Z\to X$.
\end{df}

\begin{df}
We set $\square:=(\P^1,\infty)$ and $\Gmm:=(\P^1,0+\infty)$,
where we regard $\P^1$ as a smooth scheme over $\Spec(\Z)$.
In \cite{logSH}, we have the commutative diagram of $\lSm$-premotivic $\infty$-categories over $\lSch$
\[
\begin{tikzcd}
\logSH_{\tau}^\eff\ar[r,"\Sigma_{\Gmm}^\infty"]\ar[d,"{\Lambda^{(-)}}"']&
\logSH_{\tau}\ar[d,"{\Lambda^{(-)}}"]
\\
\logDA_{\tau}^\eff(-,\Lambda)\ar[r,"\Sigma_{\Gmm}^\infty"]&
\logDA_{\tau}(-,\Lambda),
\end{tikzcd}
\]
where
\begin{gather*}
\logSH_{\tau}^\eff(S)
:=
\boxx^{-1}\infShv_{\dtau}(\lSm/S,\Spt),
\\
\logDA_{\tau}^\eff(S,\Lambda)
:=
\boxx^{-1}\infShv_{\dtau}(\lSm/S,\DLambda),
\\
\logSH_{\tau}(S)
:=
\Stab_{\Gmm}(\logSH_{\tau}^\eff(S)),
\;
\logDA_{\tau}(S,\Lambda)
:=
\Stab_{\Gmm}(\logDA_{\tau}^\eff(S,\Lambda)).
\end{gather*}
As before, $\boxx$ denotes the class of projections $X\times \boxx\to X$ for $X\in \lSm/S$.
We often omit $\sNis$ in $\logSH_{\dNis}^\eff$ and the other three $\infty$-categories if $\tau=\sNis$.
As a consequence of \cite[Proposition 2.5.11]{logSH}, and we have an isomorphism $\Sigma_{\Gmm}^\infty \Sigma_{S^1}^\infty \simeq \Sigma_{\P^1}^\infty$.

If $k$ is a field, then we also have the $\infty$-category $\logDM^\eff_{\tau}(k)$, see \cite[Definition 5.2.1]{logDM}.
We set
\[
\logDM_{\tau}(k):=\Stab_{\Gmm}(\logDM_{\tau}^\eff(k)).
\]
For $X\in \lSm/k$, let $M(X)$ denote the associated motives in $\logDM_{\tau}^\eff(k)$ and $\logDM_{\tau}(k)$.
\end{df}

\subsection{Vertical boundaries}
\label{vertical_boundaries}

\begin{prop}
\label{vert.4}
Let $\theta\colon P\to Q$ be a map of saturated monoids.
Then the following conditions are equivalent.
\begin{enumerate}
\item[\textup{(1)}]
$\theta$ is vertical,
i.e.,
the cokernel of $\theta$ computed in the category of integral monoids is a group.
\item[\textup{(2)}]
The cokernel of $\theta$ computed in the category of saturated monoids is a group.
\item[\textup{(3)}]
The face of $Q$ generated by $\theta(P)$ is $Q$.
\end{enumerate}
\end{prop}
\begin{proof}
The equivalence between (1) and (3) is due to \cite[Remark I.4.3.2]{Ogu}.
To show the equivalence between (1) and (2),
it suffices to show the following claim:
For an integral monoid $M$ such that its saturation $M^\mathrm{sat}$ is a group,
$M$ is a group.
For this,
consider an element $x\in M$.
Since $M^\mathrm{sat}$ is a group,
there exists an integer $n>0$ such that $n(-x)\in M$.
This implies $-x=n(-x)+(n-1)x\in M$.
\end{proof}

For a monoid $P$, the notation $P^*$ means the submonoid of the units of $P$,
and the notation $\ol{P}$ means $P/P^*$.
For a map of monoids $\theta\colon P\to Q$,
we have the induced map $\ol{\theta}\colon \ol{P}\to \ol{Q}$.

\begin{prop}
\label{vert.3}
Let $\theta\colon P\to Q$ and $\eta\colon Q\to R$ be maps of saturated monoids.
\begin{enumerate}
\item[{\rm (1)}]
$\theta$ is vertical if and only if $\overline{\theta}$ is vertical.
\item[{\rm (2)}]
If $\theta$ is vertical and $G$ is a face of $Q$, then the induced map $P_{\theta^{-1}(G)}\to Q_G$ is vertical.
\item[{\rm (3)}]
If $\theta$ and $\eta$ are vertical, then $\eta\theta$ is vertical.
\item[{\rm (4)}]
If $\eta\theta$ is vertical, then $\eta$ is vertical.
\item[{\rm (5)}]
If $\eta$ is exact and $\eta\theta$ is vertical, then $\theta$ is vertical.
\end{enumerate}
\end{prop}
\begin{proof}
In \cite[Proposition I.4.3.3]{Ogu}, (1)--(4) are proven.
For (5), let $G$ be the face of $Q$ generated by $\theta(P)$.
By \cite[Proposition I.4.2.2]{Ogu}, there exists a face $H$ of $R$ such that $\eta^{-1}(H)=G$.
This implies that the face generated by $\eta\theta(P)$ is contained in $H$.
Since $\eta\theta$ is vertical, we have $H=R$.
Hence we have $G=Q$, i.e., $\theta$ is vertical.
\end{proof}

\begin{prop}
\label{vert.10}
Let
\[
\begin{tikzcd}
P\ar[d,"\theta"']\ar[r]&
P'\ar[d,"\theta'"]
\\
Q\ar[r]&
Q'
\end{tikzcd}
\]
be a cocartesian square of saturated monoids.
Then $\theta$ is vertical if and only if $\theta'$ is vertical.
\end{prop}
\begin{proof}
The cokernels of $\theta$ and $\theta'$ computed in the category of saturated monoids are isomorphic.
Proposition \ref{vert.4}(2) finishes the proof.
\end{proof}

\begin{df}
Let $X$ be an fs log scheme.
The \emph{boundary of $X$}, denoted $\partial X$, is the set of points of $X$ whose log structure is nontrivial.
\end{df}

\begin{df}
\label{vert.1}
Let $f\colon X\to S$ be a morphism of fs log schemes.
The \emph{vertical boundary of $X$ over $S$}, denoted $\partial_S X$, is the set of points $x$ of $X$ such that $\cM_{S,f(x)}\to \cM_{X,x}$ is not vertical.
We say that $f$ is \emph{vertical} if $\partial_S X=\emptyset$.
If $S=\Spec(\Z)$, then we have $\partial_{S}X = \partial X$.
\end{df}

\begin{prop}
\label{vert.7}
Let $f\colon X\to S$ be a morphism of fs log schemes.
Then $f$ is vertical if and only if $f$ admits a chart $\theta\colon P\to Q$ Zariski locally on $S$ and $X$ such that $\theta$ is vertical.
\end{prop}
\begin{proof}
Let $x$ be a point of $X$.
Assume that $f$ is vertical.
We can find a chart $\theta\colon P\to Q$ near $x$ and $f(x)$ such that $\ol{P}\to \ol{Q}$ is isomorphic to $\ol{\cM}_{S,f(x)}\to \ol{\cM}_{X,x}$ by \cite[Remark II.2.3.2]{Ogu}.
Use Proposition \ref{vert.3}(1) to deduce that $\theta$ is vertical.

Conversely,
assume that $f$ admits a chart $\theta\colon P\to Q$ such that $\theta$ is vertical.
The map
$\ol{\cM}_{S,f(x)}\to \ol{\cM}_{X,x}$ is isomorphic to the induced map $P/\theta^{-1}(G)\to Q/G$ for some face $G$ of $Q$, which is vertical by Proposition \ref{vert.3}(1),(2).
\end{proof}

\begin{prop}
\label{vert.2}
Let $f\colon X\to S$ be a morphism of fs log schemes.
Then $\partial_S X$ is a closed subset of $X$.
\end{prop}
\begin{proof}
Assume $x\in X-\partial_S X$.
Proposition \ref{vert.7} implies that there exists a neighborhood $U$ of $x$ such that the induced morphism $f\colon U\to S$ is vertical.
This shows that $X-\partial_S X$ is an open subset of $X$.
\end{proof}

It follows that we can regard $X-\partial_S X$ as an open subscheme of $X$.
We impose the reduced scheme structure to consider $\partial_S X$ as a strict closed subscheme of $X$.
We also have the following consequence:
If $j\colon U\to X$ is an open immersion of fs log schemes such that the composition $fj\colon U\to S$ is vertical,
then $fj$ factors through $X-\partial_S X$.

\begin{prop}
\label{vert.9}
Let $f\colon X\to S$ and $g\colon Y\to X$ be morphisms of fs log schemes.
\begin{enumerate}
\item[{\rm (1)}]
If $f$ and $g$ are vertical, then $gf$ is vertical.
\item[{\rm (2)}]
If $gf$ is vertical, then $g$ is vertical.
\item[{\rm (3)}]
If $g$ is exact and $gf$ is vertical, then $f$ is vertical.
\item[{\rm (4)}]
If $f$ is vertical, then $Y-\partial_X Y$ is an open subscheme of $Y-\partial_S Y$.
\item[{\rm (5)}]
$Y-\partial_S Y$ is an open subscheme of $Y-\partial_X Y$.
\item[{\rm (6)}] If $g$ is exact, then $g$ maps $Y-\partial_S Y$ into $X-\partial_S X$.
\end{enumerate}
\end{prop}
\begin{proof}
Let $y$ be a point of $y$.
Apply Proposition \ref{vert.3}(3)--(5) to the induced morphisms
\[
\cM_{Y,y}
\to
\cM_{X,g(y)}
\to
\cM_{S,f(g(y))}
\]
to show (1)--(3).
These immediately imply (4)--(6).
\end{proof}

\begin{prop}
\label{vert.5}
Let
\[
\begin{tikzcd}
X'\ar[d,"f'"']\ar[r,"g'"]&
X\ar[d,"f"]
\\
S'\ar[r,"g"]&
S
\end{tikzcd}
\]
be a cartesian square of fs log schemes.
Then
\[
(X-\partial_S X)\times_S S'
\simeq
X'-\partial_{S'}X'.
\]
\end{prop}
\begin{proof}
Let $x'$ be a point of $X'$, and we set $x:=g'(x')$, $s':=f'(x')$, and $s:=f(x)$.
We need to show that $\cM_{S,s}\to \cM_{X,x}$ is vertical if and only if $\cM_{S',s'}\to \cM_{X',x'}$ is vertical.

The question is Zariski local on $S$, $S'$, and $X$.
By \cite[Remark II.2.3.2]{Ogu},
we may assume that $S$ admits a chart $P$ such that $\ol{P}$ is isomorphic to $\ol{M}_{S,s}$.
Use \cite[Proposition II.2.4.2]{Ogu} to obtain charts $P\to Q$ and $P\to P'$ for $f$ and $g$.
By \cite[Remark II.2.3.2]{Ogu} again,
we may assume that $\ol{P}\to \ol{Q}$ is isomorphic to $\ol{\cM}_{S,s}\to \ol{\cM}_{X,x}$ and $\ol{P}\to \ol{P'}$ is isomorphic to $\ol{\cM}_{S,s}\to \ol{\cM}_{S',s'}$.
For some face $F$ of $Q':=P'\oplus_P Q$ whose inverse images in $P'$ and $Q$ are $P'^*$ and $Q^*$, we have $\ol{\cM}_{X',x'}\simeq Q'/F$.
By \cite[Corollaries I.2.3.8, I.4.2.16]{Ogu},
the induced map of sets $\Spec(Q')\to \Spec(P')\times_{\Spec(P)}\Spec(Q)$ is injective.
Hence the two faces $F$ and $Q'^*$ of $Q'$ agree,
so $\ol{P'}\to \ol{Q'}$ is isomorphic to $\ol{\cM}_{S,s'}\to \ol{\cM}_{X,X'}$.
Together with Proposition \ref{vert.3}(1),
we see that $\cM_{S,s}\to \cM_{X,x}$ (resp.\ $\cM_{S's'}\to \cM_{X',x'}$) is vertical if and only if $P\to Q$ (resp.\ $P'\to Q'$) is vertical.
Apply Proposition \ref{vert.10} to the cocartesian square
\[
\begin{tikzcd}
P\ar[d]\ar[r]&
P'\ar[d]
\\
Q\ar[r]&
Q'
\end{tikzcd}
\]
to finish the proof.
\end{proof}

\begin{prop}
\label{vert.8}
Let $f\colon X\to S$ be a log unramified morphism of fs log schemes.
Then $f$ is vertical.
\end{prop}
\begin{proof}
Combine \cite[Proposition IV.3.4.1]{Ogu} and \cite[Lemma 5.6(1)]{zbMATH05809283}.
\end{proof}

In particular,
every dividing cover is vertical.

\begin{df}
For an fs monoid $P$ such that $P^\gp$ is torsion free, let $\Spec(P)$ be the fan associated with the dual monoid of $P$ in $P^\gp$.
\end{df}

\begin{df}
\label{vert.11}
For a morphism of fans $\theta\colon \Sigma\to \Delta$,
let $\Sigma-\partial_{\Delta}\Sigma$ be the subfan of $\Sigma$ consisting of $\sigma\in \Sigma$ such that $\theta$ maps every nontrivial faces of $\sigma$ to a nontrivial cone of $\Delta$.
Equivalently, $\Sigma-\partial_{\Delta}\Sigma$ is the largest subfan of $\Sigma$ such that for every cone $\sigma\in \Sigma-\partial_{\Delta}\Sigma$ with $\theta(\sigma)=0$, we have $\sigma=0$.
\end{df}

\begin{prop}
Let $\theta\colon \Sigma\to \Delta$ be a morphism of fans.
Then there is a canonical isomorphism
\[
\T_\Sigma-\partial_{\T_\Delta}\T_\Sigma
\simeq
\T_{\Sigma-\partial_\Delta \Sigma}.
\]
\end{prop}
\begin{proof}
A map $\eta\colon P\to Q$ of fs monoids is vertical if and only if for every face $F$ of $Q$ with $\eta^{-1}(F)=P$, we have $F=Q$.
From the duality theory of cones \cite[Theorem I.2.3.12(4)]{Ogu}, we deduce the following:
$\T_\theta$ is vertical if and only if for every cone $\sigma\in \Sigma$ with $\theta(\sigma)=0$, we have $\sigma=0$.
This shows that $\T_{\Sigma-\partial_{\Delta}\Sigma}$ is the largest open subscheme of $\T_{\Sigma}$ that is vertical over $\T_\Delta$.
\end{proof}

\subsection{Removing boundaries}
\label{removing}

Throughout this subsection, we fix $B\in \Sch$ and a stable symmetric monoidal $\infty$-category $\cC$ with a functor
\[
M
\colon
\lSm/B \to \cC
\]
satisfying the following properties:
\begin{itemize}
\item For $X,Y\in \lSm/B$, there exists a canonical isomorphism
\[
M(X)\otimes M(Y)
\simeq
M(X\times_B Y).
\]
\item ($\A^1$-invariance) For $X\in \lSm/B$, the morphism
\[
M(X\times \A^1)\to M(X)
\]
induced by the projection $X\times \A^1\to X$ is an isomorphism.
\item ($\boxx$-invariance) For $X\in \lSm/B$, the morphism
\[
M(X\times \boxx)\to M(X)
\]
induced by the projection $X\times \boxx\to X$ is an isomorphism.
\item (Dividing invariance) For every dividing cover $f\colon Y\to X$ in $\lSm/B$, $M(f)\colon M(Y)\to M(X)$ is an isomorphism.
\item (Strict Nisnevich descent) For every strict Nisnevich distinguished square $Q$ in $\lSm/B$, $M(Q)$ is cocartesian.
\end{itemize}

\begin{exm}
\label{2.4}
We have the following examples of $\cC$
\begin{gather*}
(\A^1)^{-1}\logSH_{\tau}^\eff(B),
\;
(\A^1)^{-1}\logSH_{\tau}(B),
\;
(\A^1)^{-1}\logDA_{\tau}^\eff(B,\Lambda),
\\
(\A^1)^{-1}\logDA_{\tau}(B,\Lambda),
\;
(\A^1)^{-1}\logDM_{\tau}^\eff(k,\Lambda),
\;
(\A^1)^{-1}\logDM_{\tau}(k,\Lambda),
\end{gather*}
where $k$ is a field in the fifth and sixth ones.
\end{exm}

\begin{lem}
\label{2.2}
Let $Z_1,\ldots,Z_r$ be the axes divisors of $B\times \A^r$ with $r\geq 1$.
If we set $Y:=(B\times \A^r,Z_1+\cdots+Z_r)$ and $W:=Z_1\cap \cdots \cap Z_r$, then the induced morphism
\[
M(Y-W)
\to
M(Y)
\]
is an isomorphism.
\end{lem}
\begin{proof}
We set $\ol{Y}:=(B\times (\P^1)^r,Z_1+\cdots+Z_r)$, and we view $\A^r$ as the open subscheme of $\P^r$ excluding the hyperplane at $\infty$.
The square
\[
\begin{tikzcd}
Y-W\ar[d]\ar[r]&
Y\ar[d]
\\
\ol{Y}-W\ar[r]&
\ol{Y}
\end{tikzcd}
\]
is a strict Nisnevich distinguished square.
Hence we need to show that the induced morphism
\[
M(\ol{Y}-W)\to M(\ol{Y})
\]
is an isomorphism.
Since $\ol{Y}\simeq B\times (\P^1,0)^r \simeq B\times \boxx^r$, it remains to show $M(\ol{Y}-W)\simeq M(B)$.

For any nonempty subset $I=\{i_1,\ldots,i_s\}\subset \{1,\ldots,r\}$, we set $Z_I:=Z_{i_1}+ \cdots + Z_{i_s}$.
Since $\ol{Y}-W=(\ol{Y}-Z_1)\cup \cdots \cup (\ol{Y}-Z_r)$, it suffices to show $M_S(\ol{Y}-Z_I)\simeq M(B)$ for all $I$ by dividing Nisnevich descent.
This follows from $\ol{Y}-Z_I\simeq B\times \A^s \times (\P^1,0)^{r-s}\simeq B\times \A^s \times \boxx^{r-s}$.
\end{proof}

\begin{prop}
\label{2.3}
For every $Y\in \lSm/B$, the induced morphism
\[
M(Y-\partial Y)
\to
M(Y)
\]
is an isomorphism.
\end{prop}
\begin{proof}
The question is Zariski local on $Y$.
Hence we may assume that $Y$ admits a chart $P$ such that $P$ is sharp by \cite[Proposition II.2.3.7]{Ogu}.
There exists a dividing cover $Y'\to Y$ such that $\partial Y'$ is a strict normal crossing divisor by toric resolution of singularities \cite[Theorem 11.1.9]{CLStoric}.
Furthermore, the induced morphism $Y'-\partial Y'\to Y-\partial Y$ is an isomorphism.
Owing to dividing Nisnevich descent, it suffices to show that the induced morphism $M(Y'-\partial Y')\to M(Y')$ is an isomorphism.
In other words, we reduce to the case when $\partial Y$ is a strict normal crossing divisor.

Let $Z_1,\ldots,Z_r$ be the irreducible components of $\partial Y$.
We proceed by induction on the maximum codimension $d$ of the nonempty intersections of $Z_1,\ldots,Z_r$ in $Y$.
If $d=0$, there is nothing to prove.
Hence assume $d>0$.

The question is strict Nisnevich local on $Y$.
Hence we may assume that there exists a cartesian square
\[
\begin{tikzcd}
\ul{\partial Y}\ar[d]\ar[r]&
B\times \Spec(\Z[x_1,\ldots,x_n]/(x_1\cdots x_r))\ar[d]
\\
\ul{Y}\ar[r]&
B\times \Spec(\Z[x_1,\ldots,x_n])
\end{tikzcd}
\]
such that the horizontal morphisms are \'etale and the vertical morphisms are the obvious closed immersions.

We set $W:=Z_1\cap \cdots\cap Z_r$.
In this setting, \cite[Construction 7.2.8]{logDM} yields a commutative diagram
\[
\begin{tikzcd}
W\ar[d]\ar[r,leftarrow]&
W\ar[d]\ar[r]&
W\ar[d]
\\
X\ar[r,leftarrow]&
X''\ar[r]&
X'
\end{tikzcd}
\]
such that each square is cartesian, the horizontal morphisms are \'etale, and the right vertical morphism can be identified with the obvious closed immersion
\[
W\times \Spec(\Z[v_1,\ldots,v_r]/(v_1,\ldots,v_r))
\to
W\times \Spec(\Z[v_1,\ldots,v_r]).
\]

We set $Z_i':=W\times \Spec(\Z[v_1,\ldots,v_r]/(v_i))$, $Z_i'':=Z_i\times_X X''$, $Y':=(X',Z_1'+\cdots+Z_r')$, and $Y'':=(X'',Z_1''+\cdots+Z_r'')$.
The squares
\[
\begin{tikzcd}
Y''-W\ar[r]\ar[d]&
Y''\ar[d]
\\
Y-W\ar[r]&
Y,
\end{tikzcd}
\;\;
\begin{tikzcd}
Y''-W\ar[r]\ar[d]&
Y''\ar[d]
\\
Y'-W'\ar[r]&
Y'
\end{tikzcd}
\]
are strict Nisnevich distinguished squares.
Hence the squares
\[
\begin{tikzcd}
M(Y''-W)\ar[r]\ar[d]&
M(Y'')\ar[d]
\\
M(Y-W)\ar[r]&
M(Y),
\end{tikzcd}
\;\;
\begin{tikzcd}
M(Y''-W)\ar[r]\ar[d]&
M(Y'')\ar[d]
\\
M(Y'-W')\ar[r]&
M(Y')
\end{tikzcd}
\]
are cocartesian.
By Lemma \ref{2.2}, the morphism $M(Y'-W)\to M(Y')$ is an isomorphism.
It follows that the morphism $M(Y-W)\to M(Y)$ is an isomorphism too.

By induction, the morphism $M((Y-W)-\partial (Y-W))\to M(Y-W)$ is an isomorphism.
To conclude, observe that the induced morphism $(Y-W)-\partial (Y-W)\to Y-\partial Y$ is an isomorphism.
\end{proof}

\subsection{Vertical localizations}
\label{vertical_localizations}

\begin{df}
For $S\in \lSch$,
let $\ver_S$ (or $\ver$ for short) be the class of morphisms in $\lSm/S$ consisting of open immersions $U\to V$ such that $U-\partial_S U\to V-\partial_S V$ is an isomorphism.
\end{df}

\begin{prop}
\label{comp.4}
Let $f\colon S'\to S$ be a morphism in $\lSch$.
If $V\to U$ is a morphism in $\ver_S$, then the pullback $V\times_S S'\to U\times_S S'$ is a morphism in $\ver_{S'}$.
If $f$ is in $\lSm$ and $V'\to U'$ is a morphism in $\ver_{S'}$, then $V'\to U'$ is a morphism in $\ver_S$.
\end{prop}
\begin{proof}
The first claim follows from Proposition \ref{vert.5}.
For the second claim, we set $W':=V'-\partial_{S'}V'\simeq U'-\partial_{S'}U'$.
Since $W'$ is an open subscheme of $V'$, $W'-\partial_S W'$ is an open subscheme of $V'-\partial_S V'$.
On the other hand, $V'-\partial_S V'$ is an open subscheme of $W'$ by Proposition \ref{vert.9}(5), so $V'-\partial_S V'$ is an open subscheme of $W'-\partial_S W'$.
Hence we have $V'-\partial_S V'\simeq W'-\partial_S W'$.
We similarly have $U'-\partial_S U'\simeq W'-\partial_S W'$.
\end{proof}

Together with Construction \ref{shv.8}, we obtain an $\lSm$-premotivic $\infty$-category over $\lSch$
\[
(\A^1\cup \ver)^{-1}\infShv_{\dtau}(\lSm/-,\Spc).
\]
Since $\boxx-\partial \boxx\simeq \A^1$, $\boxx$ is already inverted here.
Hence we have the localization functor
\begin{equation}
L_{\ver}
\colon
\square^{-1}\infShv_{\dtau}(\lSm/-,\Spc) \to (\A^1\cup \ver)^{-1}\infShv_{\dtau}(\lSm/-,\Spc).
\end{equation}

\begin{rmk}
For $S\in \lSch$,
let $\ver'$ be the class of morphisms in $\lSm/S$ consisting of the open immersions $U-\partial_S U\to U$ for all $U\in \lSm/S$.
Then $\ver'$ is a subclass of $\ver$, and there is a canonical isomorphism
\[
(\A^1\cup \ver')^{-1}\infShv_{\dtau}(\lSm/S,\Spc)
\simeq
(\A^1\cup \ver)^{-1}\infShv_{\dtau}(\lSm/S,\Spc).
\]
However, $\ver'$ does not satisfy the second claim in Proposition \ref{comp.4}.
For example, if $f\colon S'\to S$ is a morphism in $\lSm$ that is not vertical, then the identity morphism $S'\to S'$ is not in $\ver_{S}$ even though it is in $\ver_{S'}$.
\end{rmk}

\begin{const}
\label{comp.5}
For $S\in \Sch$,
let $\lambda\colon \Sm/S\to \lSm/S$ be the functor sending $X\in \Sm/S$ to $X$, and let $\omega\colon \lSm/S\to \Sm/S$ be the functor sending $Y\in \lSm/S$ to $Y-\partial Y$.
Since $\lambda$ is a left adjoint of $\omega$, we have adjoint functors
\begin{equation}
\omega^\sharp
:
\infPsh(\Sm/S,\Spc)
\rightleftarrows
\infPsh(\lSm/S,\Spc)
:
\omega_\sharp
\end{equation}
where $\omega^\sharp$ is left adjoint to $\omega_\sharp$, $\omega^\sharp (X)\simeq \lambda(X)$ for $X\in \Sm/S$, and $\omega_\sharp (X)\simeq \omega(X)$.
Since $\lambda$ and $\omega$ preserve products, $\omega^\sharp$ and $\omega_\sharp$ are monoidal.

The functor $\lambda$ maps $\A^1$ to $\A^1$ and $\ul{\tau}$-coverings to $\tau$-coverings.
The functor $\omega$ maps $\A^1$ to $\A^1$ and $\tau$-coverings to $\ul{\tau}$-coverings.
Hence we have induced adjoint functors
\begin{equation}
\label{comp.5.1}
\omega^\sharp
:
\rH_{\ul{\tau}}(S)
\rightleftarrows
(\A^1)^{-1}\infShv_\tau(\lSm/S,\Spc)
:
\omega_\sharp
\end{equation}
such that both preserve colimits.
Let $\omega^*$ be a right adjoint of $\omega_\sharp$.

For $X\in \Sm/S$, we have an isomorphism $\omega_\sharp \omega^\sharp (X) \simeq X$.
Hence $\omega_\sharp \omega^\sharp \simeq \id$, so $\omega^\sharp$ and $\omega^*$ are fully faithful.

For $\cF\in (\A^1)^{-1}\infShv_\tau(\lSm/S,\Spc)$ and $Y\in \lSm/S$, we have an isomorphism
\[
\omega^*\omega_\sharp \cF(Y)
\simeq
\cF(Y-\partial Y)=\cF(\omega(Y)).
\]
This means that $\omega^*\omega_\sharp \cF$ is $\ver$-local.
Since $\omega$ maps $\dtau$-coverings to $\ul{\tau}$-coverings, $\omega^*\omega_\sharp \cF$ is $\dtau$-local.
Hence
\[
\omega^*\omega_\sharp \cF(Y)\in (\A^1\cup \ver)^{-1}\infShv_{\dtau}(\lSm/S,\Spc).
\]
On the other hand, if $\cF\in (\A^1\cup \ver)^{-1}\infShv_{\dtau}(\lSm/S,\Spc)$, then $\omega^*\omega_\sharp \cF\simeq \cF$ since the open immersion $Y-\partial Y\to Y$ belongs to $\ver$.
It follows that the essential image of $\omega^*\omega_\sharp$ is isomorphic to $(\A^1\cup \ver)^{-1}\infShv_{\dtau}(\lSm/S,\Spc)$.
In summary, the localization functor
\[
L_{\dtau\cup \ver}
\colon
(\A^1)^{-1}\infShv_{\tau}(\lSm/S,\Spc)
\to
(\A^1\cup \ver)^{-1}\infShv_{\dtau}(\lSm/S,\Spc)
\]
satisfies
\begin{equation}
\label{comp.5.2}
L_{\dtau\cup \ver}\cF(Y)
\simeq
\cF(Y-\partial Y).
\end{equation}
Furthermore, we have a canonical equivalence of $\infty$-categories
\begin{equation}
\label{comp.5.3}
\rH_{\ul{\tau}}(S)
\simeq
(\A^1\cup \ver)^{-1}\infShv_{\dtau}(\lSm/S,\Spc).
\end{equation}
\end{const}

\begin{df}
\label{comp.3}
For $S\in \lSch$, we define
\begin{gather*}
\rH_{\tau}(S)
:=
(\A^1\cup \ver)^{-1}\infShv_{\dtau}(\lSm/S,\Spc),
\\
\rH_{*,\tau}(S)
:=
(\A^1\cup \ver)^{-1}\infShv_{\dtau}(\lSm/S,\Spc_*),
\\
\SH_{\tau}^\eff(S)
:=
(\A^1\cup \ver)^{-1}\infShv_{\dtau}(\lSm/S,\Spt),
\\
\DA_{\tau}^\eff(S,\Lambda)
:=
(\A^1\cup \ver)^{-1}\infShv_{\dtau}(\lSm/S,\DLambda),
\\
\SH_{\tau}(S)
:=
\Stab_{\G_m}(\SH_{\tau}^\eff(S)),
\;
\DA_{\tau}(S,\Lambda)
:=
\Stab_{\G_m}(\DA_{\tau}^\eff(S,\Lambda)).
\end{gather*}
Let $\cV$ be one of $\Spt$ and $\DLambda$.
For abbreviation, we also set
\begin{gather*}
\SH_{\tau}^\eff(S,\cV)
:=
(\A^1\cup \ver)^{-1}\infShv_{\dtau}(\Sm/S,\cV),
\\
\SH_{\tau}(S,\cV)
:=
\Stab_{\G_m}(\SH^\eff(S,\cV)).
\end{gather*}
We often omit the subscripts $\tau$ in $\rH_\tau(S)$ and other five $\infty$-categories if $\tau=\sNis$.
\end{df}

\begin{rmk}
\label{comp.7}
When $S\in \Sch$, our definition of $\rH_{\tau}(S)$ is equivalent to $\rH_{\ul{\tau}}(S)$ in \eqref{logSH.1} due to \eqref{comp.5.3}.
We can prove similar results for the other five $\infty$-categories, where we need to use the fact that both $\lambda$ and $\omega$ map $\G_m$ to $\G_m$ for $\SH$ and $\DA$.
We similarly have \eqref{comp.5.2} for these categories too.

Proposition \ref{comp.4} allows us to make the above categories into $\lSm$-premotivic $\infty$-categories over $\lSch$.
Furthermore, there is a commutative diagram of $\lSm$-premotivic $\infty$-categories over $\lSch$
\begin{equation}
\begin{tikzcd}[column sep=tiny, row sep=small]
&&
&
\logSH_{\tau}^\eff\ar[rr,"\Sigma_{\Gmm}^\infty"]\ar[dd,"L_{\ver}"']\ar[rd,"\Lambda^{(-)}"']&
&
\logSH_{\tau}\ar[dd,"L_{\ver}",near end]\ar[rd,"\Lambda^{(-)}"]
\\
&
&
&
&
\logDA^\eff_{\tau}(-,\Lambda)\ar[rr,"\Sigma_{\Gmm}^\infty",near start,crossing over]&
&
\logDA_{\tau}(-,\Lambda)\ar[dd,"L_{\ver}"]
\\
\rH_{\tau}\ar[rr,"(-)_+"]&&
\rH_{*,\tau}\ar[r,"\Sigma_{S^1}^\infty"]&
\SH_{\tau}^\eff\ar[rr,"\Sigma_{\G_m}^\infty",near start]\ar[rd,"\Lambda^{(-)}"']&
&
\SH_{\tau}^\eff\ar[rd,"\Lambda^{(-)}"]
\\
&
&
&
&
\DA_{\tau}^\eff(-,\Lambda)\ar[rr,"\Sigma_{\G_m}^\infty"]\ar[uu,leftarrow,"L_{\ver}"',crossing over,near start]&
&
\DA_{\tau}(-,\Lambda).
\end{tikzcd}
\end{equation}
\end{rmk}

\begin{prop}
\label{comp.8}
For $B\in \Sch$ and a field $k$,
there are equivalences of $\infty$-categories
\begin{gather*}
\SH_{\ul{\tau}}^\eff(B)
\simeq
(\A^1)^{-1}\logSH_{\tau}^\eff(B),
\;
\SH_{\ul{\tau}}(B)
\simeq
(\A^1)^{-1}\logSH_{\tau}(B),
\\
\DA_{\ul{\tau}}^\eff(B,\Lambda)
\simeq
\DA_{\tau}^\eff(B,\Lambda)
\simeq
(\A^1)^{-1}\logDA_{\tau}^\eff(B,\Lambda),
\\
\DA_{\ul{\tau}}(B,\Lambda)
\simeq
\DA_{\tau}(B,\Lambda)
\simeq
(\A^1)^{-1}\logDA_{\tau}(B,\Lambda),
\\
\DM_{\ul{\tau}}^\eff(k,\Lambda)
\simeq
\DM_{\tau}^\eff(k,\Lambda)
\simeq
(\A^1)^{-1}\logDM_{\tau}^\eff(k,\Lambda),
\\
\DM_{\ul{\tau}}(k,\Lambda)
\simeq
\DM_{\tau}(k,\Lambda)
\simeq
(\A^1)^{-1}\logDM_{\tau}(k,\Lambda).
\end{gather*}
\end{prop}
\begin{proof}
This is immediate from Proposition \ref{2.3} and Remark \ref{comp.7}.
\end{proof}

The equivalence $\DMeff(k,\Lambda)
\simeq
(\A^1)^{-1}\logDMeff(k,\Lambda)$ was proven in \cite[Theorem 8.2.16]{logDM} assuming resolution of singularities.

\begin{rmk}
Proposition \ref{comp.8} can be also written as
\[
(\A^1)^{-1} \logSH_{\tau}(B)
\simeq
(\ver)^{-1} \logSH_{\tau}(B)
\]
and similarly for the others.
However,
in the case of $\tau=\sNis$,
we do not expect an equivalence
\[
(\A^1)^{-1} \logSH(S)
\simeq
(\ver)^{-1} \logSH(S)
\]
when $S$ is an fs log scheme with a nontrivial log structure,
see Remark \ref{local.14} for an explanation.
On the other hand, we expect that there is an equivalence of $\infty$-categories
\[
(\A^1)^{-1}\logSH_{\ketale}(S)
\simeq
\SH_{\ketale}(S)
\]
at least if $S$ has a chart $\N$.
\end{rmk}

\begin{prop}
\label{comp.6}
For $B\in \Sch$,
let $f\colon \A_{\N,B}\to B$ be the projection, and let $j\colon \G_{m,B} \to \A_{\N,B}$ be the obvious open immersion.
Then the natural transformation
\[
f^*\xrightarrow{ad} j_*j^* f^*
\]
is an isomorphism for the $\lSm$-premotivic categories in \textup{Definition \ref{comp.3}}.
\end{prop}
\begin{proof}
Let $X\in \lSm/\A_{\N,B}$.
Then $X$ is exact over $\A_{\N,B}$ by \cite[Proposition III.2.5.2(1), III.2.5.3(3)]{Ogu}.
Hence the open immersion
\[
X\times_{\A_{\N}}\G_m \to X
\]
is in $\ver_{B}$ by Proposition \ref{vert.9}(6).
It follows that the morphism
\[
f_\sharp j_\sharp j^*M(X)
\to
f_\sharp M(X) 
\]
is an isomorphism.
This implies that the natural transformation $f_\sharp j_\sharp j^* \xrightarrow{ad'} f_\sharp$ is an isomorphism.
By adjunction, we get the desired natural transformation.
\end{proof}

\begin{rmk}
\label{comp.9}
In this paper,
we mainly work with sheaves instead of hypersheaves for a streamlined exposition.
To work with hypersheaves,
for $S\in \lSch$ and an $\infty$-category $\cV$,
replace the $\infty$-category of $\dtau$-sheaves
\[
\infShv_{\dtau}(\lSm/S,\cV)
\]
by the $\infty$-category of $\divi$-local $\tau$-hypersheaves
\[
(\divi)^{-1}\infShv_\tau^\wedge(\lSm/S,\cV).
\]
For example,
we can define $\rH_{\tau}^\wedge(S)$ as Definition \ref{comp.10} below,
and this agrees with the classical one $\rH_{\ul{\tau}}^\wedge(S)$ if $S$ has a trivial log structure as Remark \ref{comp.7}.
\end{rmk}

\begin{df}
\label{comp.10}
For $S\in \lSch$, we define
\begin{gather*}
\rH_{\tau}^\wedge(S)
:=
(\A^1\cup \ver\cup \divi)^{-1}\infShv_{\tau}(\lSm/S,\Spc),
\\
\rH_{*,\tau}^\wedge(S)
:=
(\A^1\cup \ver\cup \divi)^{-1}\infShv_{\tau}(\lSm/S,\Spc_*),
\\
\SH_{\tau}^{\wedge,\eff}(S)
:=
(\A^1\cup \ver\cup \divi)^{-1}\infShv_{\tau}(\lSm/S,\Spt),
\\
\DA_{\tau}^{\wedge,\eff}(S,\Lambda)
:=
(\A^1\cup \ver\cup \divi)^{-1}\infShv_{\tau}(\lSm/S,\DLambda),
\\
\SH_{\tau}^\wedge(S)
:=
\Stab_{\G_m}(\SH_{\tau}^{\wedge,\eff}(S)),
\;
\DA_{\tau}^\wedge(S,\Lambda)
:=
\Stab_{\G_m}(\DA_{\tau}^{\wedge,\eff}(S,\Lambda)).
\end{gather*}
\end{df}

\begin{rmk}
By \cite[Proposition 3.3.28]{logDM} and \cite[Remark 2.6.3]{logSH},
for $S\in \Sch$,
$\rH(S)$ agrees with $\rH_\sNis^\wedge(S)$.
The same holds for the other five $\infty$-categories too.
\end{rmk}

\section{Localization property}

In this section, we prove a logarithmic analogue of the localization property of Morel-Voevodsky \cite[Theorem 2.21, p.\ 114]{MV}.
See also \cite[Th\'er\`eme 4.5.36]{Ayo072} for the generalization to the setting of presheaves with values in certain model categories.

The content of \S \ref{preliminary} and \ref{study} is to adopt their proof in our setting.
Then we show that inverting $\A^1$ and working with the strict Nisnevich topology are enough to deduce the localization property, see Theorem \ref{loc.8}.
Based on this, we prove the localization property for our categories $\rH_{\tau,*}$, $\SH_{\tau}^\eff$, $\SH_{\tau}$, $\DA_{\tau}^\eff(-,\Lambda)$, and $\DA_{\tau}(-,\Lambda)$.
In \S \ref{consequences}, we show that the localization property implies the Grothendieck six-functor formalism for strict morphisms.

\subsection{Preliminary lemmas}
\label{preliminary}
Throughout this subsection, $\cU$ is one of $\Spc$, $\Spc_*$, and $\Spt$, and $\DLambda$.
Recall that $\tau$ is one of $\sNis$, $\setale$, and $\ketale$,
and see \S \ref{topology} for $\ul{\tau}$ and $\dtau$.

\begin{df}
A \emph{henselian local} (resp.\ \emph{strictly local}) fs log scheme $X$ is a saturated log scheme such that $\ul{X}$ is a henselian (resp.\ strictly) local scheme.
A \emph{henselization} (resp.\ \emph{strict henselization}) of an fs log scheme $Y$ is the fiber product $Y\times_{\ul{Y}}U$ for some henselization (resp.\ strict henselization) $U$ of $\ul{Y}$.

Recall from \cite[Point 6 in 2.8]{Nak} that a \emph{log strictly local} saturated log scheme $X$ is a saturated log scheme such that $\ul{X}$ is strictly local and the multiplication map $n\colon \cM(X)\to \cM(X)$ is surjective for every integer $n$ invertible in $\ul{X}$. The \emph{log strict henselization} of an fs log scheme $Y$ at a log geometric point $y_{(\mathrm{log})}$ in the sense of \cite[Definition 2.5]{Nak} is the spectrum of the log ring
\[
(\colim_U \Gamma(U,\cO_U),\colim_U(U,\cM_U)),
\]
where $U$ runs over the category of $Y$-morphisms $y_{(\mathrm{log})}\to U$ such that $U$ is Kummer \'etale over $Y$.
\end{df}

\begin{lem}
\label{loc.1}
Let $f\colon X\to S$ be a strict finite morphism in $\lSch$.
Then
\begin{equation}
\label{loc.1.2}
f_*\colon \infPsh(\lSm/X,\cU)\to \infPsh(\lSm/S,\cU)
\end{equation}
sends $\tau$-local equivalences to $\tau$-local equivalences.
\end{lem}
\begin{proof}
We argue as in \cite[Proposition 1.27, p.\ 105]{MV}.
Assume $\tau=\sNis$ (resp.\ $\setale$, resp.\ $\ketale$).
Let $\cF\to \cG$ be a $\tau$-local equivalence in $\infPsh(\lSm/X,\cU)$.
For every henselization (resp.\ strict henselization, resp.\ log strict henselization) $U$ of an fs log scheme in $\lSm/S$, we need to show that the induced morphism
\[
f_*\cF(U)\to f_*\cG(U)
\]
is an isomorphism, i.e., the morphism $\cF(U\times_S X)\to \cG(U\times_S X)$ is an isomorphism.
This follows from the fact that $U\times_S X$ is a disjoint union of henselian (resp.\ strict, resp.\ log strict) local fs log schemes since $U\times_S X\to U$ is finite.
Here, for the log strictly local case,
use the strictness of $f$ to check the condition of log strictly local fs log schemes.
\end{proof}

Suppose that $f$ is as above.
Since \eqref{loc.1.2} preserves colimits, the functor
\begin{equation}
\label{loc.1.1}
f_*\colon \infShv_{\tau}(\lSm/X,\cU)\to \infShv_{\tau}(\lSm/S,\cU)
\end{equation}
preserves colimits by Lemma \ref{loc.1}.

For any morphism $p\colon Y\to X$ in $\lSm$, recall that we use the notation $M(Y)$ for the object $p_\sharp \unit$ in an $\lSm$-premotivic $\infty$-category.

\begin{lem}
\label{loc.2}
Suppose that $f$ is as above.
Then \eqref{loc.1.1} sends $\A^1$-local equivalences to $\A^1$-local equivalences.
\end{lem}
\begin{proof}
Since \eqref{loc.1.1} preserves colimits,
we can argue as in \cite[Th\'eorem 4.5.34]{Ayo072}.
\end{proof}

Suppose that $f$ is as above.
Since \eqref{loc.1.1} preserves colimits, the functor
\begin{equation}
\label{loc.2.1}
f_*\colon (\A^1)^{-1}\infShv_{\tau}(\lSm/X,\cU)\to (\A^1)^{-1}\infShv_{\tau}(\lSm/S,\cU)
\end{equation}
preserves colimits by Lemma \ref{loc.2}.

\begin{lem}
\label{loc.3}
Let $f\colon X\to S$ be a $\tau$-covering in $\lSch$.
Then
\[
f^*\colon \infPsh(\lSm/S,\cU)
\to
\infPsh(\lSm/X,\cU)
\]
detects $(\tau\cup \A^1)$-local equivalences.
\end{lem}
\begin{proof}
Let $\alpha\colon \cF\to \cG$ be a morphism in $\infPsh(\lSm/S,\cU)$ such that $f^*(\alpha)$ is an $(\tau\cup \A^1)$-local equivalence,
and let $\sX$ be the \v{C}ech nerve of $f$.
For every integer $n\geq 0$, the projection $p_n\colon \sX_n\to S$ factors through $f$.
It follows that $p_n^*(\alpha)$ is an $(\tau\cup \A^1)$-local equivalence.
Apply $p_{n\sharp}$ to obtain an $(\tau\cup \A^1)$-local equivalence
\[
\cF\otimes M(\sX_n)
\simeq
\cG\otimes M(\sX_n).
\]
Hence we obtain $(\tau\cup \A^1)$-local equivalences
\[
\cF
\simeq
\cF\otimes \colimit_{n\in \Delta}M(\sX_n)
\simeq
\cG\otimes \colimit_{n\in \Delta}M(\sX_n)
\simeq
\cG,
\]
where $\Delta$ denotes the simplex category.
\end{proof}

\begin{lem}
\label{loc.12}
Let
\[
\begin{tikzcd}
&Y\ar[d,"f"]
\\
Z\ar[ru,"s"]\ar[r,"i"]&
X
\end{tikzcd}
\]
be a commutative triangle of fs log schemes such that $i$ is a strict closed immersion, $f$ is log \'etale, and $\ul{X}$ is henselian.
Then there exists a unique section $X\to Y$ of $f$ extending $s$.
\end{lem}
\begin{proof}
The graph morphism $\Gamma_s\colon Z\to Z\times_X Y$ is a section of the projection $Z\times_X Y\to Z$.
Hence $\Gamma_s$ is an open immersion by \cite[Lemma A.11.2]{logDM}.
Since $i$ is a strict closed immersion, so is the projection $Z\times_X Y\to Y$.
It follows that $s$ is a strict immersion.
There exists a maximal open subscheme $U$ of $Y$ that is strict over $X$ by \cite[Lemma A.4.4]{logDM}.
Since $s$ is a strict immersion, $U$ contains the image of $s$.

If $X\to Y$ is a section of $f$, then it is an open immersion by \cite[Lemma A.11.2]{logDM} again.
Hence we need to show the claim for the commutative triangle
\[
\begin{tikzcd}
&U\ar[d]
\\
Z\ar[ru]\ar[r,"i"]&
X.
\end{tikzcd}
\]
In other words, we may assume that $f$ is strict.
In this case, $f$ is strict \'etale.
Use \cite[Th\'eor\`eme IV.18.5.11(b)]{EGA} to conclude.
\end{proof}

\begin{lem}
\label{loc.21}
Let $f\colon X\to S$ be a morphism of fs log schemes.
Assume that $S$ admits a chart $P$.
Then Zariski locally on $X$, $f$ admits a chart $\eta\colon P\to R$ such that $\eta$ is injective and the cokernel of $\eta^\gp$ is torsion free.
\end{lem}
\begin{proof}
We may assume that $f$ admits a chart $P\to Q$.
Let $\beta\colon Q\to \Gamma(X,\cM_X)$ be the chart for $X$, and let $x$ be a point of $X$.
By \cite[Proposition II.2.3.7]{Ogu}, we may assume that there exists a chart $\beta'\colon Q'\to \Gamma(X,\cM_X)$ neat at $x$.
Note that $Q'$ is sharp.
Owing to \cite[Proposition II.2.3.9]{Ogu}, there exists maps $\kappa \colon Q\to Q'$ and $\gamma \colon Q \to \Gamma(X,\cM_X^*)$ such that $\beta=\beta'\circ \kappa+\gamma$.

Now let $\eta\colon P\to P^\gp \oplus Q'$ be the map given by $p\mapsto (p,\kappa(\theta(p)))$, and let $\delta\colon P^\gp \oplus Q'\to \Gamma(X,\cM_X)$ be the chart given by $(p,q)\mapsto \gamma^\gp( \theta^\gp(p)) + \beta'(q)$.
Then we have
\[
\delta(\eta(p))
=
\gamma(\theta(p)) + \beta'(\kappa(\theta(p)))
=
\beta(\theta(p))
\]
for $p\in P$, which implies that $\delta$ is a chart.
Since $Q'$ is sharp and saturated, $Q'^\gp$ is torsion free.
It follows that the cokernel of $\eta^\gp$ is torsion free.
\end{proof}

\begin{lem}
\label{loc.14}
Let $i\colon Z\to S$ be a strict closed immersion and $V\in \lSm/Z$.
Then Zariski locally on $V$,
there exists $X\in \lSm/S$ such that $X\times_S Z\simeq V$.
\end{lem}
\begin{proof}
The question is Zariski local on $S$ and $V$, so we may assume that $V\to S$ admits a chart $\theta \colon P\to Q$ such that $\theta$ is injective and the cokernel of $\theta^\gp$ is torsion free by Lemma \ref{loc.21}.
We may also assume that $\Spec(R):=Z$, $\Spec(R'):=S$, and $V$ are affine.

Let $v$ be a point of $V$.
There exists a strict closed immersion $V\to W:=Z\times_{\A_P} \A_Q \times \A^n$ over $Z$ for some integer $n\geq 0$.
Observe that $W$ is log smooth over $Z$ by \cite[Theorem IV.3.1.8]{Ogu}.
We set $\Spec(A):=W$, and let $I$ be the ideal of $A$ such that $V\simeq \Spec(A/I)$.
By \cite[Theorem IV.3.2.2]{Ogu}, the map
\begin{equation}
\label{loc.14.1}
d\colon I/I^2 \otimes k(v)\to \Omega_{A/R}^1\otimes k(v)
\end{equation}
is injective and admits a retraction.
Use Nakayama's lemma to choose generators $a_1,\ldots,a_n$ of $I$ such that $da_1,\ldots,da_n$ are linearly independent in $\Omega_{A/R}^1\otimes k(v)$.

We set $\Spec(A'):=Y:=S\times_{\A_P}\A_Q \times \A^n$.
Choose any lifts $a_1',\ldots,a_n'$ of $a_1,\ldots,a_n$ to $A'$, and let $I'$ be the ideal of $A'$ generated by $a_1',\ldots,a_n'$.
We also set $X:=\Spec(A'/I')$.
The map
\[
d\colon I'/I'^2\otimes k(v)
\to
\Omega_{A'/R'}^1\otimes k(v)
\]
is isomorphic to \eqref{loc.14.1}, which is injective and admits a retraction.
Hence the map $d\colon I'/ I'^2\to \Omega_{A'/R'}^1$ is injective and admits a retraction on an open neighborhood $U$ of $v$ in $Y$.
By \cite[Theorem IV.3.1.8]{Ogu}, $Y$ is log smooth over $S$.
Together with \cite[Theorem IV.3.2.2]{Ogu}, we deduce that $U\times_Y X$ is log smooth over $S$.
\end{proof}

\begin{lem}
\label{loc.7}
Let $f\colon \cG\to \cF$ be a morphism in $\Psh(\lSm/S)$.
Assume that for every morphism $u\colon X\to \cF$ in $\Psh(\lSm/S)$ with $X\in \lSm/S$, the induced morphism
\begin{equation}
\label{loc.7.1}
p^*\cG\times_{p^*\cF}X
\to
X
\end{equation}
in $\Psh(\lSm/X)$ is an $(\tau\cup\A^1)$-local equivalence in $\infPsh(\lSm/S,\Spc)$.
Here, the morphism $p\colon X\to S$ is the structure morphism, and the morphism $X\to p^*\cF$ used in the fiber product is a right adjoint of $p_\sharp X\simeq X\xrightarrow{u}\cF$.
Then $f$ is an $(\tau\cup \A^1)$-local equivalence in $\infPsh(\lSm/S,\Spc)$.
\end{lem}
\begin{proof}
Argue as in \cite[Corollaire 4.5.40]{Ayo072}.
\end{proof}

\subsection{Study of \texorpdfstring{$\nu_{X,s}$}{PsiXs}}
\label{study}
Throughout this subsection, we fix a commutative triangle
\[
\begin{tikzcd}
&X\ar[d,"f"]
\\
Z\ar[ru,"s"]\ar[r,"i"]&
S
\end{tikzcd}
\]
in $\lSch$ such that $i$ is a strict closed immersion and $f$ is log smooth.
We imitate the proof of \cite[Proposition 4.5.42]{Ayo072} in this subsection.
Let $\nu_{X,s}$ be the presheaf on $\lSm/S$ given by
\[
\nu_{X,s}(Y)
:=
\left\{
\begin{array}{ll}
\Hom_S(Y,X)\times_{\Hom_Z(Y\times_S Z,Y\times_S X)}*&\text{if }Y\times_S Z\neq \emptyset
\\
* &\text{if }Y\times_S Z= \emptyset,
\end{array}
\right.
\]
for $Y\in \lSm/S$, where the morphism $*\to \Hom_Z(Y\times_S Z,Y\times_S X)$ appeared in the formulation is given by the composite
\[
Y\times_S Z \to Z \xrightarrow{\Gamma_s} X\times_S Z,
\]
and the first (resp.\ second) morphism is the projection (resp.\ graph morphism).

\begin{lem}
\label{loc.5}
Let
\[
\begin{tikzcd}
&X'\arrow[d,"g"]\\
Z\arrow[ru,"s'"]\arrow[r,"s"]&X
\end{tikzcd}
\]
be a commutative triangle such that $g$ is log \'etale.
Then the evident morphism
\[
\nu_{X',s'}\rightarrow \nu_{X,s}
\]
becomes an isomorphism after strict Nisnevich sheafifications. 
\end{lem}
\begin{proof}
It suffices to construct the inverse of
\begin{equation}
\label{loc.5.1}
\nu_{X',s'}(Y)\rightarrow \nu_{X,s}(Y)
\end{equation}
for any henselization $Y$ of an fs log scheme in $\lSm/S$.
Argue as in \cite[\'{E}tape 2 in Proposition 4.5.42]{Ayo072},
but use Lemma \ref{loc.12} also.
\end{proof}

\begin{lem}
\label{loc.6}
The morphism
\[
\nu_{X,s}\to S
\]
is an $(\sNis\cup \A^1)$-local equivalence in $\infPsh(\lSm/S,\Spc)$.
\end{lem}
\begin{proof}
(I) \emph{Locality on $S$}.
Let $\{p_i\colon S_i\to S\}_{i\in I}$ be a strict Nisnevich cover.
By Lemma \ref{loc.3}, it suffices to show that $p_i^*\nu_{X,s}\to S_i$ is an $(\sNis\cup \A^1)$-local equivalence.
There is a canonical isomorphism $p_i^*\nu_{X,s}\simeq \nu_{X_i,s_i}$ if $X_i:=X\times_S S_i$ and $s_i$ is the pullback of $s$.
Hence the question is strict Nisnevich local.
In particular, we may assume that $S$ has a chart.

(II) \emph{Locality on $X$}.
Let $\{X_i\to X\}_{i\in I}$ be a Zariski cover.
We set $Z_i:=Z\times_X X_i$ and $S_i:=S-(Z-Z_i)$.
Then $Z_i\simeq Z\times_S S_i$.
We have the induced commutative diagram
\begin{equation}
\label{loc.6.2}
\begin{tikzcd}
Z_i\ar[d,"t_i''"']\ar[rd,"t_i'"']\ar[rrd,"t_i"]
\\
X_i\times_S S_i\ar[r]&
X\times_S S_i\ar[r]&
S_i.
\end{tikzcd}
\end{equation}
We reduce to the case
\[
(S,X,Z,s)=(S_i,X\times_S S_i,Z_i,t_i')
\]
by (I).
Furthermore, we reduce to the case
\begin{equation}
\label{loc.6.1}
(S,X,Z,s)=(S_i,X_i\times_S S_i,Z_i,t_i'')
\end{equation}
by Lemma \ref{loc.5}.

(III) Choose a Zariski cover $\{X_i\to X\}_{i\in I}$ such that each composition $X_i\to S$ has a chart.
The same is true for the projection $X_i\times_S S_i\to S_i$ in \eqref{loc.6.2}.
By the reduction \eqref{loc.6.1}, we may assume that $Z\to X\to S$ has a chart $P\xrightarrow{\theta} Q\xrightarrow{\eta} P'$ such that the induced morphism $X\to S\times_{\A_P} \A_Q$ is strict \'etale, $\theta$ is injective, and the cokernel of $\theta^{\gp}$ is finite and invertible in $\cO_X$.

(IV)
By \cite[Proposition I.4.2.17]{Ogu}, $\eta$ admits a factorization
\[
Q\xrightarrow{\eta'} Q'\xrightarrow{\eta''} P'
\]
such that $\eta'^\gp$ is an isomorphism and $\eta''$ is exact.
We set $X':=X\times_{\A_Q}\A_{Q'}$.
The projection $g\colon X'\to X$ is log \'etale.
Since $\overline{\eta''}$ is surjective, \cite[Proposition I.4.2.1(3),(5)]{Ogu} implies that $\overline{\eta''}$ is an isomorphism.
It follows that we have an induced commutative triangle
\[
\begin{tikzcd}
&X'\arrow[d,"g"]\\
Z\arrow[r,"s"]\arrow[ru,"s'"]&X
\end{tikzcd}
\]
in $\lSch$ such that $s'$ is strict closed immersion.
Together with \cite[Lemma A.11.2]{logDM}, we can choose a maximal open subscheme $U$ of $X'$ that is strict over $S$.
Since $i$ is strict, $U$ contains $Z$.
Hence there is a commutative triangle
\[
\begin{tikzcd}
&U\arrow[d]\\
Z\arrow[r,"s"]\arrow[ru]&X
\end{tikzcd}
\]
such that the vertical morphism is strict \'etale.
By Lemma \ref{loc.5}, we can replace $X$ by $U$.
Hence we reduce to the case when $X$ is strict over $S$.
In this case, $X$ is strict smooth over $S$.

(V) By \cite[Corollaire IV.17.12.2(d)]{EGA}, we can choose a Zariski cover $\{X_i\to X\}_{i\in I}$ such that each $X_i\to S$ has a factorization
\[
X_i\xrightarrow{u_i} \A_S^n \to S
\]
such that $u_i$ is strict \'etale, the second morphism is the projection,
and the induced morphism $Z\times_X X_i\to \A_S^n$ factors through the zero section $S\to \A_S^n$.
The projection $X_i\times_S S_i \to S_i$ in \eqref{loc.6.2} has a factorization
\[
X_i\times_S S_i \to \A_{S_i}^n \to S_i.
\]
By the reduction \eqref{loc.6.1}, we reduce to the case when $f$ has a factorization
\[
X\xrightarrow{u}\A_S^n\to S
\]
such that $u$ is strict \'etale, the second morphism is the projection, and $us\colon Z\to \A_S^n$ is the zero section.

(VI)
By Lemma \ref{loc.5} again, we may assume $(X,s)=(\mathbb{A}_S^n,s_0)$, where $s_0\colon Z\rightarrow \mathbb{A}_S^n$ is the zero section.
We have a morphism
\begin{equation}
\label{loc.6.3}
\nu_{\mathbb{A}_S^n,0}\times\mathbb{A}_S^1
\to
\nu_{\mathbb{A}_S^n,0}
\end{equation}
of presheaves that maps $(f,t)\in \nu_{\mathbb{A}_S^n,0}(Y)\times\mathbb{A}_S^1(Y)$ for $Y\in \lSm/S$ to the composite
\[
\begin{tikzcd}
Y\arrow{r}&\mathbb{A}_S^n
\times_S\mathbb{A}_S^1\arrow{rrrr}{(x_1,\ldots,x_n,t)\mapsto
(tx_1,\ldots,tx_n)}&&&&\mathbb{A}_S^n
\end{tikzcd}
\]
Use this $\A^1$-homotopy to complete the proof.
\end{proof}

\subsection{Proof of the localization property}
\label{proof}

In this subsection, we fix a strict closed immersion $i\colon Z\to S$ in $\lSch$, and let $j\colon U \to S$ be its open complement.
Recall that $\tau$ is one of $\sNis$, $\setale$, and $\ketale$,
and see \S \ref{topology} for $\ul{\tau}$ and $\dtau$.

\begin{thm}
\label{loc.8}
For every $\cF\in (\A^1)^{-1}\infShv_{\tau}(\lSm/S,\Spc)$, the commutative square
\begin{equation}
\label{loc.8.1}
\begin{tikzcd}
j_\sharp j^* \cF\ar[d]\ar[r]&
\cF\ar[d]
\\
U\ar[r]&
i_*i^*\cF
\end{tikzcd}
\end{equation}
is cocartesian in $(\A^1)^{-1}\infShv_{\tau}(\lSm/S,\Spc)$, where the lower horizontal morphism is given by the unique object of $\Map(U,i_*i^*\cF)\simeq \Map(\emptyset,i^*\cF)$.
\end{thm}
\begin{proof}
Since $i_*$ preserves colimits,
we reduce to the case when $\cF=X$ for some $X\in \lSm/S$.

In this case, we need to show that the morphism
\[
X\amalg_{X\times_S U} U \to i_*(X\times_S Z)
\]
is an $(\tau\cup \A^1)$-local equivalence.
Owing to Lemma \ref{loc.7}, it suffices to show that for any morphism $p\colon V\to S$ in $\lSm/S$ with a morphism $s\colon V\to i_*(X\times_S Z)$, the induced morphism
\[
\nu_{X,V,s}:=p^*(X\amalg_{X\times_S U} U)\times_{p^*i_*(X\times_S Z)}V \to V
\]
is an $(\tau\cup \A^1)$-local equivalence.
Argue as in \cite[Paragraphs in front of Proposition 4.5.42]{Ayo072},
but use Lemma \ref{loc.6} also.
\end{proof}

\begin{lem}
\label{loc.20}
For $X\in \lSm/S$, there is a cofiber sequence
\begin{equation}
\label{loc.20.1}
U_+
\to
(i_*(X\times_S Z))_+
\to
i_*((X\times_S Z)_+)
\end{equation}
in $(\A^1)^{-1}\infShv_{\tau}(\lSm/S,\Spc_*)$.
We have similar results for $\Spt$ and $\DLambda$.
\end{lem}
\begin{proof}
We focus on the case of $\Spc_*$ since the proofs are similar.

Let $t_\emptyset$ be the smallest topology on $\lSm/S$ containing the empty covering of $\emptyset$.
A presheaf $\cF$ of sets is a $t_\emptyset$-sheaf if and only if $\cF(\emptyset)=*$.
Using an argument similar to the argument establishing the cocartesian square \cite[(152)]{Ayo072},
we obtain the cofiber sequence \eqref{loc.20.1} in $\infShv_{t_\emptyset}(\lSm/S,\Spc_*)$.
Argue as in the proof of Lemma \ref{loc.1} to see that the functor
\[
i_*\colon \infShv_{t_\emptyset}(\lSm/Z,\Spc_*)\to \infShv_{t_\emptyset}(\lSm/X,\Spc_*)
\]
preserves colimits.
Apply $L_{\A^1}L_{\tau}$ to this, and use Lemma \ref{loc.2} to obtain the cofiber sequence \eqref{loc.20.1} in $(\A^1)^{-1}\infShv_{\tau}(\lSm/S,\Spc_*)$.
\end{proof}

\begin{thm}
\label{loc.17}
For every $\cF\in (\A^1)^{-1}\infShv_{\tau}(\lSm/S,\Spc_*)$, the sequence
\begin{equation}
\label{loc.17.1}
\begin{tikzcd}
j_\sharp j^* \cF
\xrightarrow{ad'}
\cF
\xrightarrow{ad}
i_*i^*\cF
\end{tikzcd}
\end{equation}
is a cofiber sequence in $(\A^1)^{-1}\infShv_{\tau}(\lSm/S,\Spc_*)$.
Similar results hold for $\Spt$ and $\DLambda$ too.
\end{thm}
\begin{proof}
We focus on the case of $(\A^1)^{-1}\infShv_{\tau}(\lSm/S,\Spc_*)$ since the proofs are similar.
Since $i_*$ preserves colimits,
we reduce to the case when $\cF=X_+$ for some $X\in \lSm/S$.
Then apply the colimit preserving functor $(-)_+$ to \eqref{loc.8.1} and use Lemma \ref{loc.20} to conclude.
\end{proof}

\begin{lem}
\label{loc.9}
The functor
\[
i_*
\colon
(\A^1)^{-1}\infShv_{\tau}(Z,\Spc)
\to
(\A^1)^{-1}\infShv_{\tau}(S,\Spc)
\]
sends $(\dtau\cup\ver)$-local equivalences to $(\dtau\cup\ver)$-local equivalences.
Similar results hold for $\Spc_*$, $\Spt$, and $\DLambda$ too.
\end{lem}
\begin{proof}
We first treat the case of $\Spc$.
To show that $i_*$ preserves $\ver$-local equivalences, we need to show that
\begin{equation}
\label{loc.9.1}
i_*(Y-\partial_Z Y)\to i_*Y
\end{equation}
is a $\ver$-local equivalence for $Y\in \lSm/Z$.
This question is strict Nisnevich local on $Y$, so we may assume that there exists $X\in \lSm/S$ such that $X\times_S Z\simeq Y$ by Lemma \ref{loc.14}.

In this case, we have $i^*X\simeq Y$ and $i^*(X-\partial_S X)\simeq Y-\partial_Z Y$ by Proposition \ref{vert.5}.
Theorem \ref{loc.8} gives cocartesian squares
\[
\begin{tikzcd}
X\times_S U\ar[d]\ar[r]&
X\ar[d]
\\
U\ar[r]&
i_*i^*X,
\end{tikzcd}
\quad
\begin{tikzcd}
(X-\partial_S X)\times_S U\ar[d]\ar[r]&
X-\partial_S X\ar[d]
\\
U\ar[r]&
i_*i^*(X-\partial_S X).
\end{tikzcd}
\]
Compare these two to deduce that \eqref{loc.9.1} is a $\ver$-local equivalence.

Let $Y'\to Y$ be a dividing cover in $\lSm/Z$.
We need to show that
\begin{equation}
\label{loc.9.2}
i_*Y'\to i_*Y
\end{equation}
is a $\dtau$-local equivalence.
This question is $\tau$-local on $Y$ too, so we may assume that $Y$ admits a chart $P$.
By \cite[Proposition A.11.5]{logDM}, there exists a subdivision $\Sigma$ of $\Spec(P)$ such that 
\[
Y'':=Y'\times_{\A_P}\T_\Sigma
\simeq
Y\times_{\A_P}\T_\Sigma.
\]
We only need to show that the two induced morphisms $i_*Y''\to i_*Y$ and $i_*Y''\to i_*Y'$ are isomorphisms.
Replace $(Y,Y')$ by $(Y,Y'')$ and $(Y',Y'')$ to reduce to the case when $Y$ admits a morphism $Y\to \A_P$ with an fs monoid $P$ and $Y'=Y\times_{\A_P}\A_\Sigma$ for some subdivision $\Sigma$ of $\Spec(P)$.

We may further assume that there exists $X\in \lSm/S$ with $X\times_S Z\simeq Y$ by Lemma \ref{loc.14}.
We set $X':=X\times_{\A_P}\T_\Sigma$, then we have $Y'\simeq X'\times_S Z$.
Theorem \ref{loc.8} gives cocartesian squares
\[
\begin{tikzcd}
X\times_S U\ar[d]\ar[r]&
X\ar[d]
\\
U\ar[r]&
i_*i^*X,
\end{tikzcd}
\quad
\begin{tikzcd}
X'\times_S U\ar[d]\ar[r]&
X'\ar[d]
\\
U\ar[r]&
i_*i^*X'.
\end{tikzcd}
\]
Compare these two to deduce that \eqref{loc.9.2} is a $\dtau$-local equivalence.

For the cases of $\Spc_*$ and $\DLambda$, use Theorem \ref{loc.17} instead of Theorem \ref{loc.8}.
\end{proof}

Since \eqref{loc.2.1} preserves colimits, we have a colimit preserving functor
\begin{equation}
\label{loc.9.3}
i_*\colon \rH_{\tau}(Z)\to \rH_{\tau}(S)
\end{equation}
by Lemma \ref{loc.9}.

\begin{thm}
\label{loc.10}
For every  $\cF\in \rH_{\tau}(S)$, the square
\[
\begin{tikzcd}
j_\sharp j^*\cF\ar[d]\ar[r]&
\cF\ar[d]
\\
U\ar[r]&
i_*i^*\cF
\end{tikzcd}
\]
is cocartesian in $\rH_{\tau}(S)$.
\end{thm}
\begin{proof}
Apply $L_{\dtau\cup\boxx}$ to \eqref{loc.8.1}, and use Lemma \ref{loc.9}.
\end{proof}

\begin{thm}
\label{loc.11}
For every $\cF\in \rH_{\tau,*}(S)$, the sequence
\[
\begin{tikzcd}
j_\sharp j^* \cF
\xrightarrow{ad'}
\cF
\xrightarrow{ad}
i_*i^*\cF
\end{tikzcd}
\]
is a cofiber sequence in $\rH_{\tau,*}(S)$.
We have similar results for $\SH_{\tau}^\eff$ and $\DA_{\tau}^\eff(-,\Lambda)$ too.
\end{thm}
\begin{proof}
Apply $L_{\dtau\cup\boxx}$ to \eqref{loc.17.1}, and use Lemma \ref{loc.9}.
\end{proof}

\begin{lem}
\label{loc.18}
For every $\cF\in \SH_{\tau}^\eff(Z)$, the canonical morphism
\[
\Sigma_{\Gmm}i_*\cF
\to
i_*\Sigma_{\Gmm}\cF
\]
is an isomorphism.
A similar result holds for $\DA^\eff_{\tau}(-,\Lambda)$ too.
\end{lem}
\begin{proof}
We focus on the case of $\SH_{\tau}^\eff$ since the proofs are similar.
Since $i_*$ preserves colimits,
we reduce to the case when $\cF=\Sigma_{S^1}^\infty X_+$ for some $X\in \lSm/S$.

The question is Zariski local on $X$.
Hence by Lemma \ref{loc.14}, we reduce to the case when $X\simeq V\times_X Z$ for some $V\in \lSm/S$.
We set $\cG:=\Sigma_{S^1}^\infty V_+$ for simplicity of notation.

Consider the induced commutative diagram
\[
\begin{tikzcd}
\Sigma_{\Gmm}j_\sharp j^* \cG\ar[d]\ar[r]&
\Sigma_{\Gmm}\cG\ar[d,"\id"]\ar[r]&
\Sigma_{\Gmm}i_*i^*\cG\ar[d]
\\
j_\sharp j^* \Sigma_{\Gmm}\cG\ar[r]&
\Sigma_{\Gmm}\cG\ar[r]&
i_*i^*\Sigma_{\Gmm}\cG
\end{tikzcd}
\]
whose rows are cofiber sequences by Theorem \ref{loc.11}.
To conclude, observe that $j_\sharp$, $j^*$, and $i^*$ commute with $\Sigma_{\Gmm}$.
\end{proof}

\begin{lem}
\label{loc.19}
In $\PrL$, there are commutative squares
\[
\begin{tikzcd}
\SH_{\tau}^\eff(Z)\ar[d,"i_*"']\ar[r,"\Sigma_{\Gmm}^\infty"]&
\SH_{\tau}(Z)\ar[d,"i_*"]
\\
\SH_{\tau}^\eff(S)\ar[r,"\Sigma_{\Gmm}^\infty"]&
\SH_{\tau}(S),
\end{tikzcd}
\quad
\begin{tikzcd}
\DA_{\tau}^\eff(Z,\Lambda)\ar[d,"i_*"']\ar[r,"\Sigma_{\Gmm}^\infty"]&
\DA_{\tau}(Z,\Lambda)\ar[d,"i_*"]
\\
\DA_{\tau}^\eff(S,\Lambda)\ar[r,"\Sigma_{\Gmm}^\infty"]&
\DA_{\tau}(S,\Lambda).
\end{tikzcd}
\]
\end{lem}
\begin{proof}
We focus on the left square since the proofs are similar.
By Lemma \ref{loc.18}, we have a commutative diagram
\[
\begin{tikzcd}
\SH_{\tau}^\eff(Z)\ar[d,"i_*"']\ar[r,"\Sigma_{\Gmm}"]&
\SH_{\tau}^\eff(Z)\ar[d,"i_*"']\ar[r,"\Sigma_{\Gmm}"]&
\cdots
\\
\SH_{\tau}^\eff(S)\ar[r,"\Sigma_{\Gmm}"]&
\SH_{\tau}^\eff(Z)\ar[r,"\Sigma_{\Gmm}"]&
\cdots.
\end{tikzcd}
\]
Take colimits in $\PrL$ along the rows to obtain the desired diagram.
\end{proof}

\begin{thm}
\label{loc.16}
For every $\cF\in \SH_{\tau}(S)$, the sequence
\begin{equation}
\label{loc.16.1}
\begin{tikzcd}
j_\sharp j^* \cF
\xrightarrow{ad'}
\cF
\xrightarrow{ad}
i_*i^*\cF
\end{tikzcd}
\end{equation}
is a cofiber sequence in $\SH_{\tau}(S)$.
A similar result holds for $\DA_{\tau}(-,\Lambda)$ too.
\end{thm}
\begin{proof}
We focus on $\SH_{\tau}(S)$ since the proofs are similar.
As noted in Lemma \ref{loc.19}, $i_*$ preserves colimits.
Hence we reduce to the case when $\cF=\Sigma_{\P^1}^\infty X_+$ for some $X\in \lSm/S$.

Since $\Sigma_{\P^1}^\infty$ is a functor of premotivic $\infty$-categories, $\Sigma_{\P^1}^\infty$ commutes with $j_\sharp$, $j^*$, and $i^*$.
Together with Lemma \ref{loc.19}, we reduce to showing that the sequence
\[
\Sigma_{\P^1}^\infty j_\sharp j^* X_+
\to
\Sigma_{\P^1}^\infty  X_+
\to
\Sigma_{\P^1}^\infty i_*i^* X_+
\]
is a cofiber sequence.
This follows from Theorem \ref{loc.11}.
\end{proof}

\begin{prop}
For every $\cF\in \rH_{\tau,*}(Z)$, the morphism
\[
i^*i_*\cF \xrightarrow{ad'} \cF
\]
is an isomorphism.
Similar results hold for $\SH_{\tau}^\eff$, $\SH_{\tau}$, $\DA_{\tau}^\eff(-,\Lambda)$, and $\DA_{\tau}(-,\Lambda)$ too.
\end{prop}
\begin{proof}
We focus on the case of $\rH_{\tau,*}(Z)$ since the proofs are similar.
Since $i_*$ preserves colimits,
we reduce to the case when $\cF=X_+$ for some $X\in \lSm/Z$.
Then the question is Zariski local on $X$.
Hence by Lemma \ref{loc.14}, we reduce to the case when $X\simeq V\times_S Z$ for some $V\in \lSm/S$.
We need to show that the morphism
\[
i^*i_*i^*V_+\xrightarrow{ad'} i^*V_+
\]
is an isomorphism.
We can alternatively show that the morphism
\[
i^*V_+\xrightarrow{ad} i^*i_*i^*V_+
\]
is an isomorphism.
This is a consequence of Theorem \ref{loc.17}.
\end{proof}

\subsection{Consequences of the localization property}
\label{consequences}

Throughout this subsection, we fix a base $B\in \Sch$ and an $\sSm$-premotivic $\infty$-category $\sT$ over $\lFt/B$, where $\sSm$ denotes the class of strict smooth morphisms (i.e., the class of strict morphisms $f\colon X\to S$ such that $\ul{f}$ is smooth), and $\lFt$ denotes the class of morphisms of finite type in $\lSch$.
We assume the following conditions for $\sT$:
\begin{itemize}
\item (Localization property) For every strict closed immersion $i$ in $\lFt/B$ with its open complement $j$, the sequence
\[
j_\sharp j^*\xrightarrow{ad'} \id \xrightarrow{ad} i_*i^*
\]
is a cofiber sequence, and $i_*$ is fully faithful.
\item ($\A^1$-invariance) $p^*$ is fully faithful, where $p\colon X\times \A^1\to X$ is the projection with $X\in \lFt/B$.
Then 
\item ($\P^1$-stability) For $X\in \lFt/B$,
the functor
\[
(-)\otimes M(\P^1/1) \colon \sT(X)\to \sT(X)
\]
is an equivalence of $\infty$-categories, where $M(\P^1/1):=\cofib(M(\{1\})\to M(\P^1))$.
\end{itemize}
For example, $\SH_{\tau}$ and $\DA_{\tau}(-,\Lambda)$ satisfy these conditions.

\begin{df}
For $X\in \lFt/B$, recall from \cite[Definition 7.1.2]{logDM} that a \emph{vector bundle over $X$} is a strict morphism $\xi\colon \cE\to X$ such that $\ul{\xi}\colon \ul{\cE}\to \ul{X}$ is a vector bundle over $\ul{X}$.
In this case, we set
\[
MTh(\cE)
:=
\cofib(M(\cE-Z) \to M(\cE)),
\]
where $Z$ is the zero section of $\cE$.
\end{df}

\begin{thm}
\label{consq.1}
The premotivic $\infty$-category $\sT$ enjoys the following properties.
\begin{enumerate}
\item[\textup{(1)}]
There exists a covariant functor $f\mapsto f_!$ on the subcategory of $\lFt/B$ whose objects are the same as $\lFt/B$ and whose morphisms are the strict morphisms.
If $f$ is strict proper, then $f_!\simeq f_*$.
If $f$ is an open immersion, then $f_!\simeq f_\sharp$.
\item[\textup{(2)}]
Let $f$ be a strict smooth morphism in $\lFt/B$.
There exists a natural isomorphism
\[
f_\sharp
\xrightarrow{\simeq}
f_! (MTh(T_f)\otimes -),
\]
where $T_f$ is the tangent bundle of $f$.
Furthermore, $MTh(-T_f)$ is the $\otimes$-inverse of $MTh(T_f)$.
\item[\textup{(3)}]
Let
\[
\begin{tikzcd}
Y'\ar[d,"f'"']\ar[r,"g'"]&
Y\ar[d,"f"]
\\
X'\ar[r,"g"]&
X
\end{tikzcd}
\]
be a cartesian square in $\lFt/B$.
If $f$ and $g$ are strict, then there exists a natural isomorphism
\[
Ex\colon g^*f_! \xrightarrow{\simeq} f_!'g'^*.
\]
\item[\textup{(4)}]
Let $f$ be a strict morphism in $\lFt/B$.
Then there exists a natural isomorphism
\[
Ex\colon f_!(-) \otimes (-) \xrightarrow{\simeq} f_!((-)\otimes f^*(-)).
\]
\end{enumerate}
\end{thm}
\begin{proof}
For every $X\in \lFt/B$, the restriction of $\Ho(\sT)$ to $\sFt/X\simeq \Ft/\ul{X}$ is a motivic triangulated category in the sense of \cite[Definition 2.4.45]{CD19}, where $\Ft$ (resp.\ $\sFt$) denotes the class of morphisms (resp.\ strict morphisms) of finite type in $\Sch$ (resp.\ $\lSch$).
Note that the condition \cite[Remark 2.4.47(1)]{CD19} is automatic in our setting due to \cite{zbMATH02242854}.
Use \cite[Theorem 2.4.50]{CD19} for varying $X$ to show the desired claim for $\Ho(\sT)$.
We can promoted this to $\sT$ using the technique of Liu-Zheng \cite[\S 6.2]{LZ}, see e.g.\ Mann's thesis \cite{Mann} for the rigid analytic setting.
\end{proof}

Let us explain Theorem \ref{consq.1}(1) in more detail.
Let
\[
\begin{tikzcd}
U'\ar[r,"j'"]\ar[d,"p'"']&
X'\ar[d,"p"]
\\
U\ar[r,"j"]&
X
\end{tikzcd}
\]
be a cartesian square in $\lFt/B$ such that $j$ is an open immersion and $p$ is proper.
Then we have the natural transformation
\begin{equation}
\label{consq.1.1}
Ex\colon j_\sharp p_*' \to p_*j_\sharp'
\end{equation}
given by the composition
\[
Ex\colon j_\sharp p_*'
\xrightarrow{ad}
p_*p^*j_\sharp p_*'
\xrightarrow{Ex^{-1}}
p_*j_\sharp' p'^* p_*'
\xrightarrow{ad'}
p_*j_\sharp'.
\]
If $p$ is strict proper, then the composition
\[
j_\sharp p_*' \xrightarrow{\simeq} j_!p_!'\xrightarrow{\simeq} p_!j_!' \xrightarrow{\simeq} p_*j_\sharp'
\]
is given by \eqref{consq.1.1} according to \cite[\S 2.2.a]{CD19}.

\begin{rmk}
As a consequence of the localization property for schemes,
Ayoub proved an equivalence between motives and formal motives,
see \cite[Corollaire 1.4.29]{MR3381140} for the details.
Binda, Kato, and Vezzani \cite[Theorem 2.2]{2207.00369} proved its logarithmic analogue with a setting slightly different from ours.
\end{rmk}

\section{Cohomology of boundaries}

One of the ultimate goals of this series of papers is to establish the motivic six-functor formalism for fs log schemes.
Such a formalism would immediately compute the cohomology of the standard log point $\pt_{\N,B}$ in terms of the cohomology of schemes,
where $B\in \Sch$.
This suggests that computing the cohomology of $\pt_{\N,B}$ can be an important step toward the formalism.

The main result of this section is Theorem \ref{local.5}, which expresses
\begin{equation}
\label{boundary.0.1}
\map_{\SH(\A_{\N,B})}
(M(X)(d),M(\G_{m,B}))
\end{equation}
in terms of a hom spectrum in $\SH(\ul{X-\partial_{\A_{\N,B}}X})$ for $X\in \lSm/\A_{\N,B}$ and an integer $d$.
If $X=\A_{\N,B}$, then we obtain the computation of the cohomology of $\pt_{\N,B}$ using the localization property.
Allowing general $X$ is obviously more helpful than just having the computation for $X=\A_{\N,B}$.

The outline of the proof is as follows.
We will introduce the \emph{dividing vertical localization} $L_{\dver}$ in \S \ref{localization}.
We will have an explicit description of
\[
L_{\dver}L_{\divi}L_{\A^1}L_{\sNis}\Sigma_{\P^1}^\infty (\G_{m,B})_+.
\]
Using the result of \S \ref{key}, we will show that this is $\ver$-local.
This explicit description will allow us to conclude the proof of Theorem \ref{local.5}.

The reason for taking this strategy is that we do not know an explicit description of $L_{ver}$ unlike $L_{\dver}$.
The following illustrates the scheme of the proof of Theorem \ref{local.5}.
\[
\begin{tikzpicture}
\node[draw,align=center] at (0,0)
{Proposition \ref{divlocal.9}};
\node[draw,align=center] at (4,0)
{Proposition \ref{lemma.7}};
\node[draw,align=center] at (8,0)
{Lemmas \ref{lemma.17}--\ref{lemma.12}};
\node[draw,align=center] at (0,-1)
{Lemma \ref{local.2}};
\node[draw,align=center] at (4,-1)
{Lemma \ref{local.4}};
\node[draw,align=center] at (8,-1)
{Constructions \ref{divlocal.2}, \ref{divlocal.11}};
\node[draw,align=center] at (0,-2)
{Construction \ref{divlocal.7}};
\node[draw,align=center] at (4,-2)
{Lemma \ref{local.3}};
\node[draw,align=center] at (8,-2)
{Theorem \ref{local.5}};
\draw [-stealth](0,-0.35) -- (0,-0.72);
\draw [-stealth](4,-0.35) -- (4,-0.72);
\draw [-stealth](4,-1.72) -- (4,-1.3);
\draw [-stealth](6.25,0) -- (5.5,0);
\draw [-stealth](5.88,-1) -- (5.15,-1);
\draw [-stealth](1.16,-1) -- (2.83,-1);
\draw [-stealth](1.6,-2) -- (2.83,-2);
\draw [-stealth](1.13,-1.3) -- (2.83,-1.77);
\draw [-stealth](5.13,-1.3) -- (6.72,-1.77);
\end{tikzpicture}
\]

In \S \ref{cohomology}, we explain how we extend cohomology theories of schemes to fs log schemes.

\subsection{Key comparison}
\label{key}
Throughout this subsection, we fix $B\in \Sch$ and an open subscheme $C$ of $\G_{m,B}$.
We also fix a motivic $\infty$-category $\sT$, i.e., an $\Sm$-premotivic $\infty$-category over $\Sch/B$ whose homotopy category is a motivic triangulated category in the sense of \cite[Definition 2.4.45]{CD19}.
The goal of this subsection is to prove Proposition \ref{lemma.7}.
For a stable $\infty$-category $\cC$, let $\map_{\cC}(-,-)$ denote the hom spectrum.

For an fs monoid $P$,
\cite[Lemma I.6.7]{zbMATH07027475} yields a non-canonical decomposition $P\simeq P^*\oplus \ol{P}$.
Note that $\ol{P}^\gp$ is torsion free by \cite[Proposition I.1.3.5(2)]{Ogu}.
Let $P^\tor$ be the torsion subgroup of $P^*$.
We set $P^\tf:=P/P^\tor$.
Then we have a non-canonical decomposition
\begin{equation}
P
\simeq
P^\tf \oplus P^\tor.
\end{equation}

\begin{const}
\label{lemma.11}
Recall that $C$ is an open subscheme of $\G_{m,B}$.
For $Y\in \lSch/\A_{\N,B}$ and integer $d$, we set
\[
\Phi_d(Y)
:=
\map_{\sT(\ul{Y})}
(M(\ul{Y})(d),M(\ul{Y\times_{\A_{\N,B}}C}))
\in
\Ho(\Spt),
\]
which defines a functor $\Phi_d \colon (\lSch/\A_{\N,B})^{\op}\to \Ho(\Spt)$.
We also set
\begin{align*}
\Theta_d(Y)
:=
&\Phi_d(Y-\partial_{\A_{\N,B}}Y)
\\
=
&\map_{\sT(\ul{Y-\partial_{\A_{\N,B}}Y})}
(M(\ul{Y-\partial_{\A_{\N,B}}Y})(d),M((Y-\partial_{\A_{\N,B}}Y)\times_{\A_{\N,B}}C)).
\end{align*}
Here, we have removed the underline in $\ul{(Y-\partial_{\A_{\N,B}}Y)\times_{\A_{\N,B}}C}$ since $(Y-\partial_{\A_{\N,B}}Y)\times_{\A_{\N,B}}C$ has trivial log structure.
If $Y$ is vertical over $\A_{\N,B}$, then $\Theta_d(Y)\simeq \Phi_d(Y)$.
\end{const}

\begin{lem}
\label{lemma.10}
Let $f\colon Y'\to Y$ be a dividing cover in $\lSm/X$,
where $X\in \lSch$.
Let $V\to U\to Y$ be open immersions, and assume that for every point $v$ of $V$,
$\ol{\cM}_{V,v}^\gp$ has rank $\leq 1$.
Then for every integer $d$,
there is a canonical isomorphism
\[
\map_{\sT(\ul{U})}
(M(\ul{U})(d),M(\ul{V}))
\simeq
\map_{\sT(\ul{U\times_Y Y'})}
(M(\ul{U\times_Y Y'})(d),M(\ul{V\times_Y Y'})).
\]
\end{lem}
\begin{proof}
The projection $V\times_Y Y'\to V$ is a dividing cover,
and there exists no nontrivial subdivision of a $1$-dimensional fan.
Hence the assumption on $V$ implies $V\times_Y Y'\simeq V$.

Let $g\colon \ul{U\times_Y Y'}\to \ul{U}$ be the projection.
Consider the induced cartesian square
\[
\begin{tikzcd}
\ul{V\times_Y Y'}\ar[d,"\simeq"']\ar[r,"j'"]&
\ul{U\times_Y Y'}\ar[d,"g"]
\\
\ul{V}\ar[r,"j"]&
\ul{U}.
\end{tikzcd}
\]
We have isomorphisms
\[
g_*g^*M(\ul{V})
\simeq
g_*j_\sharp' \unit
\simeq
j_\sharp \unit
=
M(\ul{V}),
\]
where we use the \cite[Theorem 2.4.50(2)]{CD19} for the second isomorphism.
By adjunction, we obtain the desired isomorphism.
\end{proof}

\begin{const}
Let $f\colon Y'\to Y$ be a dividing cover in $\lSm/\A_{\N,B}$.
By Propositions \ref{vert.9}(1) and \ref{vert.8}, $(Y-\partial_{\A_{\N,B}}Y)\times_Y Y'$ is vertical over $\A_{\N,B}$.
In other words, we have the open immersion
\begin{equation}
\label{lemma.11.1}
(Y-\partial_{\A_{\N,B}}Y)\times_Y Y'
\hookrightarrow
Y'-\partial_{\A_{\N,B}} Y'.
\end{equation}
Since $(Y-\partial_{\A_{\N,B}} Y)\times_{\A_{\N,B}} C$ has trivial log structure,
Lemma \ref{lemma.10} yields
\[
\Phi_d(Y-\partial_{\A_{\N,B}}Y)
\simeq
\Phi_d((Y-\partial_{\A_{\N,B}}Y)\times_Y Y').
\]
Hence we obtain a morphism
\begin{equation}
\label{lemma.11.2}
\Theta_d(f)\colon \Theta_d(Y')\to \Theta_d(Y)
\end{equation}
given by the composition
\[
\Phi_d(Y'-\partial_{\A_{\N,B}}Y')
\to 
\Phi_d((Y-\partial_{\A_{\N,B}}Y)\times_Y Y')
\simeq 
\Phi_d(Y-\partial_{\A_{\N,B}}Y).
\]
If $g\colon Y''\to Y'$ is another dividing cover, then use the commutative diagram
\[
\begin{tikzcd}[column sep=small, row sep=small]
(Y-\partial_{\A_{\N,B}}Y)\times_Y Y''\ar[r,hookrightarrow]\ar[d]&
(Y'-\partial_{\A_{\N,B}} Y')\times_{Y'}Y''\ar[r,hookrightarrow]\ar[d]&
Y''
\\
(Y-\partial_{\A_{\N,B}}Y)\times_Y Y'\ar[r,hookrightarrow]\ar[d]&
Y'-\partial_{\A_{\N,B}} Y'
\\
Y-\partial_{\A_{\N,B}}Y
\end{tikzcd}
\]
to show $\Theta_d(g)\Theta_d(f)\simeq \Theta_d(gf)$.
We also have $\Theta_d(\id)=\id$.
Hence $\Theta_d$ is covariant for dividing covers.
On the other hand, $\Theta_d$ is contravariant for open immersions since if $U\to Y$ is an open immersion, then we have the induced open immersion $U-\partial_{\A_{\N,B}}U\to Y-\partial_{\A_{\N,B}}Y$.

It seems that $\Theta_d$ can be turned into functors with values in $\Spt$ instead of $\Ho(\Spt)$, but we do not need this.
\end{const}

\begin{prop}
\label{lemma.7}
Let $p\colon X'\to X$ be a dividing cover in $\lSm/\A_{\N,B}$, and let $d$ be an integer.
Then $\Theta_d(p)$ is an isomorphism in the following two cases:
\begin{enumerate}
\item[\textup{(i)}]
The structure morphism $X\to \A_{\N,B}$ admits a factorization $X\to \A_{Q,B}\xrightarrow{\A_{\theta,B}}\A_{\N,B}$ such that $X\to \A_{Q,B}$ is strict \'etale and $\theta\colon \N\to Q$ is a map of fs monoids.
\item[\textup{(ii)}]
$C=\G_{m,B}$.
\end{enumerate}
\end{prop}

Lemmas \ref{lemma.17}--\ref{lemma.12} are the steps of the proof.
We only write the case of $d=0$ for notational ease, and we set $\Phi:=\Phi_0$ and $\Theta:=\Theta_0$.
For general $d$, add $(d)$ to the appropriate places in the proof.
We finish the proof at the end of this section.

Let us recall some notation and terminology in log geometry.
For an fs monoid $P$ and its ideal $I$, let $\A_{(P,I)}$ be the strict closed subscheme of $\A_P$ whose underlying scheme is given by $\Spec(\Z[P]/(I))$.
We set
\[
\A_{(P,I),B}:=\A_{(P,I)}\times B,
\;
\pt_{\N}:=\A_{(\N,\N^+)},
\;
\pt_{\N,B}:=\pt_{\N}\times B.
\]

It is well-known that for a face $F$ of $P$,
we have $\ol{P_F}\simeq P/F$,
which can be shown using $F^\gp\oplus_F P\simeq P_F$.

A \emph{fan chart} of an fs log scheme $X$ is a strict morphism $X\to \T_{\Sigma}$, where $\Sigma$ is a fan.

For $X\in \lSch$, let $X_{\divi}$ be the full subcategory of $\lSch$ consisting of $Y$ such that $Y\to X$ is a dividing cover.

For a fan $\Sigma$, let $\Sigma_{\divi}$ be the category of subdivisions of $\Sigma$.

\begin{lem}
\label{lemma.17}
Let $\theta\colon \N\to P$ be a vertical map of fs monoids.
Assume that $\ol{P}\simeq \N^n$ for some  integer $n\geq 1$.
If $F$ is a face of $P$ with $F\neq P$, then we have an isomorphism
\[
\Phi(\A_{P_F,B})
\simeq
\Phi(\A_{P,B}).
\]
\end{lem}
\begin{proof}
Replace $(P,B)$ by $(\ol{P},\A_{P^*,B})$ to reduce to the case when $P\simeq \N^n$.
We proceed by induction on $n$.
The claim is trivial for $n=1$.
Assume $n>1$.
By induction, we only need to consider the case when the rank of $F^\gp$ is $1$.
Let $\{e_1,\ldots,e_n\}$ be the standard coordinate in $\Z^n$.
Without loss of generality, we may assume that $F$ is generated by $e_1$.

We express $\theta(1)$ as $(a_1,\ldots,a_n)$ with $a_1,\ldots,a_n >0$.
Let $L$ be the lattice for the fan $\Spec(P)$.
Then its dual $L^\vee$ is the lattice for $P$,
i.e., $L^\vee\simeq P^\gp$.
We set
\begin{gather*}
Q:=((a_1e_1+a_2e_2)\Q_{\geq 0}\oplus (-e_1)\Q_{\geq 0} \oplus e_3\Q_{\geq 0}\cdots \oplus e_n \Q_{\geq 0})\cap L^\vee,
\\
Q':=((a_1e_1+a_2e_2)\Q_{\geq 0}\oplus e_2\Q_{\geq 0}\oplus e_3\Q_{\geq 0}\cdots \oplus e_n \Q_{\geq 0})\cap L^\vee.
\end{gather*}
Said differently,
$Q$ (resp.\ $Q'$) is the intersection of the lattice $L^\vee$ for $P$ and the $\Q$-cone \cite[Definition I.2.3.1]{Ogu} generated by $a_1e_1+a_2e_2$, $-e_1$ (resp.\ $e_2$), $e_3$, $\ldots$, $e_n$.
We have the maps $\N\to Q,Q'$ induced by $\theta$.
The dual cones are
\begin{gather*}
\Spec(Q)=((-a_2e_1+a_1e_2)\Q_{\geq 0} \oplus e_2 \Q_{\geq 0} \oplus \cdots \oplus e_n \Q_{\geq 0})\cap L,
\\
\Spec(Q')=((-a_2e_1+a_1e_2)\Q_{\geq 0} \oplus e_1 \Q_{\geq 0}\oplus e_3\Q_{\geq 0}\oplus \cdots \oplus e_n \Q_{\geq 0})\cap L.
\end{gather*}
We have the face
\[
G:=((a_1e_1+a_2e_2)\Q_{\geq 0}\oplus 0\oplus e_3\Q_{\geq 0}\cdots \oplus e_n \Q_{\geq 0})\cap L^\vee\]
of $Q$ and $Q'$.

By \cite[Proposition I.2.2.1]{Ogu}, there exists a map $\eta\colon Q\to \N$ such that $\eta^{-1}(0)=G$.
Consider the map $\mu\colon Q\to Q\oplus \N$ sending $\sigma\in Q$ to $(\sigma,\eta(\sigma))$, which induces a map
\[
h\colon \Z[Q]\to \Z[Q\oplus \N].
\]
The zero and one sections $\Spec(\Z)\rightrightarrows \A^1$ induce maps
\[
a_0,a_1\colon \Z[Q\oplus \N]
\to
\Z[Q].
\]
Consider the log structure map $\gamma\colon Q\oplus \N \to \Z[Q\oplus \N]$.
We set $x^\sigma:=\gamma(\sigma,0)$ for $\sigma\in Q$ and $t:=\gamma(0,1)$.
Then $a_1h=\id$ and
\[
a_0 h(x^\sigma)=a_0(x^\sigma t^{\eta(\sigma)})
=
\left\{
\begin{array}{ll}
0 & \text{if }\sigma\in Q-G,
\\
x^\sigma & \text{if }\sigma\in G.
\end{array}
\right.
\]
This shows that the induced morphism
\[
m\colon \ul{\A_Q}\times \A^1\to \ul{\A_Q}.
\]
is an elementary $\A^1$-homotopy between $ip$ and $\id$ over $\ul{\A_\N}$, where $p\colon \ul{\A_Q}\to \ul{\A_{G}}$ is the morphism induced by the inclusion $G\to Q$, and $i\colon \ul{\A_G}\simeq \ul{\A_{(Q,Q-G)}}\to \ul{\A_Q}$ is the obvious closed immersion.

Use $m$ to deduce that the composite
\begin{align*}
\hom_{\sT(\ul{\A_{Q,B}})}
(M(\ul{\A_{Q,B}}),
M(\ul{\A_{Q,B}\times_{\A_{\N,B}}C}))
\xrightarrow{i^*} &
\hom_{\sT(\ul{\A_{G,B}})}
(M(\ul{\A_{G,B}}),
M(\ul{\A_{G,B}\times_{\A_{\N,B}}C}))
\\
\xrightarrow{p^*} &
\hom_{\sT(\ul{\A_{Q,B}})}
(M(\ul{\A_{Q,B}}),
M(\ul{\A_{Q,B}\times_{\A_{\N,B}}C}))
\end{align*}
is homotopic to $\id^*\simeq \id$.
Since $pi\simeq \id$,
we obtain
\begin{equation}
\label{lemma.17.1}
\Phi(\A_{Q,B})
\simeq
\Phi(\A_{G,B}).
\end{equation}
We can similarly show
\begin{equation}
\label{lemma.17.2}
\Phi(\A_{Q',B})
\simeq
\Phi(\A_{G,B}).
\end{equation}
Let $\Delta$ be the fan in $L$ whose maximal cones are $\Spec(P)$ and $\Spec(Q)$,
i.e., $\Delta$ consists of those cones that are faces of $\Spec(P)$ and $\Spec(Q)$.
Then $\Delta$ is the star subdivision of $\Spec(Q')$ at $e_2$,
see \cite[\S 11.1]{CLStoric} for star subdivisions.
The image of $1\in \N$ in $Q'$ can be expressed as $(1,0,a_3,\ldots,a_n)$ under the $\Q$-basis $\{a_1e_1+a_2e_2,e_2,e_3,\ldots,e_n\}$ of $Q'\otimes \Q$.
It follows that $\ol{Q'\oplus_{\N} \Z}$ is generated by the image of $e_2$.
Hence $\ol{Q'\oplus_{\N}\Z}^\gp$ has rank $1$,
so $V:=\A_{Q',B}\times_{\A_{\N,B}}C$ of $\A_{Q',B}$ satisfies the condition that for every point $v$ of $V$, $\ol{\cM}_{V,v}^\gp$ has rank $\leq 1$,
Together with Lemma \ref{lemma.10},
we have
\begin{equation}
\label{lemma.17.3}
\Phi(\T_{\Delta,B})\simeq \Phi(\A_{Q',B}).
\end{equation}
The intersection of the cones $\Spec(P)$ and $\Spec(Q)$ is $\Spec(P_F)$ since
\[
\Spec(P)
\simeq e_1\N \oplus \cdots \oplus e_n \N,
\;
\Spec(P_F)
\simeq e_2 \N\oplus \cdots \oplus e_n \N.
\]
By Zariski descent, the induced square
\[
\begin{tikzcd}
\Phi(\T_{\Delta,B})\ar[d]\ar[r]&
\Phi(\A_{Q,B})\ar[d]
\\
\Phi(\A_{P,B})\ar[r]&
\Phi(\A_{P_F,B})
\end{tikzcd}
\]
is cartesian.
The lower horizontal morphism is an isomorphism since the upper horizontal morphism is an isomorphism by \eqref{lemma.17.1}, \eqref{lemma.17.2}, and \eqref{lemma.17.3}.
\end{proof}

For every fan $\Sigma$ in a lattice $L$, let $|\Sigma|_{\R}$ denote the closure of the support $|\Sigma|$ in $\R\otimes_{\Q} L$.
The origin of $|\Sigma|_{\R}$ is denoted by $0$.

\begin{lem}
\label{lemma.16}
Let $\theta\colon \N\to P$ be a nontrivial map of fs monoids.
If $P^\gp$ is torsion free and $P\neq P^*$, then the topological space
\[
|\Spec(P)-\partial_{\Spec(\N)}\Spec(P)|_{\R}-\{0\}
\]
is acyclic.
\end{lem}
\begin{proof}
If $\theta$ is vertical, then
$|\Spec(P)-\partial_{\Spec(\N)}\Spec(P)|_{\R}-\{0\}=|\Spec(P)|_{\R}-\{0\}$ is contractible since $|\Spec(P)|_{\R}-\{0\}$ is convex.

Assume $\theta$ is not vertical.
Let $f\colon |\Spec(P)|_{\R}\to |\Spec(\N)|_{\R}\simeq \R_{\geq 0}$ be the morphism induced by $\theta$.
Choose any map $\eta\colon \N\to P$ such that $\eta(1)$ is not contained in any proper face of $P$.
This induces a map $h\colon |\Spec(P)|_{\R}\to \R_{\geq 0}$ such that $h^{-1}(0)=0$.
We set
\[
V:=f^{-1}(0)\cap h^{-1}(1)
\text{ and }
W:=|\Spec(P)-\partial_{\Spec(\N)}\Spec(P)|_{\R}\cap h^{-1}(1).
\]
For a cone $\sigma\in \Spec(P)$, we have $\sigma\in \Spec(P)-\partial_{\Spec(\N)}\Spec(P)$ if and only if $\sigma$ does not contain a ray $r$ with $f(r)=0$. 
Hence $f(x)>0$ for all $x\in |\Spec(P)-\partial_{\Spec(\N)}\Spec(P)|_{\R}-\{0\}$, so we have $V\cap W=\emptyset$.
There is a homeomorphism
\[
|\Spec(P)-\partial_{\Spec(\N)}\Spec(P)|_{\R}-\{0\}
\simeq
W\times \R_{>0}.
\]
Hence it suffices to show that $W$ is acyclic.

Let $X$ be the boundary of $h^{-1}(1)$.
Since $h^{-1}(1)$ is a convex polytope, $X$ is homeomorphic to $S^{n-2}$, where $n$ is the rank of $\ol{P}^\gp$.
We assumed that $\theta$ is not vertical, so we have $W\subset X$.
By the Alexander duality
\[
\tilde{H}_q(X-W)
\simeq
\tilde{H}^{n-q-3}(W)
\]
that holds for every integer $q\geq 0$, we only need to show that $U:=X-W$ is acyclic.

Let $\Gamma$ be the set of cones $\sigma$ of $\Spec(P)$ such that there exists two rays $r_1$ and $r_2$ of $\sigma$ with $f(r_1)=0$ and $f(r_2)\neq 0$.
Equivalently, $|\sigma|_{\R}\cap V,|\sigma|_{\R}\cap W\neq \emptyset$.
We set $u(\sigma):=|\sigma|_{\R}\cap U$.
Let $\{\sigma_1,\ldots,\sigma_m\}$ Let $I:=\{i_1,\ldots,i_r\}$ be a nonempty subset of $\{1,\ldots,m\}$, and we set $\sigma_I:=\sigma_{i_1}\cap \cdots \cap \sigma_{i_r}$.
If $u(\sigma_I)=\emptyset$, then $u(\sigma_I)\cap V=\emptyset$.
On the other hand, if $u(\sigma_I)\neq \emptyset$, then $\sigma_I$ contains a ray $r$ with $f(r)=0$.
Hence we have $u(\sigma_I)\cap V\neq \emptyset$.
Since $u(\sigma_I)$ and $u(\sigma_I)\cap V$ are convex, they are contractible.
Hence the inclusion
\[
u(\sigma_I)\cap V
\to
u(\sigma_I)
\]
is a homotopy equivalence.

Since 
$\{u(\sigma_1),\ldots,u(\sigma_m),V\}$ (resp.\ $\{u(\sigma_1)\cap V,\ldots,u(\sigma_m)\cap V,V\}$) is a closed cover of $U$ (resp.\ $V$), the homomorphism of singular cohomology groups $H_*(V)\to H_*(U)$ induced by the inclusion $V\to U$ is an isomorphism by 
Mayer-Vietoris.
The assumption that $\theta$ is nontrivial implies that $V$ is not empty.
Furthermore, $V$ is convex, so $V$ is contractible.
It follows that $U$ is acyclic.
\end{proof}

\begin{exm}
Let us give a specific example for Lemma \ref{lemma.16} to help the readers understand its proof.
Suppose that $P$ is the submonoid of $\Z^3$ generated by $(1,0,0)$, $(0,1,0)$, $(1,0,1)$, $(0,1,1)$.
Then $|\Spec(P)|_{\R}$ is generated by $(1,0,0)$, $(0,1,0)$, $(0,0,1)$, and $(1,1,-1)$.
Suppose that $\theta\colon \N\to P$ (resp.\ $\eta\colon \N\to P$) is the map sending $1$ to $(1,1,0)$ (resp.\ $(1,1,1)$).
Then $f$ (resp.\ $h$) sends $(x,y,z)$ to $x+y$ (resp.\ $x+y+z$).
This implies that we have $h^{-1}(1)$ is the convex set generated by $(1,0,0)$, $(0,1,0)$, $(0,0,1)$, and $(1,1,-1)$,
and we have $V=(0,0,1)$.
The following figure illustrates several topological spaces appearing in the proof of Lemma \ref{lemma.16}.
\[
\begin{tikzpicture}
\fill[black!20] (0.8,0)--(0,0.8)--(-0.8,0)--(0,-0.8);
\draw[ultra thick] (0.8,0)--(0,0.8)--(-0.8,0)--(0,-0.8)--(0.8,0);
\filldraw (0.8,0) circle (2pt);
\filldraw (0,0.8) circle (2pt);
\filldraw (-0.8,0) circle (2pt);
\filldraw (0,-0.8) circle (2pt);
\node at (0,-1.3) {$h^{-1}(1)$};
\draw[ultra thick] (3.2,0)--(2.4,0.8)--(1.6,0)--(2.4,-0.8)--(3.2,0);
\filldraw (3.2,0) circle (2pt);
\filldraw (2.4,0.8) circle (2pt);
\filldraw (1.6,0) circle (2pt);
\filldraw (2.4,-0.8) circle (2pt);
\filldraw (3.2,0) circle (2pt);
\node at (2.4,-1.3) {$X$};
\filldraw (4.8,-0.8) circle (2pt);
\node at (4.8,-1.3) {$V$};
\draw[ultra thick] (8,0)--(7.2,0.8)--(6.4,0);
\filldraw (8,0) circle (2pt);
\filldraw (7.2,0.8) circle (2pt);
\filldraw (6.4,0) circle (2pt);
\node at (7.2,-1.3) {$W$};
\draw[ultra thick] (10.4,0)--(9.6,-0.8)--(8.8,0);
\draw (10.4,0) circle (2pt);
\filldraw (9.6,-0.8) circle (2pt);
\draw (8.8,0) circle (2pt);
\node at (9.6,-1.3) {$U$};
\end{tikzpicture}
\]
As claimed in Lemma \ref{lemma.16},
$X$ is homeomorphic to $S^1$,
$V$ is contractible,
and $W$ and $U$ are acyclic.
\end{exm}

\begin{lem}
\label{lemma.15}
Let $f\colon \Delta\to \Spec(\N)$ be a vertical morphism of fans, and let $\sigma$ be any nontrivial cone of $\Delta$.
If the topological space $|\Delta|_{\R}-\{0\}$ is acyclic, then there is an isomorphism $\Phi(\T_{\sigma,B})\simeq \Phi(\T_{\Delta,B})$.
\end{lem}
\begin{proof}
Since $\Phi$ is invariant for dividing covers, we may assume that every cone of $\Delta$ is smooth by \cite[Theorem 11.1.9]{CLStoric}.
For a cone $\delta$ of $\Delta$, we set
\[
\Phi(\delta)
:=
\Phi(\T_{\delta,B}).
\]
Observe that $\Phi$ is contravariant for the inclusions of cones.

There exists a colimit preserving functor
\[
\alpha\colon \Spt \to \sT(S)
\]
sending the sphere spectrum $\mathbb{S}$ to $\unit$ by \cite[Corollary 1.4.4.6]{HA}.
For a space $V$, we set
\[
\epsilon(V)
:=
\alpha(\map_{\Spt}(\Sigma^\infty V_+,\mathbb{S})).
\]
If $V$ is contractible, then $\epsilon(V)\simeq \unit$.
If $V$ is empty, then $\epsilon(V)\simeq 0$.

If $\delta'\subset \delta$ are nontrivial cones of $\Delta$, then Lemma \ref{lemma.17} shows $\Phi(\delta')\simeq \Phi(\delta)$.
Together with the connectivity of $|\delta|_{\R}-\{0\}$, we have
\[
\Phi(\delta)\simeq \Phi(\sigma).
\]
Observe also that $|\delta|_{\R}-\{0\}$ is contractible since it is convex.
On the other hand, $\Phi(\delta)=\Phi(0)$ if $\delta=0$.
These two formulas for $\Phi$ can be combined into a single functorial isomorphism
\begin{equation}
\label{lemma.15.3}
\Phi(\delta)
\simeq
\Phi(\sigma)\otimes \epsilon(|\delta|_{\R}-\{0\})
\oplus
\Phi(0)\otimes \fib(\unit \to \epsilon(|\delta|_{\R}-\{0\}))
\end{equation}
for every cone $\delta$ of $\Delta$.

Let $\delta_1,\ldots,\delta_m$ be all the maximal cones of $\Delta$.
If $v$ is a $0$-simplex of the simplicial set $I:=(\Delta^1)^m-\{(1,\ldots,1)\}$, we set
\[
\delta_v:=\delta_{i_1}\cap \cdots \cap \delta_{i_r},
\]
where the $i$th coordinate of $v$ is $0$ if and only if $i\in \{i_1,\ldots,i_r\}$.
If $v\to w$ is a morphism in $I$, then we have the morphism $\delta_v\to \delta_w$ given by the inclusion of the cones.
Hence we can take limits on both sides of \eqref{lemma.15.3} to obtain an isomorphism
\begin{equation}
\label{lemma.15.2}
\limit_{v\in I} \Phi(\delta)
\simeq
\limit_{v\in I}
\big(
\Phi(\sigma)\otimes \epsilon(|\delta|_{\R}-\{0\})
\oplus
\Phi(0)\otimes \fib(\unit \to \epsilon(|\delta|_{\R}-\{0\}))
\big).
\end{equation}
The left-hand side of \eqref{lemma.15.2} is isomorphic to $\Phi(\T_{\Delta,B})$ by Zariski descent.
Since $\{|\delta_1|_{\R}-\{0\},\ldots,|\delta_m|_{\R}-\{0\}\}$ is a closed cover of $|\Delta|_{\R}-\{0\}$, the right-hand side of \eqref{lemma.15.2} is isomorphic to
\[
\Phi(\sigma)\otimes \epsilon(|\Delta|_{\R}-\{0\})
\oplus
\Phi(0)\otimes \fib(\unit \to \epsilon(|\Delta|_{\R}-\{0\}))
\]
by Mayer-Vietoris, which is isomorphic to $\Phi(\sigma)$ since $|\Delta|_{\R}-\{0\}$ is acyclic.
\end{proof}

We refer to \cite[\S 11.1]{CLStoric} for the definition of a star subdivision of a fan.

\begin{lem}
\label{lemma.1}
Let $\theta\colon \N\to Q$ be a map of fs monoids, and let $\Sigma$ be a star subdivision of $\Spec(P)$ with $P:=Q^\tf$.
Then $\Theta(p)$ is an isomorphism, where $p\colon \A_{Q,B}\times_{\A_P}\T_{\Sigma}\to \A_{Q,B}$ is the projection.
\end{lem}
\begin{proof}
Replace $(Q,B)$ by $(P,\A_{Q^\tor,B})$ to reduce to the case when $Q^\gp$ is torsion free, i.e., $P=Q$.

If $\theta$ is trivial, then
\[
\A_{P,B}-\partial_{\A_{\N,B}}\A_{P,B}
\simeq
\A_{P,B}-\partial\A_{P,B}
\simeq 
\T_{\Sigma,B}-\partial_{\A_{\N,B}}\T_{\Sigma,B}
\simeq
\T_{\Sigma,B}-\partial\T_{\Sigma,B}.
\]
Hence $\Theta(p)$ is an isomorphism.

Assume that $\theta$ is nontrivial.
Let $N$ be the dual lattice of $P^\gp$, and let $v$ be the point of $|\Spec(P)|\cap N$ such that $\Sigma$ is the star subdivision of $\Spec(P)$ at $v$.
If $f(v)=0$, then $\Sigma-\partial_{\Spec(\N)}\Sigma\simeq \Spec(P)-\partial_{\Spec(\N)}\Spec(P)$.
Hence we are done.

Assume $f(v)\neq 0$.
For every $x\in |\Sigma-\partial_{\Spec(\N)}\Sigma|_{\R}$, the segment whose two ends are $x$ and $v$ is contained in $|\Sigma-\partial_{\Spec(\N)}\Sigma|_{\R}$.
It follows that $|\Sigma-\partial_{\Spec(\N)}\Sigma|_{\R}-\{0\}$ is contractible.
Together with Lemmas \ref{lemma.16} and \ref{lemma.15}, we obtain $\Theta(\A_P)\simeq \Theta(\T_\Sigma)$.
\end{proof}

\begin{lem}
\label{lemma.8}
Let $\theta\colon \N\to Q$ be a map of fs monoids, let $X\to \A_{Q,B}$ be a strict \'etale morphism in $\lSch$, and let $\Sigma$ be a star subdivision of $\Spec(P)$ with $P:=Q^\tf$.
Then $\Theta(p)$ is an isomorphism, where $p\colon X':=X\times_{\A_P}\T_{\Sigma}\to X$ is the projection.
\end{lem}
\begin{proof}
We proceed by induction on $m:=\rank(\ol{P}^\gp)$.
If $m=0,1$, then there is no nontrivial star subdivision of $\Spec(P)$.
Hence the claim is trivial.

Assume $m\geq 2$.
Let $i\colon \A_{(P,P^+),B}\to \A_{P,B}$ be the obvious closed immersion, let $j$ be its open complement, and let $q\colon \T_\Sigma\to \A_P$ be the projection.
We have the induced commutative diagram with cartesian squares
\begin{equation}
\label{lemma.8.1}
\begin{tikzcd}
\A_{(P,P^+),B}\times_{\A_P}\T_{\Sigma,B}\ar[d]\ar[r,"i'"]&
\T_{\Sigma,B}\ar[d,"q"]\ar[r,leftarrow,"j'"]&
(\A_{P,B}-\A_{(P,P^+),B})\times_{\A_P}\T_\Sigma\ar[d]
\\
\A_{(P,P^+),B}\ar[r,"i"]&
\A_{P,B}\ar[r,leftarrow,"j"]&
\A_{P,B}-\A_{(P,P^+),B}.
\end{tikzcd}
\end{equation}

Consider the morphism $\A_{P,B}\to \ul{ \A_{(P,P^+),B}}\simeq \A_{P^*,B}$ induced by the inclusion $P^*\to P$.
We set
\[
Y:=\ul{X\times_{\A_{P,B}} \A_{(P,P^+),B}}\times_{\ul{\A_{(P,P^+),B}}}\A_{P,B},
\;
Y':=\ul{X\times_{\A_{P,B}} \A_{(P,P^+),B}}\times_{\ul{\A_{(P,P^+),B}}}\T_{\Sigma,B}.
\]
Since $X$ is strict smooth over $\A_{P,B}$, $Y$ (resp.\ $Y'$) is strict smooth over $\A_{P,B}$ (resp.\ $\T_{\Sigma,B}$).
There is a commutative diagram with cartesian squares
\begin{equation}
\label{lemma.8.2}
\begin{tikzcd}[column sep=small, row sep=small]
X'\ar[r,leftarrow]\ar[d]&
X'\times_{\A_{P,B}}\A_{(P,P^+),B}\ar[r]\ar[d]&
Y'\ar[d]
\\
X\ar[r,leftarrow]\ar[d]&
X\times_{\A_{P,B}}\A_{(P,P^+),B}\ar[d]\ar[r]&
Y\ar[d]
\\
\A_{P,B}\ar[r,leftarrow,"i"]&
\A_{(P,P^+),B}\ar[r,"i"]&
\A_{P,B}.
\end{tikzcd}
\end{equation}

For simplicity of notation, we set $U:=C\times_{\A_{\N}}\A_P$ and $U':=U\times_{\A_P}\T_{\Sigma}\simeq U$.
By Theorem \ref{consq.1}, we have an isomorphism $\ul{q}_*M(U')\simeq M(U)$.
Since $q$ is log smooth, there is a commutative diagram
\begin{equation}
\label{lemma.2.3}
\begin{tikzcd}
\ul{q}^*\ul{j}_\sharp \ul{j}^*\ar[r,"ad'"]\ar[d,"\simeq"']&
\ul{q}^*\ar[r,"ad"]\ar[d,"\simeq"]&
\ul{q}^*\ul{i}_*\ul{i}^*\ar[d,"\simeq"]
\\
\ul{j}_\sharp'\ul{j}'^*\ul{q}^*\ar[r,"ad'"]&
\ul{q}'^*\ar[r,"ad"]&
\ul{i}_*'\ul{i}'^*\ul{q}^*
\end{tikzcd}
\end{equation}
whose rows are cofiber sequences and vertical morphisms are isomorphisms.
We set $Y'':=(Y-\partial_{\A_{\N,B}}Y)\times_Y Y'$.
From \eqref{lemma.2.3}, we obtain a commutative diagram
\[
\begin{tikzcd}[row sep=small, column sep=tiny]
\map_{\sT(\ul{\T_{\Sigma,B}})}(\ul{i}_*'\ul{i}'^*M(\ul{Y''}),M(U'))\ar[d]\ar[r,"\simeq"]&
\map_{\sT(\ul{\A_{P,B}})}(\ul{i}_*\ul{i}^*M(\ul{Y-\partial_{\A_{\N,B}}Y}),M(U))\ar[d]
\\
\map_{\sT(\ul{\T_{\Sigma,B}})}(M(\ul{Y''}),M(U'))\ar[d]\ar[r,"\simeq"]&
\map_{\sT(\ul{\A_{P,B}})}(M(\ul{Y-\partial_{\A_{\N,B}}Y}),M(U))\ar[d]
\\
\map_{\sT(\ul{\T_{\Sigma,B}})}(\ul{j}_\sharp' \ul{j}'^*M(\ul{Y''}),M(U'))\ar[r,"\simeq"]&
\map_{\sT(\ul{\A_{P,B}})}(\ul{j}_\sharp\ul{j}^*M(\ul{Y-\partial_{\A_{\N,B}}Y}),M(U))
\end{tikzcd}
\]
whose columns are fiber sequences.
Together with the open immersion $Y''\hookrightarrow Y'-\partial_{\A_{\N,B}}Y'$, we obtain a commutative diagram
\begin{equation}
\label{lemma.2.1}
\begin{tikzpicture}[baseline= (a).base]
\node[scale=.9] (a) at (0,0)
{
\begin{tikzcd}[row sep=small, column sep=tiny]
\map_{\sT(\ul{\T_{\Sigma,B}})}(\ul{i}_*'\ul{i}'^*M(\ul{Y'-\partial_{\A_{\N,B}}Y'}),M(U'))\ar[d]\ar[r]&
\map_{\sT(\ul{\A_{P,B}})}(\ul{i}_*\ul{i}^*M(\ul{Y-\partial_{\A_{\N,B}}Y}),M(U))\ar[d]
\\
\map_{\sT(\ul{\T_{\Sigma,B}})}(M(\ul{Y'-\partial_{\A_{\N,B}}Y'}),M(U'))\ar[d]\ar[r]&
\map_{\sT(\ul{\A_{P,B}})}(M(\ul{Y-\partial_{\A_{\N,B}}Y}),M(U))\ar[d]
\\
\map_{\sT(\ul{\T_{\Sigma,B}})}(\ul{j}_\sharp' \ul{j}'^*M(\ul{Y'-\partial_{\A_{\N,B}}Y'}),M(U'))\ar[r]&
\map_{\sT(\ul{\A_{P,B}})}(\ul{j}_\sharp\ul{j}^*M(\ul{Y-\partial_{\A_{\N,B}}Y}),M(U))
\end{tikzcd}
};
\end{tikzpicture}
\end{equation}
whose columns are fiber sequences.

The fs log scheme $\A_{P,B}-\A_{(P,P^+),B}$ admits a Zariski cover
\[
\{\A_{P_F,B}:
F\text{ is a nonzero face of }P\}.
\]
By induction, the morphism $\Theta(Y'\times_{\A_P}\A_{P_F})\to \Theta(Y\times_{\A_P}\A_{P_F})$ is an isomorphism.
Hence by Zariski descent and adjunction, the third row of \eqref{lemma.2.1} is an isomorphism.

By \cite[Lemma I.6.7]{zbMATH07027475},
we have isomorphisms
\[
Y:=\ul{X\times_{\A_{P,B}} \A_{(P,P^+),B}}\times \A_{\ol{P},B},
\;
Y':=\ul{X\times_{\A_{P,B}} \A_{(P,P^+),B}}\times \T_{\Sigma',B},
\]
where $\Sigma':=\Sigma\times_{\Spec(P)}\Spec(\ol{P})$ is a star subdivision of $\Spec(\ol{P})$.
Replace $B$ by $\ul{X\times_{\A_{P,B}} \A_{(P,P^+),B}}$ and use Lemma \ref{lemma.1} to deduce that the second row of \eqref{lemma.2.1} is an isomorphism.
It follows that the first row of \eqref{lemma.2.1} is an isomorphism too.

We also have a commutative diagram
\begin{equation}
\label{lemma.2.2}
\begin{tikzpicture}[baseline= (a).base]
\node[scale=.9] (a) at (0,0)
{
\begin{tikzcd}[row sep=small, column sep=tiny]
\map_{\sT(\ul{\T_{\Sigma,B}})}(\ul{i}_*'\ul{i}'^*M(\ul{X'-\partial_{\A_{\N,B}}X'}),M(U'))\ar[d]\ar[r]&
\map_{\sT(\ul{\A_{P,B}})}(\ul{i}_*\ul{i}^*M(\ul{X-\partial_{\A_{\N,B}}X}),M(U))\ar[d]
\\
\map_{\sT(\ul{\T_{\Sigma,B}})}(M(\ul{X'-\partial_{\A_{\N,B}}X'}),M(U'))\ar[d]\ar[r]&
\map_{\sT(\ul{\A_{P,B}})}(M(\ul{X-\partial_{\A_{\N,B}}X}),M(U))\ar[d]
\\
\map_{\sT(\ul{\T_{\Sigma,B}})}(\ul{j}_\sharp' \ul{j}'^*M(\ul{X'-\partial_{\A_{\N,B}}X'}),M(U'))\ar[r]&
\map_{\sT(\ul{\A_{P,B}})}(\ul{j}_\sharp\ul{j}^*M(\ul{X-\partial_{\A_{\N,B}}X}),M(U))
\end{tikzcd}
};
\end{tikzpicture}
\end{equation}
whose rows are fiber sequences.
From \eqref{lemma.8.2}, we see that the first rows of \eqref{lemma.2.1} and \eqref{lemma.2.2} are isomorphic.
By induction, the morphism $\Theta(X'\times_{\A_P}\A_{P_F})\to \Theta(X\times_{\A_P}\A_{P_F})$ is an isomorphism for every nonzero face $F$ of $P$.
Hence by Zariski descent and adjunction, the third row of \eqref{lemma.2.2} is an isomorphism.
It follows that the second row of \eqref{lemma.2.2} is an isomorphism too.
By adjunction, we obtain the desired isomorphism.
\end{proof}

\begin{lem}
\label{lemma.14}
Let $\theta\colon \N\to Q$ be a map of fs monoids, let $X\to \A_{Q,B}$ be a strict \'etale morphism, let $\Sigma$ be a subdivision of $\Spec(P)$ with $P:=Q^\tf$, and let $\Delta$ be a subdivision of $\Sigma$ obtained by finite successions of star subdivisions.
Then $\Theta(q)$ is an isomorphism, where $q\colon X\times_{\A_{P}}\T_{\Delta}\to X\times_{\A_{P}}\T_{\Sigma}$ be the induced morphism.
\end{lem}
\begin{proof}
We only need to show the claim when $\Delta$ is a star subdivision of $\Sigma$.
By Zariski descent, we reduce to the case when $\Sigma= \Spec(R)$ for an fs monoid $R$.
Then $X\times_{\A_P}\A_R\simeq X\times_{\A_Q}\A_{R\oplus Q^\tor}$ is strict \'etale over $\A_{R\oplus Q^\tor,B}$.
We can use Lemma \ref{lemma.8} in this situation.
\end{proof}

\begin{lem}
\label{lemma.13}
Let $\theta\colon \N\to Q$ be a map of fs monoids, let $X\to \A_{Q,B}$ be a strict \'etale morphism in $\lSch$, and let $\Sigma$ be a subdivision of $\Spec(P)$ with $P:=Q^\tf$.
Then $\Theta(p)$ is an isomorphism, where $p\colon X\times_{\A_{P}}\T_{\Sigma}\to X$ is the projection.
\end{lem}
\begin{proof}
Let $\cB$ be the category of subdivisions of $\Spec(P)$, and let $\cA$ be the class of morphisms in $\cA$ that are obtained by finite successions of star subdivisions relative to two dimensional cones.
By \cite[Lemma A.3.15]{logDM}, $\cA$ admits a calculus of right fractions in $\cB$.

Consider the functor
\[
F\colon \cB\to \Ho(\Spt).
\]
sending $\Sigma\in \cB$ to $\Theta(X\times_{\A_{P}}\T_{\Sigma})$.
By Lemma \ref{lemma.14}, $\Theta$ sends every morphism in $\cA$ to an isomorphism.
Together with \cite[Lemmas A.3.15, C.2.1]{logDM}, we deduce that $\Theta$ sends every morphism in $\cB$ to an isomorphism.
\end{proof}

\begin{lem}
\label{lemma.12}
Let $\N\to Q$ be a map of fs monoids, let $X\to \A_{Q,B}$ be a strict \'etale morphism in $\lSch$, let $X'\to X$ be a log \'etale monomorphism, and let $\Sigma$ be a subdivision of $\Spec(P)$ with $P:=Q^\tf$.
Then $\Theta(p)$ is an isomorphism, where $p\colon X'\times_{\A_{P}}\T_{\Sigma}\to X'$ is the projection.
\end{lem}
\begin{proof}
The question is Zariski local on $X$ and $X'$.
Hence by \cite[Lemma A.11.3]{logDM}, we may assume that there exists a map $\theta\colon P\to P'$ such that $\theta^\gp$ is an isomorphism and the induced morphism $X'\to X\times_{\A_P}\A_{P'}$ is an open immersion.
This means that the induced morphism $X'\to \A_{P'\oplus Q^{\tor},B}$ is strict \'etale.

We set $\Sigma':=\Spec(P')\times_{\Spec(P)}\Sigma$, which is a subdivision of $\Spec(P')$.
Then we have $X'\times_{\A_{P}}\T_{\Sigma}\simeq X'\times_{\A_{P'}}\T_{\Sigma'}$, and use Lemma \ref{lemma.13} to conclude.
\end{proof}

\begin{proof}[Proof of Proposition \ref{lemma.7}]
We first treat the case (i).
By applying \cite[Proposition A.11.5]{logDM} to the dividing cover $X'\to X$ and the chart $P:=Q^\tf$ of $X$,
we obtain a subdivision $\Sigma$ of $\Spec(P)$ such that the pullback
\[
r\colon X'\times_{\A_P}\T_{\Sigma}\to X\times_{\A_P}\T_{\Sigma}
\]
is an isomorphism.
Let $p\colon X\times_{\A_P}\T_\Sigma\to X$ and $p'\colon X'\times_{\A_P}\T_\Sigma\to X'$ be the projections.
By Lemma \ref{lemma.12}, $\Theta(p)$ and $\Theta(p')$ are isomorphisms.
Since $\Theta(r)\Theta(p)=\Theta(p')\Theta(f)$, $\Theta(f)$ is an isomorphism.

Next, we treat the case (ii).
The question is Zariski local on $X$.
Hence we may assume that the structure morphism $q\colon X\to \A_{\N,B}$ admits a chart $\eta\colon \N\to P$ by \cite[Proposition II.2.4.2]{Ogu}.
Since $X\in \lSm/B$, we may assume that there exists a neat chart $P'$ of $X$ at a point $x\in X$ such that the induced morphism $X\to \A_{P',B}$ is strict smooth by \cite[Lemma A.5.9]{logDM}.
The structure morphism $q\colon X\to \A_{\N,B}$ induces a map $\eta'\colon \N \to \ol{\cM}_{X,x}\simeq P'$.
We have the morphism $q'\colon X\to \A_{\N,B}$ induced by $\eta'$,
which can differ from the original structure morphism $q\colon X\to \A_{\N,B}$.

We may also assume that the morphism $X\to \A_{P',B}$ factors through a strict \'etale morphism $X\to \A_{P'\oplus \Z^n,B}$ for some integer $n\geq 0$ by \cite[Corollaire IV.17.11.4]{SGA4}.
We set $Q':=P'\oplus \Z^n$.
and let $\theta'\colon \N\to Q'$ be the map given by $a\mapsto (\eta'(a),0)$.

Let $\alpha \colon P\to \Gamma(X,\cM_X)$ and $\alpha'\colon P'\to \Gamma(X,\cM_X)$ be the structure maps.
By \cite[Proposition II,2.3.9]{Ogu},
we may assume that there exist maps $\kappa\colon P\to P'$ and $\beta\colon P\to \Gamma(X,\cO_X^*)$ such that $\alpha=\alpha'\kappa+\beta$.
Consider the diagram
\[
\begin{tikzcd}
\N\ar[r,"\eta"]\ar[rd,"\eta'"']&
P\ar[r,"\alpha_x"]\ar[d,"\kappa"]&
\ol{\cM}_{X,x}
\\
&
P',\ar[ru,"\alpha_x'"']
\end{tikzcd}
\]
where $\alpha_x$ and $\alpha_x'$ are induced by $\alpha$ and $\alpha'$.
The relation $\alpha=\alpha'\kappa+\beta$ implies that the right triangle commutes. Since $P'$ is neat at $x$, $\alpha_x'$ is an isomorphism.
By the construction of $\eta'$,
the outer diagram commutes.
Hence the left triangle commutes.

Let $y$ be a point of $X$.
Consider the diagram
\[
\begin{tikzcd}
\N\ar[r,"\eta"]\ar[rd,"\eta'"']&
P\ar[r,"\alpha_y"]\ar[d,"\kappa"]&
\ol{\cM}_{X,y}
\\
&
P',\ar[ru,"\alpha_y'"']
\end{tikzcd}
\]
where $\alpha_y$ and $\alpha_y'$ are induced by $\alpha$ and $\alpha'$.
We already know that the left triangle commutes,
and the relation $\alpha=\alpha'\kappa+\beta$ implies that the right triangle commutes.
Hence we have an induced commutative diagram
\[
\begin{tikzcd}
\N\ar[r]\ar[rd]&
P/\alpha_y^{-1}(0)\ar[r,"\simeq"]\ar[d]&
\ol{\cM}_{X,y}
\\
&
P'/\alpha_y'^{-1}(0),\ar[ru,"\simeq"']
\end{tikzcd}
\]
The commutativity of the outer diagram implies that the verticality of $y$ over $\A_{\N,B}$ is unchanged if we replace $q$ by $q'$.
Hence
$X-\partial_{\A_{\N,B}} X$ is unchanged if we replace $q$ by $q'$.
The assumption $C=\G_{m,B}$
yields
\[
\Theta(X)
\simeq
\hom_{\sT(\ul{X-\partial_{\A_{\N,B}}X})}(M(\ul{X-\partial_{\A_{\N,B}}X}),M(X-\partial X)).
\]
This is also unchanged if we replace $q$ by $q'$.
We have a similar result for $\Theta(X')$.
Hence we may assume that $q$ and $q'$ agree.

Then we are in the situation of the case (i).
\end{proof}

\subsection{Localization functors}
\label{localization}

For a category $\cC$ and $X\in \cC$,
we set
\[
h(X)
:=
\left\{
\begin{array}{ll}
\Sigma_{S^1}^\infty X_+ & \text{if }\cV=\Spt,
\\
\Lambda^X[0] & \text{if }\cV=\DLambda,
\end{array}
\right.
\]
which is an object of $\infPsh(\cC,\cV)$.
We set $h_t(X):=L_t h(x)$ for a topology $t$ on $\cC$.
This defines a functor
\[
h_t\colon \cC \to \infShv_t(\cC,\cV).
\]

Throughout this subsection, $S\in \lSch$.
Furthermore, $\cV$ is either $\Spt$ or $\DLambda$.
Recall that $\tau$ is one of $\sNis$, $\setale$, and $\ketale$,
and see \S \ref{topology} for $\ul{\tau}$ and $\tau$.

A \emph{fan isochart} of an fs log scheme $X$ is a Kummer morphism $X\to \T_\Sigma$, where $\Sigma$ is a fan.
Here, ``iso'' in isochart stands for isogeny.
Let $\lSmIFan/S$ denote the full subcategory of $\lSm/S$ consisting of those disjoint unions $\amalg_{i\in I} X_i$ such that each $X_i$ admits a fan isochart.

We have the induced $\tau$-topology and $\dtau$-topology on $\lSmIFan/S$.
If $\tau$ is one of $\sNis$ and $\setale$, then fan charts are enough for the below discussions, but we need fan isocharts for the case of $\tau=\ketale$.

Every $X\in \lSm/S$ admits a Zariski cover $Y\to X$ with $Y\in \lSmIFan/S$ according to \cite[Proposition II.2.3.7]{Ogu}.
Hence we have equivalences of topoi
\[
\Shv_{\tau}(\lSm/S)
\simeq
\Shv_{\tau}(\lSmIFan/S),
\text{ }
\Shv_{\dtau}(\lSm/S)
\simeq
\Shv_{\dtau}(\lSmIFan/S)
\]
by the implication (i)$\Rightarrow$(ii) in \cite[Th\'eor\`eme III.4.1]{SGA4}.
Also, we have an equivalence of $\infty$-categories
\begin{equation}
\Shv_{\tau}(\lSm/S,\cU)
\simeq
\Shv_{\tau}(\lSmIFan/S,\cU),
\text{ }
\Shv_{\dtau}(\lSm/S,\cU)
\simeq
\Shv_{\dtau}(\lSmIFan/S,\cU)
\end{equation}
for $\cU=\Spc,\Spc_*,\Spt,\DLambda$ by \cite[Proposition A.3.11]{Mann}.

\begin{prop}
\label{divlocal.12}
Let $f\colon Y\to X$ be a quasi-compact morphism in $\lSch$.
If $X$ admits a fan isochart $\Sigma$, then there exists a subdivision $\Sigma'$ of $\Sigma$ such that the pullback
\[
f'\colon Y\times_{\T_\Sigma}\T_{\Sigma'}\to X\times_{\T_\Sigma}\T_{\Sigma'}
\]
is exact.
\end{prop}
\begin{proof}
We first treat the case when $\Sigma$ is a fan chart of $X$.
By \cite[Proposition II.2.4.2]{Ogu},
we can choose a Zariski cover $Y_1\amalg \cdots \amalg Y_n$ of $Y$ such that each $Y_i\to X$ admits a fan chart $\Delta_i\to \Sigma$.
According to the proof of \cite[Theorem III.2.6.7]{Ogu} and the implication (1)$\Rightarrow$(2) in \cite[Theorem III.2.2.7]{Ogu}, there exists a subdivision $\Sigma_i$ of $\Sigma$ such that the pullback
\[
Y_i\times_{\T_{\Sigma}}\T_{\Sigma_i}
\to
X\times_{\T_{\Sigma}}\T_{\Sigma_i}
\]
is exact.
We set
\[
\Sigma'
:=
\{\sigma_1\cap \cdots \cap \sigma_n
:
\sigma_1\in \Sigma_1,\ldots,\sigma_n\in \Sigma_n\},
\]
which is a common subdivision of $\Sigma_1,\ldots,\Sigma_n$.
Then the pullback $Y_i\times_{\T_{\Sigma}}\T_{\Sigma'}\to X\times_{\T_{\Sigma}}\T_{\Sigma'}$ is exact for all $1\leq i\leq n$.
This implies that $f'$ is exact.

Now we treat the general fan isochart case.
By \cite[Theorem I.4.9.1]{Ogu},
there exists an integer $n\geq 1$ such that the projection $p\colon X':=X\times_{\T_\Sigma}\T_{\Delta}\to \T_{\Delta}$ is saturated,
where $\Delta$ is $\Sigma$, but the morphism $\Delta\to \Sigma$ is obtained by the multiplication of the lattice by $n$.
Since $p$ is Kummer,
\cite[Proposition I.4.8.12]{Ogu} implies that $p$ is strict.
Hence $\Delta$ is a fan chart of $X'$.
By the above special case,
there exists a subdivision $\Delta'$ of $\Delta$ such that the induced morphism
\[
g'\colon
Y'\times_{\T_{\Delta}}\T_{\Delta'}
\to
X'\times_{\T_{\Delta}}\T_{\Delta'}
\]
is exact,
where $Y':=Y\times_X X'$.
Let $\Sigma'$ be the subdivision of $\Sigma$ whose cones are precisely corresponding to the cones of $\Delta'$.
Then we have $\T_{\Sigma'}\times_{\T_{\Sigma}}\T_{\Delta}\simeq \T_{\Delta'}$.
Hence we have an induced cartesian square
\[
\begin{tikzcd}
Y'\times_{\T_\Sigma}\T_{\Sigma'}\ar[d,"h'"']\ar[r,"g'"]&
X'\times_{\T_\Sigma}\T_{\Sigma'}\ar[d,"h"]
\\
Y\times_{\T_\Sigma}\T_{\Sigma'}\ar[r,"f'"]&
X\times_{\T_\Sigma}\T_{\Sigma'}.
\end{tikzcd}
\]
The induced morphism $\T_{\Sigma'}\to \T_{\Sigma}$ is surjective and Kummer,
so $h$ and $h'$ are surjective by \cite[Corollary III.2.2.5]{Ogu} and Kummer.
Since $g'$ is an exact,
\cite[Proposition III.2.2.1(1)]{Ogu} implies that $f'$ is exact.
\end{proof}

\begin{lem}
\label{divlocal.10}
For $X\in \lSch$,
the following hold.
\begin{enumerate}
\item[\textup{(1)}]
$X_{\divi}$ is filtered.
\item[\textup{(2)}]
If $X$ admits a fan isochart $\Sigma$, then the functor
\[
\Sigma_{\divi}\to X_{\divi}
\]
sending $\Sigma'\in \Sigma_{\divi}$ to $X\times_{\T_\Sigma}\T_{\Sigma'}$ is final.
\item[\textup{(3)}]
For every dividing cover $f\colon Y\to X$ in $\lSch$, the functor $f^*\colon X_{\divi}\to Y_{\divi}$ is final.
\end{enumerate}
\end{lem}
\begin{proof}
By \cite[Proposition A.11.14]{logDM}, we have (1).

For (2),
let $f\colon Y\to X$ be a dividing cover in $\lSch$.
By Proposition \ref{divlocal.12},
there exists a subdivision $\Sigma'$ of $\Sigma$ such that the pullback $f'\colon Y\times_{\T_\Sigma}\T_{\Sigma'}\to X\times_{\T_\Sigma}\T_{\Sigma'}$ is exact.
Since $f'$ is a proper surjective log \'etale monomorphism,
\cite[Proposition A.11.13]{logDM} implies that $f'$ is an isomorphism.
Hence we have (2).

For (3), assume $Y'\in Y_{\divi}$.
Then $Y'\times_X Y\simeq Y'$ since $Y\to X$ is a monomorphism.
This shows that $f^*$ is final.
\end{proof}

\begin{const}
\label{divlocal.7}
Consider the functor
\[
G_{\divi}\colon \infPsh(\lSmIFan/S,\cV)
\to
\infPsh(\lSmIFan/S,\cV)
\]
given by
\begin{equation}
\label{divlocal.7.1}
G_{\divi}\cF(X)
:=
\colimit_{X'}\cF(X'),
\end{equation}
where $X'$ runs over the full subcategory of $X_{\divi}$ consisting of those objects $X'$ such that $X'\in \lSmIFan/S$.
If $X$ admits a fan isochart $\Sigma$, then
\begin{equation}
\label{divlocal.7.3}
G_{\divi}\cF(X)
\simeq
\colimit_{\Sigma'\in \Sigma_{\divi}}\cF(X_{\Sigma'})
\end{equation}
by Lemma \ref{divlocal.10}(2), where $X_{\Sigma'}:=X\times_{\T_\Sigma}\T_{\Sigma'}$.

Assume that $X\in \lSmIFan/S$ has a fan isochart $\Sigma$ and
\begin{equation}
\label{divlocal.7.2}
Q
:=
\begin{tikzcd}
W\ar[d]\ar[r]&
Z\ar[d]
\\
Y\ar[r]&
X
\end{tikzcd}
\end{equation}
is a strict Nisnevich distinguished square in $\lSmIFan/S$.
For a subdivision $\Sigma'$ of $\Sigma$, we set
\[
X_{\Sigma'}:=X\times_{\T_{\Sigma}}\T_{\Sigma'},
\;
Y_{\Sigma'}:=Y\times_{\T_{\Sigma}}\T_{\Sigma'},
\;
Z_{\Sigma'}:=Z\times_{\T_{\Sigma}}\T_{\Sigma'},
\;
W_{\Sigma'}:=W\times_{\T_{\Sigma}}\T_{\Sigma'}.
\]
The induced square
\[
\begin{tikzcd}
W_{\Sigma'}\ar[d]\ar[r]&
Z_{\Sigma'}\ar[d]
\\
Y_{\Sigma'}\ar[r]&
X_{\Sigma'}
\end{tikzcd}
\]
is a strict Nisnevich distinguished square.
Apply $\colim_{\Sigma'\in \Sigma_{\divi}}\cF$ to this square and use \eqref{divlocal.7.3} to deduce that the square $G_{\divi}\cF(Q)$ is cartesian.
Hence the above $G_{\divi}$ induces 
\[
G_{\divi}
\colon
\infShv_{\tau}(\lSm/S,\cV)
\to
\infShv_{\tau}(\lSm/S,\cV)
\]
if $\tau=\sNis$.
We similarly obtain $G_{\divi}$ for the cases of $\tau=\detale,\ketale$ too.
By Proposition \ref{divlocal.8} and Lemma \ref{divlocal.10}(3), $\infShv_{\dtau}(\lSm/S,\cV)$ is equivalent to the essential image of $G_{\divi}$.

There is a natural transformation $\beta_{\divi}\colon \id \to G_{\divi}$ such that the morphism
\[
\beta_{\divi}\cF(X) \colon \cF(X)\to \colimit_{X'\in X_{\divi}}\cF(X')
\]
for $\cF\in \infShv_{\tau}(\lSm/S,\cV)$ is induced by the morphisms $\cF(X)\to \cF(X')$ for all $X'\in X_{\divi}$.
Together with the description
\[
G_{\divi}G_{\divi}\cF(X)
\simeq
\colimit_{X'\in X_{\divi},X''\in X_{\divi}'}\cF(X''),
\]
we deduce that the two morphisms
\[
G_{\divi}\beta_{\divi}\cF,\beta_{\divi}G_{\divi}\cF
\colon
G_{\divi}\cF\to G_{\divi}G_{\divi}\cF
\]
are isomorphisms.
Hence by \cite[Proposition 5.2.7.4]{HTT}, we have the following result:
\end{const}

\begin{lem}
\label{divlocal.13}
The functor $G_{\divi}\colon \infShv_{\tau}(\lSm/S,\cV)
\to
\infShv_{\tau}(\lSm/S,\cV)$ is isomorphic to $\iota_{\divi} L_{\divi}$, where  
\[
\iota_{\divi}
\colon
\infShv_{\dtau}(\lSm/S,\cV)
\to
\infShv_{\tau}(\lSm/S,\cV).
\]
is the inclusion functor,
and $L_{\divi}$ is its left adjoint.
\end{lem}

\begin{const}
\label{divlocal.2}
Let $f\colon X'\to X$ be a dividing cover in $\lSm/S$.
Consider the induced morphisms
\[
X-\partial_S X
\leftarrow
(X-\partial_S X)\times_X X'
\hookrightarrow
X'-\partial_S X'
\]
The left morphism becomes an isomorphism in the category of $\dtau$-sheaves.
Hence we obtain a morphism
\[
h_{\dtau}(X-\partial_S X)
\to
h_{\dtau}(X'-\partial_S X').
\]
Now, assume $X\in \lSmIFan/S$.
The \emph{dividing verticalization of $X$ over $S$} is defined to be
\begin{equation}
\label{divlocal.2.8}
h_{\dver}(X)
:=
\colimit_{X'\in X_{\divi}^{\op}}
h_{\dtau}(X'-\partial_S X'),
\end{equation}
where the colimit is taken in the category of $\dtau$-sheaves.
If $X$ admits a fan isochart $\Sigma$, then we have an isomorphism
\begin{equation}
h_{\dver}(X)
\simeq
\colimit_{\Sigma'\in \Sigma_{\divi}^{\op}}
h_{\dtau}(X_{\Sigma'}-\partial_S X_{\Sigma'})
\end{equation}
by Lemma \ref{divlocal.10}(2), where $X_{\Sigma'}:=X\times_{\T_\Sigma}\T_{\Sigma'}$.

If $f\colon Y\to X$ is a morphism in $\lSmIFan/S$, then there exists $X'\in X_{\divi}$ such that the projection $Y':=Y\times_X X'\to X'$ is exact by Proposition \ref{divlocal.12}.
For every $X''\in X_{\divi}'$ with $Y'':=Y\times_X X''$, we have the induced morphism
\[
Y''-\partial_S Y''\to X''-\partial_S X''
\]
by Proposition \ref{vert.9}(6).
Hence we obtain a morphism $h_{\dver}(Y')\to h_{\dver}(X')$.
Since $h_{\dver}(X)\simeq h_{\dver}(X')$ and $h_{\dver}(Y)\simeq h_{\dver}(Y')$ by Lemma \ref{divlocal.10}(3), we obtain a morphism $h_{\dver}(Y)\to h_{\dver}(X)$.
As a consequence, we obtain a functor
\[
h_{\dver}
\colon
\lSmIFan/S
\to
\infShv_{\dtau}(\lSm/S,\cV).
\]
Take colimits to the morphisms $h_{\dtau}(X'-\partial_S X')\to h_{\dtau}(X')\simeq h_{\dtau}(X)$ induced by the open immersion $X'-\partial_S X'\to X'$ for $X'\in X_{\divi}$ to obtain a natural transformation
\begin{equation}
\label{divlocal.2.5}
\alpha_{\dver}\colon h_{\dver}\to h_{\dtau}.
\end{equation}
By \cite[Theorem 5.1.5.6]{HTT}, $h_{\dver}$ naturally extends to a colimit preserving functor
\begin{equation}
\label{divlocal.2.1}
\infPsh(\lSmIFan/S,\cV)
\to
\infShv_{\dtau}(\lSm/S,\cV).
\end{equation}
\end{const}

\begin{const}
\label{divlocal.11}
Assume that $X\in \lSmIFan/S$ has a fan isochart $\Sigma$ with a strict Nisnevich distinguished square $Q$ of the form \eqref{divlocal.7.2}.
For a subdivision $\Sigma'$ of $\Sigma$, consider $X_{\Sigma'}$, $Y_{\Sigma'}$, $Z_{\Sigma'}$, and $W_{\Sigma'}$ in Construction \ref{divlocal.7}.

The induced square
\[
\begin{tikzcd}
W_{\Sigma'}-\partial_S W_{\Sigma'}\ar[d]\ar[r]&
Z_{\Sigma'}-\partial_S Z_{\Sigma'}\ar[d]
\\
Y_{\Sigma'}-\partial_S Y_{\Sigma'}\ar[r]&
X_{\Sigma'}-\partial_S X_{\Sigma'}
\end{tikzcd}
\]
is a strict Nisnevich distinguished square.
It follows that the square $h_{\dver}(Q)$ is cocartesian in $\infShv_{\dtau}(\lSm/S,\cV)$ if $\tau=\sNis$.
The same holds if $\tau=\setale,\ketale$ by a similar argument.

Let $X'\to X$ be a dividing cover in $\lSmIFan/S$.
Since $X_{\divi}\to X_{\divi}'$ is final by Lemma \ref{divlocal.10}(3), the induced morphism
\[
h_{\dver}(X')
\to
h_{\dver}(X)
\]
is an isomorphism.

Together with \cite[Proposition 5.5.4.20]{HTT}, we deduce that \eqref{divlocal.2.1} factors through a colimit preserving functor
\begin{equation}
\label{divlocal.2.2}
F_{\dver}
\colon
\infShv_{\dtau}(\lSm/S,\cV)
\to
\infShv_{\dtau}(\lSm/S,\cV).
\end{equation}
This sends $h_{\dtau}(X)$ for $X\in \lSmIFan/S$ to $h_{\dver}(X)$.
The natural transformation $\alpha_{\dver}$ gives a natural transformation
\begin{equation}
\beta_{\dver}
\colon
F_{\dver} \to \id.
\end{equation}

For $X\in \lSmIFan/S$ and $X'\in X_{\divi}$, the morphism
\[
F_{\dver}(h_{\dtau}(X'-\partial_S X'))
\to
h_{\dtau}(X'-\partial_S X')
\]
induced by $\beta_{\dver}$ is an isomorphism since $X'-\partial_S X'$ is already vertical over $S$.
Take colimits to see that 
\[
\beta_{\dver}h_{\dver}(X)
\colon
F_{\dver}h_{\dver}(X)\to h_{\dver}(X)
\]
is an isomorphism.
Together with \cite[Theorem 5.1.5.6, Proposition 5.5.4.20]{HTT}, we deduce that
\begin{equation}
\label{divlocal.2.6}
\beta_{\dver}F_{\dver}(X)
\colon
F_{\dver}F_{\dver}(X)\to F_{\dver}(X)
\end{equation}
is an isomorphism.
There is a commutative triangle
\[
\begin{tikzcd}
F_{\dver}h_{\dver}(X)\ar[rd,"\beta_{\dver}h_{\dver}(X)"']\ar[rr,"F_{\dver}\alpha_{\dver}(X)"]&
&
F_{\dver}h_{\dtau}(X)\ar[ld,"\simeq",leftarrow]
\\
&
h_{\dver}(X).
\end{tikzcd}
\]
Together with \cite[Theorem 5.1.5.6, Proposition 5.5.4.20]{HTT}, we deduce that
\begin{equation}
\label{divlocal.2.7}
F_{\dver}\beta_{\dver}(X)
\colon
F_{\dver}F_{\dver}(X)\to F_{\dver}(X)
\end{equation}
is an isomorphism.

Let $G_{\dver}$ be a right adjoint of $F_{\dver}$, and let $\dver$ denote the class of morphisms $\alpha_{\dver}(X)$ for all $X\in \lSm/S$.
For $\cF\in \Shv_{\dtau}(\lSm/S,\cV)$, $G_{\dver}\cF$ is $\dver$-local since \eqref{divlocal.2.7} is an isomorphism.
On the other hand, $G_{\dver}\cF\simeq \cF$ if $\cF$ is $\dver$-local.
It follows that the essential image of $G_{\dver}$ is equivalent to $\dver^{-1}\infShv_{\dtau}(\lSm/S,\cV)$.
Furthermore, \cite[Proposition 5.2.7.4]{HTT} shows that $G_{\dver}$ is isomorphic to $\iota_{\dver}L_{\dver}$, where
\[
L_{\dver}\colon \infShv_{\dtau}(\lSm/S,\cV) \to \dver^{-1}\infShv_{\dtau}(\lSm/S,\cV)
\]
is the localization functor and $\iota_{\dver}$ is the inclusion functor.

For $\cF\in \infShv_{\dtau}(\lSm/S,\cV)$ and $X\in \lSmIFan/S$, \eqref{divlocal.2.8} gives
\begin{equation}
\label{divlocal.2.4}
L_{\dver}(\cF)(X)
\simeq
\limit_{X'\in X_{\divi}^{\op}}\cF(X'-\partial_S X').
\end{equation}
\end{const}

\begin{lem}
\label{divlocal.5}
If $X\in \lSmIFan/S$ and $Y\to S$ is a strict morphism in $\Sm/S$, then there is a canonical isomorphism
\[
h_{\dver}(X\times_S Y)
\simeq
h_{\dver}(X)\otimes h_{\dtau}(Y).
\]
\end{lem}
\begin{proof}
By Zariski descent, we may assume that $X$ admits a fan isochart $\Sigma$.
Then we have isomorphisms
\begin{align*}
h_{\dver}(X\times_S Y)
&\simeq
\colimit_{\Sigma'\in \Sigma_{\divi}^{\op}}
h_{\dtau}(X_{\Sigma'}\times_S Y-\partial_S (X_{\Sigma'}\times_S Y))
\\
&\simeq
\colimit_{\Sigma'\in \Sigma_{\divi}^{\op}}
h_{\dtau}(X_{\Sigma'}-\partial_S X_{\Sigma'}) \otimes h_{\dtau}(Y)
\\
& \simeq
h_{\dver}(X)\otimes h_{\dtau}(Y),
\end{align*}
where $X_{\Sigma'}:=X\times_{\T_\Sigma}\T_{\Sigma'}$.
\end{proof}

As a consequence of Lemma \ref{divlocal.5} for $Y=\A_S^1$, we see that $L_{\dver}\cF$ is $\A^1$-local whenever $\cF\in \infShv_{\dtau}(\lSm/S,\cV)$ is $\A^1$-local.
Hence the localization functor
\[
L_{\dver}
\colon
(\A^1)^{-1}\infShv_{\dtau}(\lSm/S,\cV)
\to
(\A^1\cup \dver)^{-1}\infShv_{\dtau}(\lSm/S,\cV)
\]
satisfies \eqref{divlocal.2.4}.
Lemma \ref{divlocal.5} for $Y=\G_{m,S}$ allows us to stabilize this to obtain a functor
\begin{align*}
L_{\dver}
\colon
 & \Stab_{\G_m}((\A^1)^{-1}\infShv_{\dtau}(\lSm/S,\cV))
\\
\to &
\Stab_{\G_m}((\A^1\cup \dver)^{-1}\infShv_{\dtau}(\lSm/S,\cV)).
\end{align*}

Recall that an object $\cF$ of an $\infty$-category is called \emph{compact} if $\Map(\cF,-)$ commutes with filtered colimits.

\begin{const}
Let $\cC$ be a category with a final object $\pt$, a topology $t$, and an interval $I$ in the sense of \cite[\S 2.3]{MV}.
Then $I$ is equipped with morphisms $m\colon I\times I\to I$ and $i_0,i_1\colon I\to \pt$.
Let $p\colon I\to \pt$ be the canonical morphism.
The following two conditions are required:
\begin{enumerate}
\item[(i)] $m(i_0\times \id)=m(\id\times i_0)=i_0p$ and $m(i_1\times \id)=m(\id \times i_1)=\id$.
\item[(ii)] The morphism $i_0\amalg i_1\colon \pt \amalg \pt \to I$ is a monomorphism.
\end{enumerate}

Let us recall Morel's construction of the $I$-localization functor in the proof of \cite[Theorem 4.2.1]{zbMATH05052895}, where he only considered the case when $I=\A^1$.
See also \cite[\S 5.2.c]{CD19}.
We set $U_I:=\cofib(h_t(\pt)\xrightarrow{h_t(i_0)} h_t(I))$.
For $\cF\in \infShv_t(\cC,\cV)$, consider the composition
\[
ev_1\colon \uMap(U_I,\cF) \xrightarrow{\simeq} h_t(\pt) \otimes \uMap(U_I,\cF) \to h_t(U_I) \otimes \uMap(U_I,\cF) \xrightarrow{ev} \cF.
\]
We set
\[
G_{I}^{(1)}(\cF)
:=
\cofib(\uMap(U_I,\cF)\xrightarrow{ev_1} \cF).
\]
This defines a functor $G_{I}^{(1)}\colon \infShv_t(\cC,\cV)\to \infShv_t(\cC,\cV)$.
We have a natural transformation $r_1\colon \id \to G_{I}^{(1)}$.
We inductively define
\[
G_{I}^{(n+1)}:=G_{I}^{(1)}(G_{I}^{(n)})
\]
for integer $n\geq 1$.
We set $r_n:=r_1 G_I^{(n-1)}$ and
\[
G_{I}:=\colimit(G_{I}^{(1)}\xrightarrow{r_2} G_{I}^{(2)} \xrightarrow{r_3} \cdots ).
\]
The natural transformation $r_1$ induces a natural transformation $r\colon \id \to G_{I}$.

Now, we assume that $h_t(X)$ is compact in $\infShv_t(\cC,\cV)$.
Adapt the proofs of \cite[Lemma 5.2.27, Proposition 5.2.28]{CD19} to our setting to see that $\cF\to G_I(\cF)$ is an $I$-equivalence and $G_I(\cF)\in I^{-1}\infShv_t(\cC,\cV)$.
This shows $G_{I}\simeq \iota_I L_{I}$, where 
\[
L_I
\colon
\infShv(\cC,\cV) \to I^{-1}\infShv(\cC,\cV)
\]
is the localization functor, and $\iota_I\colon I^{-1}\infShv(\cC,\cV)\to \infShv(\cC,\cV)$ is the inclusion functor.
\end{const}

Let $\varphi\colon \cC\to \cC'$ be a morphism of sites, and let $t$ (resp.\ $t'$) be the topology on $\cC$ (resp.\ $\cC'$).
Then there is an adjunction
\begin{equation}
\varphi_\sharp : \infShv_t(\cC,\cV)\rightleftarrows \infShv_{t'}(\cC',\cV) : \varphi^*,
\end{equation}
where $\varphi_\sharp$ sends $h_t(X)$ for $X\in \cC$ to $h_{t'}(\varphi(X))$.

Assume that $I$ is an interval object of $\cC$, $\cC$ and $\cC'$ have final objects, and $\varphi$ preserves a final object.
We set $I':=\varphi(I)$, which is an interval object of $\cC'$.
There is an induced adjunction
\begin{equation}
\varphi_\sharp^I : I^{-1}\infShv_t(\cC,\cV)\rightleftarrows I'^{-1}\infShv_{t'}(\cC',\cV) : \varphi_I^*
\end{equation}
such that the square
\begin{equation}
\begin{tikzcd}
\infShv_t(\cC,\cV)\ar[d,"L_I"']\ar[r,"\varphi_\sharp"]&
\infShv_{t'}(\cC',\cV)\ar[d,"L_{I'}"]
\\
I^{-1}\infShv_t(\cC,\cV)\ar[r,"\varphi_\sharp^I"]&
I'^{-1}\infShv_t(\cC',\cV)
\end{tikzcd}
\end{equation}
commutes.

\begin{prop}
\label{divlocal.9}
With the above notation, 
there is an isomorphism
\[
G_I\varphi^*
\simeq
\varphi^* G_{I'}.
\]
\end{prop}
\begin{proof}
For $\cF\in \infShv_t(\cC,\cV)$ and $X\in \cC$, compare
\begin{align*}
& \map_{\infShv_t(\cC,\cV)}(h_t(X),G_I^{(1)}\varphi^*\cF)
\\
\simeq &
\cofib(
\map_{\infShv_t(\cC,\cV)}(h_t(X)\otimes U_I,\varphi^*\cF)
\to
\map_{\infShv_t(\cC,\cV)}(h_t(X),\varphi^*\cF))
\end{align*}
and
\begin{align*}
& \map_{\infShv_t(\cC,\cV)}(h_t(X),\varphi^*G_{I'}^{(1)}\cF)
\simeq 
\map_{\infShv_{t'}(\cC',\cV)}(h_{t'}(\varphi(X)),G_{I'}^{(1)}\cF)
\\
\simeq &
\cofib(
\map_{\infShv_{t'}(\cC',\cV)}(h_{t'}(\varphi(X))\otimes U_{I'},\cF)
\to
\map_{\infShv_{t'}(\cC',\cV)}(h_{t'}(\varphi(X)),\cF))
\end{align*}
to obtain a canonical isomorphism
\[
G_I^{(1)}\varphi^*
\simeq
\varphi^* G_{I'}^{(1)}.
\]
This induces a canonical isomorphism $G_I^{(n)}\varphi^* \simeq \varphi^* G_{I'}^{(n)}$ for all integer $n\geq 1$, and take colimits to conclude.
\end{proof}

As a consequence of Proposition \ref{divlocal.9}, we see that the square
\begin{equation}
\label{divlocal.9.1}
\begin{tikzcd}
\infShv_{t'}(\cC',\cV)\ar[d,"L_I"']\ar[r,"\varphi^*"]&
\infShv_{t}(\cC,\cV)\ar[d,"L_{I'}"]
\\
I^{-1}\infShv_{t'}(\cC',\cV)\ar[r,"\varphi_I^*"]&
I'^{-1}\infShv_t(\cC,\cV)
\end{tikzcd}
\end{equation}
commutes.

\subsection{The main computation}

Throughout this subsection, $\cV$ is either $\Spt$ or $\DLambda$.
Recall that $\tau$ is one of $\sNis$, $\setale$, and $\ketale$,
and see \S \ref{topology} for $\ul{\tau}$ and $\dtau$.
We fix $B\in \Sch$ and an open subscheme $C$ of $\G_{m,B}$.

For $X\in \Sch$, we set
\[
\SH_{\ul{\tau}}(\Sch/X,\cV):=\Stab_{\G_m}((\A^1)^{-1}\infShv_{\ul{\tau}}(\Sch/X,\cV)).
\]
If $V\in \Sch/X$, let $M(X)$ denote the images of $h_{\ul{\tau}}(V)$ in $\SH_{\ul{\tau}}(\Sch/X,\cV)$.
For $Y\in \lSch$, we set
\begin{gather*}
\sC_{\tau}(Y):=\Stab_{\G_m}((\A^1)^{-1}\infShv_{\tau}(\lSm/Y,\cV)),
\\
\sC_{\dtau}(Y):=\Stab_{\G_m}((\A^1)^{-1}\infShv_{\dtau}(\lSm/Y,\cV)).
\end{gather*}
If $W\in \lSm/Y$, let $M(W)$ denote the images of $h_{\tau}(W)$ and $h_{\dtau}(W)$ in $\sC_\tau(Y)$ and $\sC_{\dtau}(Y)$.

According to \eqref{divlocal.2.8}, $\dver$ is already inverted in $\ver^{-1}\sC_{\dtau}(Y)\simeq \SH_{\ul{\tau}}(Y,\cV)$.
Hence we obtain a commutative triangle
\begin{equation}
\begin{tikzcd}
\sC_{\dtau }(Y)\ar[r,"L_{\dver}"]\ar[rd,"L_{ver}"']
&
\dver^{-1}\sC_{\dtau}(Y)\ar[d,"L_{ver}"]
\\
&
\SH_{\ul{\tau}}(Y,\cV).
\end{tikzcd}
\end{equation}

The inclusion functor $\eta\colon \Sm/X\to \Sch/X$ sends $\ul{\tau}$-coverings to $\ul{\tau}$-coverings, $\A^1$ to $\A^1$, and $\G_m$ to $\G_m$.
Hence $\eta$ induces an adjunction
\begin{equation}
\eta_\sharp :
\SH_{\ul{\tau}}(X,\cV) \rightleftarrows \SH_{\ul{\tau}}(\Sch/X,\cV)
: \eta^*,
\end{equation}
where $\eta_\sharp M(V):=M(V)$ for $V\in \Sm/X$.
Consider the functor $\rho\colon \lSm/Y\to \Sch/\ul{Y}$ sending $V\in \lSm/Y$ to $\ul{V}$.
Then $\rho$ sends strict $\tau$-coverings to $\tau$-coverings, $\A^1$ to $\A^1$, and $\G_m$ to $\G_m$.
Hence $\rho$ induces an adjunction
\begin{equation}
\rho_\sharp :
\sC_{\tau}(Y)\rightleftarrows \SH_{\ul{\tau}}(\Sch/\ul{Y},\cV)
: \rho^*,
\end{equation}
where $\rho_\sharp M(V):=M(\ul{V})$.

\begin{rmk}
\label{local.11}
It is well-known that $h_{\Nis}(V)$ is compact in $\infShv_{\Nis}(\Sm/X,\cV)$ for $X\in \Sch$ and $V\in \Sm/X$, see \cite[Proposition 4.5.62]{Ayo072} and its following paragraph.
One can similarly show that $h_{\Nis}(V)$ is compact in $\infShv_{\Nis}(\Sch/X,\cV)$ for $V\in \Sch/X$ and $h_{\sNis}(W)$ is compact in $\infShv_{\sNis}(\lSm/Y,\cV)$ for $Y\in \lSch$ and $W\in \lSm/Y$.

On the other hand,
argue as in \cite[Proposition 2.4.20(1)]{AGV} (see also \cite[Remark 2.4.23]{AGV}) to show that $\infShv_{\et}(\Sm/X,\DLambda)$, $\infShv_{\et}(\Sch/X,\DLambda)$, $\infShv_{\setale}(\lSm/Y,\DLambda)$, and $\infShv_{\ketale}(\lSm/Y,\DLambda)$ (but not the case of $\cV=\Spt$ since the sphere spectrum is not a bounded above spectrum) are compactly generated.
For the case of $\tau=\ketale$,
we also need a result on the Kummer \'etale cohomological dimension \cite[Proposition 5.3]{Nak}.
\end{rmk}

\begin{lem}
\label{local.2}
Assume one of the following conditions:
\begin{enumerate}
\item[\textup{(i)}]
$\tau=\sNis$,
\item[\textup{(ii)}]
$\tau=\setale,\ketale$ and $\cV=\DLambda$.
\end{enumerate}
For $Y\in \lSch$, $V\in \lSm/Y$, $W\in \Sm/(Y-\partial Y)$, and integer $d$,
there is a canonical isomorphism
\[
\map_{\sC_{\tau}(Y)}
(M(V)(d),M(W))
\simeq
\map_{\SH_{\ul{\tau}}(\ul{V},\cV)}
(M(\ul{V})(d),M(\ul{V\times_{Y}W})).
\]
\end{lem}
\begin{proof}
Consider the functors $\eta\colon \Sm/\ul{V}\to \Sch/\ul{V}$ and $\rho\colon \lSm/Y\to \Sch/\ul{Y}$.
We have $\eta^*h(\ul{V\times_Y W})\simeq h(\ul{V\times_Y W})$ and $\rho^*h(W)\simeq h(W)$.
By the compactness result in Remark \ref{local.11}, we can use the commutativity of \eqref{divlocal.9.1} to obtain isomorphisms
\[
\eta^*M(\ul{V\times_Y W})\simeq M(\ul{V\times_Y W}),
\text{ }
\rho^*M(W)\simeq M(W).
\]
By adjunction, we have isomorphisms
\begin{gather*}
\map_{\sC_{\tau}(Y)}
(M(V)(d),M(W))
\simeq
\map_{\SH_{\ul{\tau}}(\Sch/\ul{Y},\cV)}
(M(\ul{V})(d), M(W)),
\\
\map_{\SH_{\ul{\tau}}(\ul{V},\cV)}(M(\ul{V})(d),M(\ul{V\times_Y W}))
\simeq
\map_{\SH_{\ul{\tau}}(\Sch/\ul{V},\cV)}(M(\ul{V})(d),M(\ul{V\times_Y W})).
\end{gather*}
We also have an isomorphism
\[
\map_{\SH_{\ul{\tau}}(\Sch/\ul{Y},\cV)}(M(\ul{V})(d), M(W))
\simeq
\map_{\SH_{\ul{\tau}}(\Sch/\ul{V},\cV)}
(M(\ul{V})(d),M(\ul{V\times_Y W})).
\]
Combine these to obtain the desired isomorphism.
\end{proof}

\begin{lem}
\label{local.3}
Assume one of the following conditions:
\begin{enumerate}
\item[\textup{(i)}]
$\tau=\sNis$,
\item[\textup{(ii)}]
$\tau=\setale,\ketale$ and $\cV=\DLambda$.
\end{enumerate}
For $X\in \lSm/{\A_{\N,B}}$, open subscheme $C$ of $\A_{\N,B}$, and integer $d$ such that $X$ is vertical over $\A_{\N,B}$,
there is a canonical isomorphism
\[
\map_{\sC_{\tau}(\A_{\N,B})}
(M(X)(d),M(C))
\simeq
\map_{\sC_{\dtau}(\A_{\N,B})}
(M(X)(d),M(C)).
\]
\end{lem}
\begin{proof}
The question is Zariski local on $X$.
Hence we may assume $X\in \lSmIFan/S$.
Let $X'\to X$ be a dividing cover in $\lSmIFan/S$.
Lemmas \ref{lemma.10} and \ref{local.2} give an isomorphism
\[
\map_{\sC_{\tau}(\A_{\N,B})}
(M(X)(d),M(C))
\simeq
\map_{\sC_{\tau}(\A_{\N,B})}
(M(X')(d),M(C)).
\]
Together with Lemma \ref{divlocal.13}, we obtain the desired isomorphism.
\end{proof}

For $X\in \lSm/\A_{\N,B}$ and open subscheme $C$ of $\A_{\N,B}$, recall from Construction \ref{lemma.11} that we have
\[
\Theta_d(X)
=
\map_{\SH_{\ul{\tau}}(\ul{X-\partial_{\A_{\N,B}}X},\cV)}
(M(\ul{X-\partial_{\A_{\N,B}}X})(d),M((X-\partial_{\A_{\N,B}}X)\times_{\A_{\N,B}}C)).
\]

\begin{lem}
\label{local.4}
Let $X\in \lSm/{\A_{\N,B}}$.
Assume one of the following conditions:
\begin{enumerate}
\item[\textup{(i)}]
$\tau=\sNis$ and $C=\G_{m,B}$,
\item[\textup{(ii)}]
$\tau=\setale,\ketale$, $\cV=\DLambda$, and $C$ is an open subscheme of $\G_{m,B}$.
\end{enumerate}
Then for every integer $d$,
there is a canonical isomorphism
\[
\Theta_d(X)
\simeq 
\map_{\dver^{-1}\sC_{\dtau}(\A_{\N,B})}
(M(X)(d),M(C)).
\]
\end{lem}
\begin{proof}
We have isomorphisms
\begin{align*}
& \map_{\dver^{-1}\sC_{\dtau}(\A_{\N,B})}
(M(X)(d),M(C))
\\
\simeq &
\limit_{X'\in X_{\divi}^{\op}} \map_{\sC_{\dtau}(\A_{\N,B})}
(M(X'-\partial_{\A_{N,B}}X')(d),M(C))
\\
\simeq &
\limit_{X'\in X_{\divi}^{\op}}
\Theta_d(X')
\simeq
\Theta_d(X),
\end{align*}
where we use \eqref{divlocal.2.4} for the first one, Lemmas \ref{local.2} and \ref{local.3} for the second one, and Proposition \ref{lemma.7} for the third one.
\end{proof}

\begin{thm}
\label{local.5}
Let $X\in \lSm/{\A_{\N,B}}$.
Assume one of the following conditions:
\begin{enumerate}
\item[\textup{(i)}]
$\tau=\sNis$ and $C=\G_{m,B}$,
\item[\textup{(ii)}]
$\tau=\setale,\ketale$, $\cV=\DLambda$, and $C$ is an open subscheme of $\G_{m,B}$.
\end{enumerate}
Then for every integer $d$,
there is a canonical isomorphism
\[
\Theta_d(X)
\simeq
\map_{\SH_{\tau}(X,\cV)}
(M(X)(d),M(X\times_{\A_{\N,B}}C)).
\]
\end{thm}
\begin{proof}
Lemma \ref{local.4} shows that $M(C)\in \dver^{-1}\sC_{\dtau}(\A_{\N,B})$ is $\ver$-local.
Hence there is a canonical isomorphism
\[
\Theta_d(X)
\simeq
\map_{\SH_{\tau}(\A_{\N,B},\cV)}
(M(X)(d),M(C)).
\]
By adjunction, we obtain the desired isomorphism.
\end{proof}

\begin{rmk}
\label{local.9}
For $X\in \lSm/\A_{\N,B}$ such that $X$ is vertical over $\A_{\N,B}$,
Theorem \ref{local.5} gives a canonical isomorphism
\begin{equation}
\label{local.5.1}
\begin{split}
&\map_{\SH_{\ul{\tau}}(\ul{X},\cV)}
(M(\ul{X})(d),M(X\times_{\A_{\N,B}}C))
\\
\xrightarrow{\simeq}&
\map_{\SH_{\tau}(X,\cV)}
(M(X)(d),M(X\times_{\A_{\N,B}}C))
\end{split}
\end{equation}
for every integer $d$.
This is isomorphic to the composition
\begin{align*}
&\map_{\SH_{\ul{\tau}}(\ul{X},\cV)}
(M(\ul{X})(d),M(X\times_{\A_{\N,B}}C))
\\
\xrightarrow{\simeq}&
\map_{\SH_{\ul{\tau}}(\Sch/\ul{X},\cV)}
(M(\ul{X})(d),M(X\times_{\A_{\N,B}}C))
\\
\xrightarrow{\simeq}&
\map_{\sC_{\tau}(X)}
(M(X)(d),M(X\times_{\A_{\N,B}}C))
\\
\xrightarrow{\simeq}&
\map_{\SH_{\tau}(X,\cV)}
(M(X)(d),M(X\times_{\A_{\N,B}}C)),
\end{align*}
where the inverse of the first isomorphism is induced by $\eta_\sharp$, the second (resp.\ third) isomorphism is induced by $\rho_\sharp$ (resp.\ $L_{\ver\cup \dtau}$).
The square
\[
\begin{tikzcd}
\Sm/\ul{X}\ar[r,"\lambda"]\ar[d,"\eta"']&
\lSm/\ul{X}\ar[d,"\mu"]
\\
\Sch/\ul{X}\ar[r,leftarrow,"\rho"]&
\lSm/X
\end{tikzcd}
\]
commutes, where $\mu$ sends $V\in \lSm/\ul{X}$ to $V\times_{\ul{X}} X$.
Recall that $\lambda$ sends $W\in \Sm/\ul{X}$ to $W$.
It follows that \eqref{local.5.1} is isomorphic to the induced morphism
\begin{align*}
&\map_{\SH_{\ul{\tau}}(\ul{X},\cV)}
(M(\ul{X})(d),M(X\times_{\A_{\N,B}}C))
\\
\to &
\map_{\SH_{\tau}(X,\cV)}
(f^*M(\ul{X})(d),f^*M(X\times_{\A_{\N,B}}C)),
\end{align*}
where $f\colon X\to \ul{X}$ is the morphism removing the log structure.
\end{rmk}

\begin{lem}
\label{local.13}
For $X\in \lSm/B$,
Zariski locally on $X$, there exists a vertical log smooth morphism $X\to \A_{\N,B}$.
\end{lem}
\begin{proof}
By \cite[Lemma A.5.9]{logDM}, we may assume that there exists a strict smooth morphism $X\to \A_{P,B}$ for some sharp fs monoid $P$.
If $P=0$, then the claim is trivial.
Hence assume $P\neq 0$.

By \cite[Corollary I.2.3.17(3)]{Ogu},
we can choose an element $p\in P$ not contained in any proper face of $P$.
Since $P$ is saturated,
we have $\Q_{\geq 0}p\cap P\simeq \N$.
Replace $p$ by the least nonzero element of this monoid to assume that $p$ is not a multiple of any other element of $P$.
Let $\theta\colon \N\to P$ be the map sending $n\in \N$ to $np$.
Observe that $\theta$ is vertical.
If the cokernel of $\theta^\gp$ has torsion, then $rq=sp$ for some $q\in P^\gp-\theta(P)$ and integers $r,s>0$.
Since $P$ is saturated, we have $q\in P$.
There exist integers $m,n\geq 0$ with $mr+ns=\gcd(r,s)$.
Then we have $p=r(mp+nq)/\gcd(r,s)$, which contradicts to the assumption on $p$.
Hence the cokernel of $\theta^\gp$ is torsion free.
It follows that $\A_\theta$ is log smooth by \cite[Theorem IV.3.1.8]{Ogu}, and 
then we obtain a vertical log smooth morphism $X\to \A_{\N,B}$.
\end{proof}

In the following, we will use the natural transformation \eqref{consq.1.1}.

\begin{thm}
\label{local.6}
Assume one of the following conditions:
\begin{enumerate}
\item[\textup{(i)}]
$\tau=\sNis$,
\item[\textup{(ii)}]
$\tau=\setale,\ketale$ and $\cV=\DLambda$.
\end{enumerate}
Let $f\colon X\to \ul{X}$ be the morphism removing the log structure with $X\in \lSm/B$, let $j'\colon X-\partial X\to X$ be the obvious open immersion, and let $p\colon X-\partial X\to B$ be the structure morphism.
We set $j:=fj'$.
Then the natural transformation
\[
j_\sharp p^* \xrightarrow{Ex} f_*j_\sharp' p^*
\]
of functors $\SH_{\tau}(B,\cV)\to \SH_{\tau}(X,\cV)$ is an isomorphism.
\end{thm}
\begin{proof}
The question is Zariski local on $B$ and $X$.
By Lemma \ref{local.13}, we may assume that there is a vertical log smooth morphism $X\to \A_{\N,B}$.

Let $v\colon V\to B$ be any proper morphism in $\Sch$.
Consider the induced commutative diagram with cartesian squares
\[
\begin{tikzcd}
(X-\partial X)\times_B V\ar[d,"v'''"']\ar[r,"u'"]&
X\times_B V\ar[d,"v''"]\ar[r,"g"]&
\ul{X}\times_B V\ar[d,"v'"]\ar[r]&
V\ar[d,"v"]&
\\
X-\partial X\ar[r,"j'"]&
X\ar[r,"f"]&
\ul{X}\ar[r]&
B.
\end{tikzcd}
\]
We set $u:=gu'$.
By \cite[Lemma 2.2.23]{Ayo071} or \cite[Proposition 4.2.13]{CD19}, it suffices to show that the morphism
\[
j_\sharp  p^* v_*\unit \xrightarrow{Ex} f_*j_\sharp' p^* v_* \unit
\]
is an isomorphism.
Let $q\colon (X-\partial X)\times_B V\to V$ be the projection.
Consider the commutative diagram
\[
\begin{tikzcd}
j_\sharp p^* v_*\unit\ar[d,"Ex"']\ar[r,"Ex"]&
j_\sharp v_*''' q^*\unit\ar[d,"Ex"]\ar[rr,"Ex"]&
&
v_*'u_\sharp q^*\unit\ar[d,"Ex"]
\\
f_*j_\sharp'p^* v_*\unit\ar[r,"Ex"]&
f_*j_\sharp'v_*''' q^*\unit\ar[r,"Ex"]&
f_*v_*''u_\sharp'q^*\unit\ar[r,"\simeq"]&
v_*'g_* u_\sharp'q^*\unit.
\end{tikzcd}
\]
The left upper and left lower horizontal morphisms are isomorphisms by Theorem \ref{consq.1}(3).
The right upper and middle lower horizontal morphisms are isomorphisms by Theorem \ref{consq.1}(1).
Hence it suffices to show that the right vertical morphism is an isomorphism.
Replace the pair $(B,X)$ by $(V,X\times_B V)$ to reduce to showing that the morphism
\[
j_\sharp \unit \xrightarrow{Ex} f_*j_\sharp' \unit
\]
is an isomorphism.

It suffices to show that the induced morphism
\[
\map_{\SH_{\ul{\tau}}(\ul{X},\cV)}
(M(\ul{Y})(d),M(X-\partial X))
\to
\map_{\SH_{\tau}(X,\cV)}
(M(Y)(d),M(X-\partial X))
\]
is an isomorphism for all $\ul{Y}\in \Sm/\ul{X}$ and integer $d$, where $Y:=\ul{Y}\times_{\ul{X}}X$.
Equivalently, it suffices to show that the induced morphism
\[
\map_{\SH_{\ul{\tau}}(\ul{Y},\cV)}
(M(\ul{Y})(d),M(Y-\partial Y))
\to
\map_{\SH_{\tau}(Y,\cV)}
(M(Y)(d),M(Y-\partial Y))
\]
is an isomorphism.
We conclude together with Theorem \ref{local.5} and Remark \ref{local.9}.
\end{proof}

\begin{thm}
\label{local.8}
Assume one of the following conditions:
\begin{enumerate}
\item[\textup{(i)}]
$\tau=\sNis$,
\item[\textup{(ii)}]
$\tau=\setale,\ketale$ and $\cV=\DLambda$.
\end{enumerate}
For $X\in \lSm/B$,
consider the commutative diagram with cartesian squares
\[
\begin{tikzcd}
\partial X\ar[r,"i'"]\ar[d,"f'"']&
X\ar[d,"f"]\ar[r,leftarrow,"j'"]&
X-\partial X\ar[d,"\id"]
\\
\ul{\partial X}\ar[r,"i"]&
\ul{X}\ar[r,leftarrow,"j"]&
X-\partial X,
\end{tikzcd}
\]
where $i$ and $j$ are the obvious immersions, and $f$ is the morphism removing the log structure.
Let $p\colon \ul{X}\to B$ be the structure morphism.
Then the square
\[
\begin{tikzcd}
p^* \ar[r,"ad"]\ar[d,"ad"']&
f_*f^*p^*\ar[d,"ad"]
\\
i_*i^*p^*\ar[r,"ad"]&
g_*g^*p^*
\end{tikzcd}
\]
of functors $\SH_{\tau}(B,\cV)\to \SH_{\tau}(\ul{X},\cV)$ is cartesian in the sense that it becomes a cartesian square when applied to any object of $\SH_{\tau}(B,\cV)$, where $g:=if'=fi'$.
\end{thm}
\begin{proof}
Consider the commutative diagram
\[
\begin{tikzcd}
j_\sharp j^*p^*\ar[r,"Ex"]\ar[d,"ad'"']&
f_* j_\sharp' j^*p^*\ar[r,"\simeq"]&
f_*j_\sharp' j'^* f^*p^*\ar[d,"ad'"]
\\
p^*\ar[rr,"ad"]&
&
f_*f^*p^*.
\end{tikzcd}
\]
The upper left horizontal natural transformation is an isomorphism by Theorem \ref{local.6}.
The localization property yields cofiber sequences
\[
j_\sharp j^*p^*
\xrightarrow{ad'}
p^*
\xrightarrow{ad}
i_*i^*p^*,
\;
f_*j_\sharp'j'^*f^*p^*
\xrightarrow{ad'}
f_*f^*p^*
\xrightarrow{ad}
f_*i_*'i'^*f^*p^*.
\]
Combine what we have discussed above to conclude.
\end{proof}

\begin{thm}
\label{local.10}
Assume that $\tau=\setale,\ketale$ and $\cV=\DLambda$.
For $X\in \lSch$ with a chart $\N$,
consider the induced cartesian square
\[
\begin{tikzcd}
\partial X\ar[d,"f'"']\ar[r,"i'"]&
X\ar[d,"f"]
\\
\ul{\partial X}\ar[r,"i"]&
\ul{X}.
\end{tikzcd}
\]
Then the square
\[
\begin{tikzcd}
\id \ar[r,"ad"]\ar[d,"ad"']&
f_*f^*\ar[d,"ad"]
\\
i_*i^*\ar[r,"ad"]&
g_*g^*
\end{tikzcd}
\]
of functors $\DA_{\tau}(\ul{X},\Lambda)\to \DA_{\tau}(\ul{X},\Lambda)$ is cartesian, where $g:=if'=f'i$.
\end{thm}
\begin{proof}
By \cite[Lemma 2.2.23]{Ayo071} or \cite[Proposition 4.2.13]{CD19}, it suffices to show that the square
\[
\begin{tikzcd}
v_*\unit \ar[r,"ad"]\ar[d,"ad"']&
f_*f^*v_*\unit\ar[d,"ad"]
\\
i_*i^*v_*\unit\ar[r,"ad"]&
g_*g^*v_*\unit
\end{tikzcd}
\]
is cartesian for all proper morphism $v\colon V\to \ul{X}$ in $\Sch$.
Use Theorem \ref{consq.1}(3) and replace $X$ by $V\times_{\ul{X}}X$ to reduce to showing that the square
\begin{equation}
\label{local.10.1}
\begin{tikzcd}
\unit \ar[r,"ad"]\ar[d,"ad"']&
f_*f^*\unit\ar[d,"ad"]
\\
i_*i^*\unit\ar[r,"ad"]&
g_*g^*\unit
\end{tikzcd}
\end{equation}
is cartesian.
We set $B:=\ul{X}$.
Let $a\colon X\to \A_{\N,\ul{X}}$ be the strict closed immersion induced by the chart $X\to \A_{\N}$.
Since $a_*$ is fully faithful by the localization property, it suffices to show that \eqref{local.10.1} is cartesian after applying $a_*$.
Use Theorem \ref{consq.1}(3) and replace $X$ by $\A_{\N,\ul{X}}$ to reduce to showing that the square
\[
\begin{tikzcd}
M(U) \ar[r,"ad"]\ar[d,"ad"']&
f_*f^*M(U)\ar[d,"ad"]
\\
i_*i^*M(U)\ar[r,"ad"]&
g_*g^*M(U)
\end{tikzcd}
\]
is cartesian whenever $X=\A_{\N,B}$ for some $B\in \Sch$ and $U\to \A_{\N,B}$ is an open immersion.
By the localization property, it suffices to show that the induced morphism
\[
\map_{\SH_{\ul{\tau}}(\A_B^1,\cV)}(M(Y)(d),M(C))
\to
\map_{\SH_{\tau}(\A_{\N,B},\cV)}(M(Y\times_{\A^1}\A_{\N})(d),M(C))
\]
is an isomorphism for all $Y\in \Sm/\A_B^1$, where $C:=U\times_{\A_{\N}}\G_m$.
This is a consequence of Theorem \ref{local.5}.
\end{proof}

\begin{rmk}
\label{local.14}
Let $P$ be the submonoid of $\N^2$ generated by $(2,0)$, $(1,1)$, and $(0,2)$, and let $\theta \colon \N\to P$ the map sending $1$ to $(2,0)$.
Then $\A_P\to \A_{\N}$ is log smooth.
However, it is unclear whether the open immersion
\[
\A_{P}-\partial_{\A_{\N}}\A_P\to \A_P
\]
is inverted in $(\A^1)^{-1}\logSH(\A_{\N})$ or not.
Hence unlike Theorem \ref{local.5}, we do \emph{not} know how to express
\[
\map_{(\A^1)^{-1}\logSH(\A_{\N})}(M(\A_P),M(\G_m))
\]
as a hom spectrum in $\SH(\A^1)$.
In particular, we do \emph{not} know whether there is an equivalence of $\infty$-categories
\[
(\A^1)^{-1}\logSH(\A_{\N})
\simeq
\SH(\A_{\N})
\]
or not.
This is the reason why we invert both $\A^1$ and $\ver$ instead of just $\A^1$ for our definition of $\SH(S)$ when $S$ has a nontrivial log structure.
\end{rmk}

\begin{rmk}
Assume that $\tau=\setale,\ketale$.
Using $(\divi)^{-1}\infShv_\tau^\wedge(\lSm/-,\cV)$ instead of  $\infShv_{\dtau}(\lSm/-,\cV)$ as noted in Remark \ref{comp.9},
we can prove Theorems \ref{local.5} and \ref{local.6}--\ref{local.10} for $\SH_\tau^\wedge$ and $\DA_\tau^\wedge(-,\Lambda)$
if for every $X\in \lSch/B$,
$\ul{X}$ is $(\Lambda,\ul{\tau})$-admissible in the sense of \cite[Definition 2.4.14(1)]{AGV},
which is a certain bounded cohomological dimension assumption.
Here, $\Lambda$ is the sphere spectrum in the case of $\SH_\tau^\wedge$.
The condition is needed in the proof of Lemma \ref{local.2} to ensure that $\infShv_\tau^\wedge(\Sm/\ul{X},\cV)$ and $\infShv_\tau^\wedge(\lSch/\ul{X},\cV)$ are compactly generated,
which can be shown by arguing as in \cite[Proposition 2.4.20(2)]{AGV} (see also \cite[Remark 2.4.23]{AGV}).
For the case of $\tau=\ketale$,
we also need a result on the Kummer \'etale cohomological dimension \cite[Proposition 5.3]{Nak}.
\end{rmk}

\subsection{Cohomology theories}
\label{cohomology}

Throughout this subsection, $B$ is a base scheme with an object $\E\in \SH(B)$.
For every morphism $p\colon X\to B$ in $\lSch$ and integers $p$ and $q$, the \emph{$\E$-cohomology of $X$} is defined to be
\begin{equation}
\E^{p,q}(X)
:=
\pi_0\Hom_{\SH(X)}(\Sigma_{\P^1}^\infty X_+,\Sigma^{p,q} p^*\E).
\end{equation}

\begin{exm}
\label{coh.1}
Consider the following three fundamental examples of $\E$:
\begin{enumerate}
\item[(i)] motivic Eilenberg-MacLane spectrum $\ML$,
\item[(ii)] homotopy $K$-theory spectrum $\KGL$,
\item[(iii)] algebraic cobordism spectrum $\MGL$.
\end{enumerate}
Voevodsky introduced these in \cite[\S 6]{zbMATH01194164}, but the definition of $\ML$ that we will use below is Spitzweck's version \cite[Definition 4.27]{zbMATH07015021}.
He defined $\rM\Z$ in $\SH(\Spec(\Z))$.
If $q\colon B\to \Spec(\Z)$ denotes the structure morphism, his definition of $\rM\Z$ in $\SH(B)$ is $q^*\rM\Z$.
The two different definitions of $\rM\Z$ are isomorphic if $B$ is smooth over a field by \cite[Theorem 8.22]{zbMATH07015021}.
Spitzweck's construction can be generalized to any ring $\Lambda$, see the statement of \cite[Corollary 10.4]{zbMATH07015021}.

If $p\colon X\to B$ is a morphism in $\Sch$, then $p^*\KGL\simeq \KGL$ and $p^*\MGL\simeq \MGL$ because $\KGL$ and $\MGL$ admit geometric models.
By definition, $p^* \ML \simeq \ML$.
\end{exm}

\begin{df}
\label{coh.2}
For $X\in \lSch/B$ and integers $p$ and $q$,
the \emph{motivic cohomology of $X$} is defined to be
\[
H_{\cM}^p(X,\Lambda(q))
:=
\ML^{p,q}(X).
\]
The \emph{homotopy $K$-theory spectrum of $X$} is defined to be
\[
\mathrm{KH}(X)
:=
\map_{\SH(X)}(\Sigma_{\P^1}^\infty X_+,\KGL).
\]
This is isomorphic to the original definition when $X$ has the trivial log structure due to \cite{zbMATH06156613}.
The \emph{algebraic cobordism of $X$} is defined to be $\MGL^{p,q}(X)$.
\end{df}

\begin{cor}
\label{coh.3}
For $X\in \lSm/B$ and integer $q$, there is a canonical long exact sequence
\[
\cdots
\to
\E^{p,q}(\ul{X})
\to
\E^{p,q}(X-\partial X)
\oplus
\E^{p,q}(\ul{\partial X})
\to
\E^{p,q}(\partial X)
\to
\E^{p+1,q}(\ul{X})
\to
\cdots
\]
\end{cor}
\begin{proof}
By Theorem \ref{local.8},
the induced square
\[
\begin{tikzcd}
\hom_{\SH(\ul{X})}(\unit,\Sigma^{0,q}\E)\ar[d]\ar[r]&
\hom_{\SH(X)}(\unit,\Sigma^{0,q}\E)\ar[d]
\\
\hom_{\SH(\ul{\partial X})}(\unit,\Sigma^{0,q}\E)\ar[r]&
\hom_{\SH(\partial X)}(\unit,\Sigma^{0,q}\E)
\end{tikzcd}
\]
is cartesian.
We also have
\begin{align*}
&
\hom_{\SH(X)}
(\unit,\Sigma^{0,q}\E)
\simeq
\hom_{\SH(B)}
(\Sigma_{\P^1}^\infty X_+,\Sigma^{0,q}\E)
\\
\simeq &
\hom_{\SH(B)}
(\Sigma_{\P^1}^\infty (X-\partial X)_+,\Sigma^{0,q}\E)
\simeq
\hom_{\SH(X-\partial X)}
(\unit,\Sigma^{0,q}\E).
\end{align*}
Combine what we have discussed above to conclude.
\end{proof}

\begin{exm}
\label{coh.4}
Assume $X=\A_{\N,B}$.
Then Corollary \ref{coh.3} gives an isomorphism
\[
\E^{p,q}(\pt_{\N,B})
\simeq
\E^{p,q}(\G_{m,B})
\]
for all integers $p$ and $q$.
\end{exm}

\begin{rmk}
Let $k$ be a field,
and let $X\to \A_{\N^r,k}$ be a strict smooth morphism in $\lSch$ for some integer $r\geq 0$.
In general, our motivic cohomology $H_{\cM}^p(\partial X,\Z(q))$ is not isomorphic to the log-motivic cohomology $H_{\log-\cM}^p(\partial X,\Z(q))$ of Gregory-Langer in \cite{2108.02845}.
For example, if $X=\A_{\N,k}$, then
\[
H_{\cM}^1(\partial X,\Z(1))\simeq H_{\cM}^1(\G_{m,k},\Z(1))\simeq k^*\oplus \Z
\simeq
H^0(\partial X,\cM_{\partial X}^\gp)
\]
by Example \ref{coh.4}, while 
\[
H_{\log-\cM}^1(\partial X,\Z(1))
\simeq
H^0(\partial X,\cM_{\partial X}^\gp/\Z_{\partial X})
\simeq
k^*.
\]
Nevertheless, we expect that these two cohomology theories are closely related.
In the above case, the factor $\Z$ in $k^*\oplus \Z\simeq H_\cM^1(\partial X,\Z(1))$ is the only difference.
\end{rmk}

\bibliography{bib_EJM}

\begin{thebibliography}{10}

\bibitem{SGA4}
{\sc M.~Artin, A.~Grothendieck, and J.~L. Verdier}, {\em Th\'eorie des topos et cohomologie \'etale des sch\'emas}, vol.~269, 270, 305 of Lecture Notes in Mathematics, Springer-Verlag, 1972--1973.
\newblock S\'eminaire de G\'eom\'etrie Alg\'ebrique du Bois--Marie 1963--64.

\bibitem{Ayo071}
{\sc J.~Ayoub}, {\em Les six op\'{e}rations de {G}rothendieck et le formalisme des cycles \'{e}vanescents dans le monde motivique {(I)}}, Ast\'{e}risque, 314 (2007).

\bibitem{Ayo072}
\leavevmode\vrule height 2pt depth -1.6pt width 23pt, {\em Les six op\'{e}rations de {G}rothendieck et le formalisme des cycles \'{e}vanescents dans le monde motivique {(II)}}, Ast\'{e}risque, 315 (2007).

\bibitem{MR3381140}
\leavevmode\vrule height 2pt depth -1.6pt width 23pt, {\em Motifs des vari\'{e}t\'{e}s analytiques rigides}, M\'{e}m. Soc. Math. Fr. (N.S.),  (2015), pp.~vi+386.

\bibitem{AGV}
{\sc J.~Ayoub, M.~Gallauer, and A.~Vezzani}, {\em The six-functor formalism for rigid analytic motives}, Forum Math. Sigma, 10 (2022), pp.~Paper No. e61, 182.

\bibitem{2207.00369}
{\sc F.~Binda, K.~Hiroki, and A.~Vezzani}, {\em On the p-adic weight-monodromy conjecture for complete intersetions in toric varieties}, 2022.
\newblock \url{https://arxiv.org/pdf/2207.00369v1.pdf}.

\bibitem{BLPO}
{\sc F.~Binda, T.~Lundemo, D.~Park, and P.~A. {\O}stv{\ae}r}, {\em A {H}ochschild-{K}ostant-{R}osenberg theorem and residue sequences for logarithmic {H}ochschild homology}, Advances in Math., 435 (2023), p.~109354.

\bibitem{logDMCras}
{\sc F.~Binda, D.~Park, and P.~A. {\O}stv{\ae}r}, {\em Motives and homotopy theory in logarithmic geometry}, C. R., Math., Acad. Sci. Paris, 360 (2022), pp.~717--727.

\bibitem{logDM}
\leavevmode\vrule height 2pt depth -1.6pt width 23pt, {\em Triangulated categories of logarithmic motives over a field}, {Ast\'erisque}, 433 (2022).

\bibitem{logSH}
\leavevmode\vrule height 2pt depth -1.6pt width 23pt, {\em Logarithmic motivic homotopy theory}, 2024.
\newblock \url{https://arxiv.org/pdf/2303.02729v2.pdf}.

\bibitem{MR0133812}
{\sc M.~Brown}, {\em Locally flat imbeddings of topological manifolds}, Ann. of Math. (2), 75 (1962), pp.~331--341.

\bibitem{zbMATH06156613}
{\sc D.-C. {Cisinski}}, {\em {Descente par \'eclatements en \(K\)-th\'eorie invariante par homotopie}}, {Ann. Math. (2)}, 177 (2013), pp.~425--448.

\bibitem{CD19}
{\sc D.-C. {Cisinski} and F.~{D\'eglise}}, {\em Triangulated categories of mixed motives}, Cham: Springer, 2019.

\bibitem{CLStoric}
{\sc D.~Cox, J.~Little, and H.~Schenck}, {\em Toric Varieties}, Graduate studies in mathematics, American Mathematical Soc., 2011.

\bibitem{EGA}
{\sc J.~Dieudonn{\'e} and A.~Grothendieck}, {\em \'{E}l\'ements de g\'eom\'etrie alg\'ebrique}, Inst. Hautes \'Etudes Sci. Publ. Math., 4, 8, 11, 17, 20, 24, 28, 32 (1961--1967).

\bibitem{2206.01564}
{\sc A.~Dubouloz, F.~D\'eglise, and P.~A. {\O}stv{\ae}r}, {\em Punctured tubular neighborhoods and stable homotopy at infinity}.
\newblock \url{http://arxiv.org/pdf/2206.01564v1.pdf}.

\bibitem{2108.02845}
{\sc O.~Gregory and A.~Langer}, {\em A log-motivic cohomology for semistable varieties and its $p$-adic deformation theory}, 2023.
\newblock \url{https://arxiv.org/pdf/2108.02845v4}.

\bibitem{MR1922832}
{\sc L.~Illusie}, {\em An overview of the work of {K}. {F}ujiwara, {K}. {K}ato, and {C}. {N}akayama on logarithmic \'{e}tale cohomology}, Ast\'{e}risque,  (2002), pp.~271--322.
\newblock Cohomologies $p$-adiques et applications arithm\'{e}tiques, II.

\bibitem{MR2302525}
{\sc M.~Levine}, {\em Motivic tubular neighborhoods}, Doc. Math., 12 (2007), pp.~71--146.

\bibitem{LZ}
{\sc Y.~Liu and W.~Zheng}, {\em Gluing restricted nerves of $\infty$-categories}.
\newblock \url{https://arxiv.org/pdf/1211.5294.pdf}, 2024.

\bibitem{HTT}
{\sc J.~{Lurie}}, {\em Higher topos theory}, vol.~170, Princeton, NJ: Princeton University Press, 2009.

\bibitem{HA}
\leavevmode\vrule height 2pt depth -1.6pt width 23pt, {\em {Higher algebra}}.
\newblock \url{https://www.math.ias.edu/~lurie/papers/HA.pdf}, 2017.

\bibitem{Mann}
{\sc L.~Mann}, {\em A $p$-adic 6-functor formalism in rigid-analytic geometry}.
\newblock \url{https://arxiv.org/pdf/2206.02022.pdf}, 2022.

\bibitem{zbMATH05052895}
{\sc F.~{Morel}}, {\em {The stable \(\mathbb{A}^1\)-connectivity theorems}}, {\(K\)-Theory}, 35 (2005), pp.~1--68.

\bibitem{MV}
{\sc F.~Morel and V.~Voevodsky}, {\em {${\bf A}^1$}-homotopy theory of schemes}, Inst. Hautes \'{E}tudes Sci. Publ. Math.,  (1999), pp.~45--143 (2001).

\bibitem{Nak}
{\sc C.~Nakayama}, {\em Logarithmic \'{e}tale cohomology}, Math. Ann., 308 (1997), pp.~365--404.

\bibitem{Nak2}
\leavevmode\vrule height 2pt depth -1.6pt width 23pt, {\em Logarithmic \'etale cohomology, {II}}, Adv. Math., 314 (2017), pp.~663--725.

\bibitem{zbMATH05809283}
{\sc C.~{Nakayama} and A.~{Ogus}}, {\em {Relative rounding in toric and logarithmic geometry}}, {Geom. Topol.}, 14 (2010), pp.~2189--2241.

\bibitem{Niz}
{\sc W.~Niziol}, {\em Toric singularities: Log-blow-ups and global resolutions}, J. Algebraic Geom., 15 (2006), pp.~1--29.

\bibitem{Ogu}
{\sc A.~Ogus}, {\em Lectures on Logarithmic Algebraic Geometry}, Cambridge Studies in Advanced Mathematics, Cambridge University Press, 2018.

\bibitem{ParThesis}
{\sc D.~Park}, {\em Triangulated categories of motives over fs log schemes}, PhD thesis, University of California, Berkeely, 2016.

\bibitem{1707.09435}
\leavevmode\vrule height 2pt depth -1.6pt width 23pt, {\em Triangulated categories of motives over fs log schemes}, 2017.
\newblock \url{http://arxiv.org/pdf/1707.09435v1.pdf}.

\bibitem{logGysin}
\leavevmode\vrule height 2pt depth -1.6pt width 23pt, {\em Log motivic {G}ysin isomorphisms}, 2023.
\newblock \url{http://arxiv.org/pdf/2303.12498v1.pdf}.

\bibitem{divspc}
{\sc D.~Park}, {\em Inverting log blowups in log geometry}, Tunis. J. Math., 6 (2024), pp.~405--453.

\bibitem{logshriek}
{\sc D.~Park}, {\em Log motivic exceptional direct image functors}, 2024.
\newblock \url{https://arxiv.org/pdf/2403.06692v1.pdf}.

\bibitem{lognearby}
\leavevmode\vrule height 2pt depth -1.6pt width 23pt, {\em Log motivic nearby cycles}, 2024.
\newblock \url{https://arxiv.org/pdf/2405.14083v1.pdf}.

\bibitem{logsix}
\leavevmode\vrule height 2pt depth -1.6pt width 23pt, {\em Motivic six-functor formalism for log schemes}, 2024.
\newblock \url{https://arxiv.org/pdf/2403.07645v1.pdf}.

\bibitem{Robalothesis}
{\sc M.~{Robalo}}, {\em Th\'eorie homotopique motivique des espaces non-commutatifs}.
\newblock PhD Thesis, University of Montpellier, 2014.

\bibitem{zbMATH06374152}
\leavevmode\vrule height 2pt depth -1.6pt width 23pt, {\em {\(K\)-theory and the bridge from motives to noncommutative motives}}, {Adv. Math.}, 269 (2015), pp.~399--550.

\bibitem{zbMATH02242854}
{\sc J.~{Rosick\'y}}, {\em {Generalized Brown representability in homotopy categories}}, {Theory Appl. Categ.}, 14 (2005), pp.~451--479.

\bibitem{zbMATH07015021}
{\sc M.~Spitzweck}, {\em A commutative {{\({\mathbb P}^1\)}}-spectrum representing motivic cohomology over {Dedekind} domains}, M{\'e}m. Soc. Math. Fr., Nouv. S{\'e}r., 157 (2018), pp.~1--110.

\bibitem{zbMATH07027475}
{\sc T.~Tsuji}, {\em Saturated morphisms of logarithmic schemes}, Tunis. J. Math., 1 (2019), pp.~185--220.

\bibitem{zbMATH01194164}
{\sc V.~Voevodsky}, {\em {\(\mathbb{A}^1\)-homotopy theory}}, {Doc. Math.}, Extra Vol. (1998), pp.~579--604.

\bibitem{zbMATH05042488}
{\sc J.~{Wildeshaus}}, {\em {The boundary motive: definition and basic properties}}, {Compos. Math.}, 142 (2006), pp.~631--656.

\end{thebibliography}
\bibliographystyle{siam}

\end{document}